\documentclass[draft]{amsart}

\usepackage[pdfborder={0 0 0}, draft=false, pagebackref, unicode=true]{hyperref}
\usepackage{esint,amssymb,mathtools,amsthm}

\usepackage{cite}
 
\def\multichoose#1#2{\ensuremath{\left(\kern-.3em\left(\genfrac{}{}{0pt}{}{#1}{#2}\right)\kern-.3em\right)}}

\newtheorem{thm}[equation]{Theorem}
\newtheorem{lem}[equation]{Lemma}

\newtheorem{prp}[equation]{Proposition}

\theoremstyle{definition}
\newtheorem{defn}[equation]{Definition}
\newtheorem{rmk}[equation]{Remark}

\numberwithin{equation}{section}

\usepackage{cleveref}
\crefname{equation}{}{}
\crefname{lem}{}{}
\crefrangelabelformat{equation}{(#3#1#4)--(#5#2#6)}

\newcommand\abs[2][empty]{\csname#1\endcsname \lvert{#2}\csname#1\endcsname\rvert}
\newcommand\doublebar[2][empty]{\csname#1\endcsname \lVert{#2}\csname#1\endcsname\rVert}

\newcommand\mat[1]{\boldsymbol{#1}}
\newcommand\arr[1]{\boldsymbol{\dot{#1}}}
\newcommand\dist{\mathop{\mathrm{dist}}\nolimits}
\newcommand\Div{\mathop{\mathrm{div}}\nolimits}
\newcommand\Tr{\mathop{\smash{\boldsymbol{\rlap{$\arr{\phantom{T}}$}\mathrm{Tr}}}\vphantom{T}}\nolimits}

\newcommand\Trace{\mathop{\mathrm{Tr}}\nolimits}
\newcommand\M{\mathop{\smash{\arr{\mathrm{M}}}\vphantom{M}}\nolimits}

\newcommand\vecM{\mathop{\smash{\vec{\mathrm{M}}}\vphantom{M}}\nolimits}
\newcommand\scalarM{\mathop{\mathrm{M}}\nolimits}

\newcommand\esssup{\mathop{\mathrm{ess\,sup}}}
\newcommand\sgn{\mathop{\mathrm{sgn}}}
\newcommand\re{\mathop{\mathrm{Re}}\nolimits}

\newcommand\R{\mathbb{R}}
\renewcommand\C{\mathbb{C}}

\newcommand\N{\mathbb{N}}
\newcommand\1{\mathbf{1}}
\newcommand\D{\mathcal{D}}
\newcommand\s{\mathcal{S}}

\newcommand\XX{\mathfrak{X}}
\newcommand\DD{\mathfrak{D}}
\newcommand\NN{\mathfrak{N}}

\newcommand\pureH{\parallel}

\newcommand\dmn{{n+1}}
\newcommand\pdmn{{(n+1)}}
\newcommand\dmnMinusOne{n}

\begin{document}

\title{The $\dot W^{-1,p}$ Neumann problem for higher order elliptic equations}

\author{Ariel Barton}
\address{Ariel Barton, Department of Mathematical Sciences,
			309 SCEN,
			University of Ar\-kan\-sas,
			Fayetteville, AR 72701}
\email{aeb019@uark.edu}
%\thanks{}

\begin{abstract}
We solve the Neumann problem in the half space~$\R^{n+1}_+$, for higher order elliptic differential equations with variable self-adjoint $t$-independent coefficients, and with boundary data in the negative smoothness space $\dot W^{-1,p}$, where $\max(0,\frac{1}{2}-\frac{1}{n}-\varepsilon) <\frac{1}{p} <\frac{1}{2}$. Our arguments are inspired by an argument of Shen and build on known well posedness results in the case $p=2$.

We use the same techniques to establish nontangential and square function estimates on layer potentials with inputs in $L^p$ or $\dot W^{\pm1,p}$ for a similar range of~$p$, based on known bounds for $p$ near~$2$; in this case we may relax the requirement of self-adjointess.
\end{abstract}



\maketitle

\tableofcontents

\section{Introduction}

In this paper we study the Neumann boundary value problem and layer potentials for higher order elliptic differential operators of the form
\begin{equation}\label{eqn:divergence}
Lu = (-1)^m \sum_{\abs{\alpha}=\abs{\beta}=m} \partial^\alpha (A_{\alpha\beta} \partial^\beta u),\end{equation}
where $m$ is a positive integer, and with coefficients $\mat A$ that are $t$-independent in the sense that
\begin{equation}\label{eqn:t-independent}\mat A(x,t)=\mat A(x,s)=\mat A(x) \quad\text{for all $x\in\R^n$ and all $s$, $t\in\R$}.\end{equation}
Our coefficients may be merely bounded measurable in the $n$ horizontal variables.
We remark that $t$-independent coefficients have been studied extensively in the second order case (the case $2m=2$). See, for example, \cite{JerK81A,KenP93,KenKPT00,KenR09,Rul07, AusAH08,AusAM10A, AlfAAHK11, Bar13, HofKMP15A, HofKMP15B, HofMitMor15, BarM16A, MaeM16, AusS16, MaeM17, AmeA18, AusM19}. $t$-independent coefficients in the higher order case have received much more limited study; Hofmann and Mayboroda together with the author of the present paper have begun their study in \cite{BarHM17,BarHM17pA,BarHM17pB,BarHM18,BarHM18p}.

Specifically, in \cite{BarHM18}, we established the following result. Suppose that $L$ is an operator of the form~\eqref{eqn:divergence} associated to coefficients $\mat A$ that are $t$-independent, bounded, self-adjoint in the sense that $A_{\alpha\beta}=\overline{A_{\beta\alpha}}$ whenever $\abs\alpha=\abs\beta=m$, and satisfy the ellipticity condition
\begin{equation}
\label{eqn:elliptic:slices}
\re\int_{\R^n}
\sum_{\abs\alpha=\abs\beta=m}\overline{\partial^\alpha\varphi(x,t)}\, A_{\alpha\beta}(x)\,\partial^\beta\varphi(x,t)\rangle\,dx \geq \lambda\doublebar{\nabla^m \varphi(\,\cdot\,,t)}_{L^2(\R^n)}^2
\end{equation}
for all $t\in \R$, all $\varphi\in C^\infty_0(\R^\dmn)$, and some $\lambda>0$ independent of $t$ and~$\varphi$. Then for every $\arr g\in L^2(\R^n)$ there is a solution~$w$, unique up to adding polynomials of degree~$m-1$, to the $L^2$ Neumann problem
\begin{equation}
\label{eqn:neumann:regular:2}
\left\{\begin{gathered}\begin{aligned}
Lw&=0 \text{ in }\R^\dmn_+
,\\
\M_{\mat A}^+ w &\owns \arr g,
\end{aligned}\\
\doublebar{\mathcal{A}_2^+(t\nabla^m \partial_t w)}_{L^2(\R^n)} + \sup_{t>0}\doublebar{\nabla^{m}w(\,\cdot\,,t)}_{L^2(\R^n)}
\leq C\doublebar{\arr g}_{L^2(\R^\dmnMinusOne)}
.\end{gathered}\right.\end{equation}

Here
$\mathcal{A}_2^+$ is the Lusin area integral given (in $\R^\dmn_+$) by
\begin{equation}\label{dfn:lusin:+}\mathcal{A}_2^+ H(x) = \biggl(\int_0^\infty \int_{\abs{x-y}<t} \abs{H(y,t)}^2 \frac{dy\,dt}{t^\dmn}\biggr) \quad\text{for all $x\in\R^n$}.\end{equation}
We adopt the convention that if a $t$ appears inside the argument of a tent space operator such as $\mathcal{A}_2^+$, then it denotes the $\pdmn$th coordinate function.

$\M_{\mat A}^+ w$ denotes the Neumann boundary values of $w$, and is the equivalence class of functions given by
\begin{equation}
\label{eqn:Neumann:intro}
\arr g\in \M_{\mat A}^+ w \text{ if }
\sum_{\abs\gamma=m-1}\int_{\R^n} \partial^\gamma\varphi(x,0)\,g_\gamma\,dx
= \sum_{\abs\alpha=\abs\beta=m} \int_{\R^\dmn_+} \partial^\alpha \varphi\,A_{\alpha\beta}\,\partial^\beta w\end{equation}
for all smooth test functions~$\varphi$ that are compactly supported in~$\R^\dmn$. An integration by parts argument shows that the right hand side depends only on the behavior of $\varphi$ near the boundary, and so $\M_{\mat A}^+ w$ is well defined as an operator on the space $\{\nabla^{m-1}\varphi\big\vert_{\partial\R^\dmn_+}:\varphi\in C^\infty_0(\R^\dmn)\}$. In the second order case $2m=2$, $\scalarM_{\mat A}^+ w$ consists of a single function or distribution; however, if $m\geq 2$, then by equality of mixed partials $\M_{\mat A}^+ w$ contains many arrays of distributions, and so is indeed an equivalence class. This is the formulation of Neumann boundary data used in \cite{Bar17,BarHM17,BarHM17pB,BarHM18,BarHM18p}, and is closely related to the Neumann boundary values for the bilaplacian in \cite{CohG85,Ver05,She07B,MitM13B} and for general constant coefficient systems in \cite{MitM13A,Ver10}. We refer the reader to \cite{BarM16B,BarHM17} for further discussion of higher order Neumann boundary data.

Let $\dot W^{1,q}(\R^n)$ denote the homogeneous Sobolev space in $\R^n$ with $\doublebar{\varphi}_{\dot W^{1,q}(\R^n)}=\doublebar{\nabla_\pureH \varphi}_{L^q(\R^n)}$, where $\nabla_\pureH$ denotes the gradient in $\R^n$ (rather than $\R^\dmn$). If $1/p+1/p'=1$, let $\dot W^{-1,p}(\R^n)$ be the dual space $(\dot W^{1,p'}(\R^n))^*$ to $\dot W^{1,p'}(\R^n)$. A second result of \cite{BarHM18} is that for every $\arr g$ in the negative Sobolev space $\dot W^{-1,2}(\R^n)$, there is a solution~$v$, unique up to adding polynomials of degree~$m-2$, to the $\dot W^{-1,2}$ Neumann problem
\begin{equation}
\label{eqn:neumann:rough:2}
\left\{\begin{gathered}\begin{aligned}
Lv&=0 \text{ in }\R^\dmn_+
,\\
\M_{\mat A}^+ v &\owns \arr g,
\end{aligned}\\
\doublebar{\mathcal{A}_2^+(t\nabla^m v)}_{L^2(\R^n)} + \sup_{t>0}\doublebar{\nabla^{m-1}v(\,\cdot\,,t)}_{L^2(\R^n)}
\leq C\doublebar{\arr g}_{\dot W^{-1,2}(\R^\dmnMinusOne)}
.\end{gathered}\right.\end{equation}
In this case the definition of $\M_{\mat A}^+ v$ is more delicate, because $\nabla^m v$ need not be locally integrable up to the boundary of $\R^\dmn_+$. We refer the reader to \cite[Section~2.3.2]{BarHM17pB} for the precise formulation of the Neumann boundary values $\M_{\mat A}^+ v$ of a solution $v$ to $Lv=0$ with ${\mathcal{A}_2^+(t\nabla^m v)}\in{L^2(\R^n)}$.

The main result of the paper \cite{BarHM18p} was that the solutions $w$ and $v$ to the problems~\eqref{eqn:neumann:regular:2} and~\eqref{eqn:neumann:rough:2} also satisfy the estimates
\begin{align}
\label{eqn:neumann:N:2}
\doublebar{\widetilde{N}_+(\nabla^{m-1}v)}_{L^2(\R^n)}
&\leq C\doublebar{\arr g}_{\dot W^{-1,2}(\R^\dmnMinusOne)}
,&
\doublebar{\widetilde{N}_+(\nabla^{m}w)}_{L^2(\R^n)}
&\leq C\doublebar{\arr g}_{L^2(\R^\dmnMinusOne)}
,\end{align}
where $\widetilde N_+$ is the modified nontangential maximal operator introduced in \cite{KenP93} and given (in the half space) by
\begin{equation}
\label{dfn:NTM:modified:+}
\widetilde N_+ H(x) = \sup
\biggl\{\biggl(\fint_{B((y, s),{s}/2)} \abs{H(z,t)}^2\,dz\,dt\biggr)^{1/2}:
s>0,\>
\abs{x-y}< s
\biggr\}
.\end{equation}
We remark that if $\mat A$ is $t$-independent, then ${\widetilde{N}_+(\nabla^{m-1}v)}\in {L^2(\R^n)}$ is a stronger statement than $\sup_{t>0}\doublebar{\nabla^{m-1}v(\,\cdot\,,t)}_{L^2(\R^n)}<\infty$; see \cite[Lemma~3.20]{BarHM17pA}, reproduced as Lemma~\ref{lem:slices} below.

The first of the two main results of the present paper is the following theorem.

\begin{thm}\label{thm:neumann:p:rough}
Suppose that $L$ is an elliptic operator of the form~\eqref{eqn:divergence} (in the weak sense of formula~\eqref{eqn:weak} below) of order~$2m$ associated with coefficients $\mat A$ that are bounded, $t$-independent in the sense of formula~\eqref{eqn:t-independent}, self-adjoint (that is, $A_{\alpha\beta}(x)=\overline{A_{\beta\alpha}(x)}$ for all $\abs\alpha=\abs\beta=m$ and all $x\in\R^n$), and satisfy the ellipticity condition~\eqref{eqn:elliptic:slices}.

Then there is some $\varepsilon>0$ depending only on the dimension $\dmn$, the order $m$ of the operator~$L$, the number $\lambda$ in the bound~\eqref{eqn:elliptic:slices}, and $\doublebar{\mat A}_{L^\infty(\R^n)}$, such that, if $\dmn\geq 4$ and $2-\varepsilon<p<\frac{2n}{n-2}+\varepsilon$, or if $\dmn=2$ or $\dmn=3$ and $2-\varepsilon<p<\infty$,
then for every $\arr g\in \dot W^{-1,p}(\R^n)\cap \dot W^{-1,2}(\R^n)$ the solution $v$ to the problem~\eqref{eqn:neumann:rough:2} also satisfies
\begin{equation}
\label{eqn:neumann:rough:p}
\left\{\begin{gathered}\begin{aligned}
Lv&=0 \text{ in }\R^\dmn_+
,\\
\M_{\mat A}^+ v &\owns \arr g,
\end{aligned}\\
\doublebar{\mathcal{A}_2^+(t\nabla^m v)}_{L^p(\R^n)} + \doublebar{\widetilde{N}_+(\nabla^{m-1}v)}_{L^p(\R^n)}
\leq C_p\doublebar{\arr g}_{\dot W^{-1,p}(\R^\dmnMinusOne)}
\end{gathered}\right.\end{equation}
where $\M_{\mat A}^+ v$ is as defined in \cite{BarHM18}, and where $C_p$ depends only on $p$, $n$, $m$, $\lambda$, and~$\doublebar{\mat A}_{L^\infty(\R^n)}$.
\end{thm}

The technical requirement $\arr g\in \dot W^{-1,p}(\R^n)\cap \dot W^{-1,2}(\R^n)$, rather than merely $\arr g\in \dot W^{-1,p}(\R^n)$, is due to difficulties in defining $\M_{\mat A}^+ v$. Specifically, as mentioned above, $\nabla^m v$ need not be locally integrable up to the boundary and so we must use the definition of $\M_{\mat A}^+ v$ formulated in \cite{BarHM17pB} rather than formula~\eqref{eqn:Neumann:intro}. One of the main results of \cite{BarHM17pB} is that if $Lu=0$ in $\R^\dmn_+$ and $\mathcal{A}_2^+(t\nabla^m v)\in L^p(\R^n)$ for some $p$ with $1<p\leq 2$, then this formulation of
$\M_{\mat A}^+ v$ exists and lies in $\dot W^{-1,p}(\R^n)$. If $\mathcal{A}_2^+(t\nabla^m v)\in L^p(\R^n)$ for $p>2$, then $\M_{\mat A}^+ v$ is only guaranteed to exist under the additional technical assumption that $\nabla^m v\in L^2(\R^n\times(\varepsilon,\infty))$ for all $\varepsilon>0$; by requiring that $\arr g\in \dot W^{-1,2}(\R^n)$ and so $\mathcal{A}_2^+(t\nabla^m v)\in L^2(\R^n)$, we ensure that $v$ satisfies this condition and so $\M_{\mat A}^+ v$ exists.

In a forthcoming paper \cite{Bar19pC}, we will show that solutions to the problem~\eqref{eqn:neumann:rough:p} are unique. We will remove the technical requirement $\nabla^m v\in L^2(\R^n\times(\varepsilon,\infty))$ from the main result of \cite{BarHM17pB}, so that $\M_{\mat A}^+ v$ exists whenever $\mathcal{A}_2^+(t\nabla^m v)\in L^p(\R^n)$; a density argument will imply existence of solutions to the problem~\eqref{eqn:neumann:rough:p} for all $\arr g\in \dot W^{-1,p}(\R^n)$ rather than merely all $\arr g\in \dot W^{-1,p}(\R^n)\cap \dot W^{-1,2}(\R^n)$. Finally, we will establish well posedness of the Neumann problem with boundary data in a Lebesgue space $L^p(\R^n)$ rather than in a negative Sobolev space $\dot W^{-1,p}(\R^n)$ for a range of $p$ dual to that of Theorem~\ref{thm:neumann:p:rough}.

\subsection{Layer potentials}
\label{sec:intro:potentials}

We now discuss the method of proof of \cite{BarHM18,BarHM18p}. This is the classic method of layer potentials. The second main result of this paper (Theorem~\ref{thm:potentials} below) consists of some bounds on layer potentials that will be of use in the present paper and are of interest in their own right.

The double and single layer potentials for Laplace's equation are given by
\begin{align*}
\D^I_\Omega f(X)&=\int_{\partial\Omega} \nu(Y)\cdot \nabla_Y E^{-\Delta}(X,Y)\,f(Y)\,d\sigma(Y),
\\
\s^{-\Delta}_\Omega g(X)&=\int_{\partial\Omega} E^{-\Delta}(X,Y)\,g(Y)\,d\sigma(Y)
\end{align*}
where $\nu$ is the unit outward normal to~$\partial\Omega$ and $E^{-\Delta}$ is the fundamental solution for $-\Delta$ given by $E^{-\Delta}(X,Y)=-\frac{1}{2\pi}\log\abs{X-Y}$ (in $\R^2$) or $E^{-\Delta}(X,Y)=c_n\abs{X-Y}^{1-n}$ (in $\R^\dmn$, $n\geq 2$). We remark that if $-\Delta u=0$ in~$\Omega$, then the classical Neumann boundary values $\nu\cdot \nabla u$ of~$u$ coincide with the boundary values $\scalarM_I^+ u$ given by formula~\eqref{eqn:Neumann:intro}.
Layer potentials have a number of useful properties: for reasonably well behaved domains $\Omega$ and input functions $f$ and~$g$, we have that \begin{equation*}-\Delta\D^I_\Omega f=-\Delta\s^{-\Delta}_\Omega g=0\text{ in }\Omega\text{ and in }\R^\dmn\setminus\overline\Omega,\end{equation*}
and the jump and continuity relations
\begin{equation*}
\begin{aligned}
-\D^I_\Omega f\big\vert_{\partial\Omega} +\D^I_\Omega f\big\vert_{\partial W} &= f,
&
\nu_\Omega\cdot\nabla\D^I_\Omega f\big\vert_{\partial\Omega} +\nu_W\cdot\nabla\D^I_\Omega f\big\vert_{\partial W} &= 0,
\\
-\s^{-\Delta}_\Omega g\big\vert_{\partial\Omega} +\s^{-\Delta}_\Omega g\big\vert_{\partial W} &= 0,
&
\nu_\Omega\cdot\nabla\s^{-\Delta}_\Omega g\big\vert_{\partial\Omega} +\nu_W\cdot\nabla\s^{-\Delta}_\Omega g\big\vert_{\partial W} &= g,
\end{aligned}
\end{equation*}
where $W=\R^\dmn\setminus\overline\Omega$, $\nu_\Omega$ is the unit outward normal to~$\Omega$, and $\nu_W=-\nu_\Omega$ is the unit outward normal to~$W$.

Layer potentials may be generalized from $L=-\Delta$ to more general elliptic operators in such a way that analogues to the above useful properties are true. In the second order case, the generalization is straightforward once a fundamental solution has been constructed. Such potentials have been studied in many papers, including \cite{DahKV88,FabKV88,Fab88,Gao91,She07B,KenR09,Rul07,AlfAAHK11,MitM11,Bar13,HofMitMor15, HofKMP15B, HofMayMou15,GraH16,BarM16A}. (An equivalent formulation involving semigroups was given in \cite{Ros13} and used in \cite{AusM14,AusS16,AmeA18,AusM19}.) In the higher order case, layer potentials can be generalized either by using the fundamental solution and a careful integration by parts (see \cite{Agm57,CohG83,CohG85,Ver05,She07B,MitM13A,MitM13B}), or by using the Lax-Milgram theorem to directly construct functions that obey appropriate jump relations (see \cite{Bar17,BarHM17} or Sections~\ref{sec:dfn:D} and~\ref{sec:dfn:S} below).

The classic method of layer potentials constructs a solution to the Neumann problem
\begin{equation*}Lu=0 \text{ in }\Omega,\qquad \M_{\mat A}^\Omega u\owns\arr g\end{equation*}
by showing that either $\arr f\to \M_{\mat A}^\Omega \D^{\mat A}_\Omega \arr f$ or $\arr h\to \M_{\mat A}^\Omega \s^L_\Omega \arr h$ is invertible between appropriate function spaces and letting $u= \D^{\mat A}_\Omega (\M_{\mat A}^\Omega \D^{\mat A}_\Omega)^{-1}\arr g$ or $u=\s^L_\Omega(\M_{\mat A}^\Omega \s^L_\Omega)^{-1}\arr g$. Then $Lu=0$ in $\Omega$ because $L(\D^{\mat A}_\Omega\arr f)=0$ or $L(\s^L_\Omega\arr h)=0$ for all $\arr f$ or~$\arr h$, and $\M_{\mat A}^\Omega u\owns \arr g$ by definition. Estimates on solutions to the Neumann problem, such as the nontangential and area integral estimates in the problem~\eqref{eqn:neumann:rough:p}, may be derived from estimates on layer potentials.

This method was used in \cite{FabJR78,Ver84,DahK87,FabMM98,Zan00,May05} in the case of harmonic functions, in \cite{DahKV88,FabKV88,Fab88,Gao91,She07B} for second order constant coefficient systems, in \cite{AlfAAHK11,Bar13,HofMitMor15,HofKMP15B,BarM16A} for second order operators with variable $t$-inde\-pen\-dent coefficients, in \cite{Agm57,CohG83,CohG85,Ver05,She07B,MitM13A,MitM13B} for higher order operators with constant coefficients, and in \cite{BarHM18} for higher order operators with variable $t$-independent coefficients. In the case of higher order $t$-independent operators of interest in the present paper, extensive preliminaries were necessary.

The main results of \cite{BarHM17,BarHM17pA} were the $p\leq 2$ cases of the following estimates on layer potentials. (The $p>2$ cases were established later in \cite{BarHM18p}.)
\begin{align}
\label{eqn:S:lusin:2}
\doublebar{\mathcal{A}_2^*(t\nabla^m\partial_t\s^L\arr g)}_{L^p(\R^n)}
&\leq C_p\doublebar{\arr g}_{L^p(\R^n)}
,&2-\varepsilon&< p<2+\varepsilon
,\\
\label{eqn:D:lusin:2}
\doublebar{\mathcal{A}_2^*(t\nabla^m\partial_t\D^{\mat A}\arr \varphi)}_{L^p(\R^n)}
&\leq C_p\doublebar{\arr \varphi}_{\dot W\!A^{1,p}_{m-1}(\R^n)}
,&2&\leq p<2+\varepsilon
,\\
\label{eqn:S:lusin:rough:2}
\doublebar{\mathcal{A}_2^*(t\nabla^m\s^L_\nabla\arr h)}_{L^p(\R^n)}
&\leq C_p\doublebar{\arr h}_{L^p(\R^n)}
,&2-\varepsilon&< p<2+\varepsilon
,\\
\label{eqn:D:lusin:rough:2}
\doublebar{\mathcal{A}_2^*(t\nabla^m\D^{\mat A}\arr f)}_{L^p(\R^n)}
&\leq C_p\doublebar{\arr f}_{\dot W\!A^{0,p}_{m-1}(\R^n)}
,&2&\leq p<2+\varepsilon
.\end{align}
The operators $\D^{\mat A}$ and $\s^L$ are the generalizations mentioned above of $\D^{I}_\Omega$ and $\s^{-\Delta}_\Omega$ in the case $\Omega=\R^\dmn_+$. See Sections~\ref{sec:dfn:D} and~\ref{sec:dfn:S} below.
The modified single layer potential $\s^L_\nabla$ was introduced in \cite{BarHM17pA} based on an analogous operator in the second order case used in \cite{AlfAAHK11,HofMayMou15,HofMitMor15}; we remark that the estimate~\eqref{eqn:S:lusin:rough:2} is equivalent to the two estimates \eqref{eqn:S:lusin:2} and \begin{equation}
\label{eqn:S:lusin:rough:intro:AS16}
\doublebar{\mathcal{A}_2^*(t\nabla^m\s^L\arr g)}_{L^p(\R^n)}\leq C_p\doublebar{\arr g}_{\dot W^{-1,p}(\R^n)}.\end{equation}

These estimates are valid for operators $L$ of the form~\eqref{eqn:divergence} associated to coefficients~$\mat A$ that are bounded, $t$-independent in the sense of formula~\eqref{eqn:t-independent}, and satisfy the ellipticity condition
\begin{equation}
\label{eqn:elliptic}
\re\int_{\R^\dmn}
\sum_{\abs\alpha=\abs\beta=m}\overline{\partial^\alpha\varphi(x,t)}\, A_{\alpha\beta}(x)\,\partial^\beta\varphi(x,t)\rangle\,dx\,dt \geq \lambda\doublebar{\nabla^m \varphi}_{L^2(\R^\dmn)}^2
\end{equation}
for some $\lambda>0$ independent of~$\varphi$. Observe that this is a weaker condition than the condition~\eqref{eqn:elliptic:slices}. The number $\varepsilon$ depends only on $\lambda$, $\doublebar{\mat A}_{L^\infty(\R^n)}$, the order $2m$ of the operator~$L$, and the dimension~$\dmn$, and $C_p$ depends only on $p$, $\lambda$, $\doublebar{\mat A}_{L^\infty(\R^n)}$, $n$, and~$m$.

These estimates are valid for $\arr g$ and $\arr h$ in dense subspaces of $L^p(\R^n)$ for the indicated range of~$p$, and for $\arr \varphi$ and $\arr f$ that satisfy $\arr \varphi(x)=\nabla^{m-1} \Phi(x,0)$ and $\arr f(x)=\nabla^{m-1}F(x,0)$ for some functions $\Phi$, $F\in C^\infty_0(\R^\dmn)$. We may extend $\s^L$ and $\s^L_\nabla$ by density to operators on all of $L^p(\R^n)$, and $\D^{\mat A}$ to an operator on closed subspaces of the Sobolev space $\dot W^{1,p}(\R^n)$ and the Lebesgue space $L^p(\R^n)$ (which we denote by $\dot W\!A^{1,p}_{m-1}(\R^n)$, $\dot W\!A^{0,p}_{m-1}(\R^n)$). If $m\geq 2$, then by equality of mixed partials, $\dot W\!A^{k,p}_{m-1}(\R^n)$ is a proper subspace of $\dot W^{1,p}(\R^n)$ or $L^p(\R^n)$.

$\mathcal{A}_2^*$ is the two-sided area integral given by
\begin{equation}\label{dfn:lusin:*}\mathcal{A}_2^* H(x) = \biggl(\int_{-\infty}^\infty \int_{\abs{x-y}<\abs{t}} \abs{H(y,t)}^2 \frac{dy\,dt}{\abs{t}^\dmn}\biggr) \quad\text{for all $x\in\R^n$}.\end{equation}

In \cite{BarHM18}, we used these four bounds, the trace theorems of \cite{BarHM17pB}, the classic method of layer potentials described above, and some extensions to the classic method of layer potentials pioneered in \cite{Ver84,BarM13,BarM16A} and extended to the higher order case in \cite{Bar17}, to establish existence and uniqueness of solutions to the Neumann problems \eqref{eqn:neumann:regular:2} and~\eqref{eqn:neumann:rough:2}. In particular, the solutions $w$ and $v$ were both given as $\D^{\mat A} ((\M_{\mat A}^+\D^{\mat A})^{-1}\arr g)$, and invertibility of the operator 
\begin{align*}\M_{\mat A}^+\D^{\mat A}:\dot W\!A^{1,2}_{m-1}(\R^n)&\to (\dot W\!A^{0,2}_{m-1}(\R^n))^*,\\ \M_{\mat A}^+\D^{\mat A}:\dot W\!A^{0,2}_{m-1}(\R^n)&\to (\dot W\!A^{1,2}_{m-1}(\R^n))^*\end{align*}
was established. Thus, the bounds
\begin{equation*}
\doublebar{\mathcal{A}_2^+(t\nabla^m \partial_t w)}_{L^2(\R^n)}\leq C\doublebar{\arr g}_{L^2(\R^n)}
,\qquad
\doublebar{\mathcal{A}_2^+(t\nabla^m v)}_{L^2(\R^n)}\leq C\doublebar{\arr g}_{\dot W^{-1,2}(\R^n)}
\end{equation*} follow directly from the bounds \eqref{eqn:D:lusin:2} and~\eqref{eqn:D:lusin:rough:2}.

In \cite{BarHM18p}, we showed that if $L$ and $\mat A$ are as in the bounds~\cref{eqn:S:lusin:2,eqn:D:lusin:2,eqn:D:lusin:rough:2,eqn:S:lusin:rough:2}, then we also have the nontangential bounds
\begin{align}
\label{eqn:S:N:2}
\doublebar{\widetilde N_*(\nabla^m \s^L\arr g)}_{L^p(\R^n)}
&\leq C_p\doublebar{\arr g}_{L^p(\R^n)}
,&2-\varepsilon&<p<2+\varepsilon
,\\
\label{eqn:S:N:rough:2}
\doublebar{\widetilde N_*(\nabla^{m-1} \s^L_\nabla\arr h)}_{L^p(\R^n)}
&\leq C_p\doublebar{\arr h}_{L^p(\R^n)}
,&2-\varepsilon&<p<2+\varepsilon
,\\
\label{eqn:D:N:2}
\doublebar{\widetilde N_*(\nabla^m \D^{\mat A}\arr \varphi)}_{L^p(\R^n)}
&\leq C_p\doublebar{\arr \varphi}_{\dot W\!A^{1,p}_{m-1}(\R^n)}
,&2-\varepsilon&<p<2+\varepsilon
,\\
\label{eqn:D:N:rough:2}
\doublebar{\widetilde N_*(\nabla^{m-1} \D^{\mat A}\arr f)}_{L^p(\R^n)}
&\leq C_p\doublebar{\arr f}_{\dot W\!A^{0,p}_{m-1}(\R^n)}
,&2-\varepsilon&<p<2+\varepsilon
\end{align}
for $\arr g$, $\arr h$ in $L^p(\R^n)$ and $\arr \varphi$, $\arr f$ in the Whitney-Sobolev or Whitney-Lebesgue spaces indicated.
Here $\widetilde N_*$ is the two-sided modified nontangential maximal function
\begin{equation}
\label{dfn:NTM:modified:*}
\widetilde N_* H(x) = \sup
\biggl\{\biggl(\fint_{B((y, s),\abs{s}/2)} \abs{H(z,t)}^2\,dz\,dt\biggr)^{1/2}:
s\in\R,\>
\abs{x-y}< \abs{s}
\biggr\}
.\end{equation}
As the problems~\eqref{eqn:neumann:regular:2} and~\eqref{eqn:neumann:rough:2} were solved using the method of layer potentials, the bounds \eqref{eqn:D:N:2} and \eqref{eqn:D:N:rough:2} immediately yield the bounds~\eqref{eqn:neumann:N:2}.

We remark that the bounds in \cite{BarHM18p} are generally stated in terms of the one-sided nontangential maximal operator~$\widetilde N_+$ (or $\mathcal{A}_2^+$); however, \cite[Section~3.3]{BarHM18p} gives the two-sided estimates.

The second of the two main results of the present paper is to expand the range of $p$ in the bounds \cref{eqn:S:lusin:2,eqn:D:lusin:2,eqn:D:lusin:rough:2,eqn:S:lusin:rough:2} and \cref{eqn:S:N:2,eqn:D:N:2,eqn:D:N:rough:2,eqn:S:N:rough:2}.
We will use these bounds on potentials to prove Theorem~\ref{thm:neumann:p:rough}.
In a forthcoming paper \cite{Bar19pC}, we will use these bounds to establish existence of solutions to the $L^p$ Neumann problem and uniqueness of solutions to the $\dot W^{-1,p}$ Neumann problem~\eqref{eqn:neumann:rough:p}. In future work, we hope to use the method of layer potentials to solve the Dirichlet problem as well as the Neumann problem.

To describe the ranges of $p$ in our results, we recall the higher order generalization of Meyers's reverse H\"older estimate proven in \cite{Cam80,AusQ00,Bar16}. Specifically, by \cite[Theorem~24]{Bar16}, if $j$ is an integer with $0\leq j\leq m$, and if $L$ is an operator of order $2m$ of the form~\eqref{eqn:divergence} associated to coefficients $\mat A$ that are uniformly bounded and satisfy the ellipticity condition~\eqref{eqn:elliptic}, then there is an extended real number $p_{j,L}^+$ in $(2,\infty]$ such that, if $0<q<p<p_{j,L}^+$, there is a constant $c(j,L,p,q)$ such that
\begin{equation}\label{eqn:Meyers}\biggl(\int_{B(X_0,r)} \abs{\nabla^{m-j}u}^p\biggr)^{1/p}
\leq \frac{c(j,L,p,q)}{r^{\pdmn(1/q-1/p)}} \biggl(\int_{B(X_0,2r)} \abs{\nabla^{m-j}u}^q\biggr)^{1/q}
\end{equation}
whenever $Lu=0$ in $B(X_0,2r)$.
As noted in \cite{Bar16}, it follows from the Gagliardo-Nirenberg-Sobolev inequality that
\begin{equation}
\label{eqn:Meyers:p}
\frac{1}{p_{j,L}^+} \leq \max\biggl(0, \frac{1}{p_{0,L}^+}-\frac{j}{\dmn}\biggr)
\leq \max\biggl(0, \frac{1}{2}-\frac{j}{\dmn}\biggr).\end{equation}
Furthermore, $p_{0,L}^+\geq 2+\varepsilon$ for some $\varepsilon$ depending only on $m$, $\dmn$, $\doublebar{\mat A}_{L^\infty}$ and the number $\lambda$ in the bound~\eqref{eqn:elliptic}, and the constant $c(j,L,p,q)$ may be bounded from above by a constant depending only on $p$, $q$ and the same parameters, at least for $1/q>1/p>1/(2+\varepsilon)-j/\pdmn$.

In Propositions \ref{prp:2d:p0} and~\ref{prp:Meyers:t-independent} below, we will show that if $\mat A$ is $t$-independent, then we have the stronger bounds
\begin{equation}\label{eqn:Meyers:bound}
\left\{\begin{aligned}
p_{0,L}^+=p_{1,L}^+&=\infty &\text{ if }\dmn=2,\\
p_{1,L}^+&=\infty &\text{ if }\dmn=3,\\
p_{1,L}^+&>\frac{2n}{n-2}+\varepsilon&\text{ if }\dmn\geq 4
\end{aligned}\right.\end{equation}
where as usual $\varepsilon$ depends only on $m$, $n$, $\lambda$, and $\doublebar{\mat A}_{L^\infty(\R^n)}$.

We will establish the following bounds on layer potentials.

\begin{thm}\label{thm:potentials} Suppose that $L$ is an operator of the form~\eqref{eqn:weak} of order~$2m$ associated with bounded coefficients $\mat A$ that satisfy the ellipticity condition \eqref{eqn:elliptic} and are $t$-independent in the sense of formula~\eqref{eqn:t-independent}.

Then the double and single layer potentials $\D^{\mat A}$, $\s^L$ and $\s^L_\nabla$, originally defined as in Sections~\ref{sec:dfn:D} and~\ref{sec:dfn:S} below, extend by density to operators that satisfy the following bounds for all $p$ in the given ranges and all inputs $\arr f$, $\arr g$, $\arr h$, and $\arr\varphi$ in the indicated spaces.
\begin{align}
\label{eqn:D:N:intro}
\doublebar{\widetilde N_*(\nabla^m\D^{\mat A}\arr \varphi)}_{L^p(\R^n)}
&\leq C(0,p)\doublebar{\arr \varphi}_{\dot W\!A_{m-1,1}(\R^n)}, & 2&\leq p<p_{0,L}^+
,\\
\label{eqn:S:N:intro}
\doublebar{\widetilde N_*(\nabla^m\s^L\arr g)}_{L^p(\R^n)}
&\leq C(0,p)\doublebar{\arr g}_{L^p(\R^n)}, & 2&\leq p<p_{0,L}^+
,\\
\label{eqn:D:N:rough:intro}
\doublebar{\widetilde N_*(\nabla^{m-1}\D^{\mat A}\arr f)}_{L^p(\R^n)}
&\leq C(1,p)\doublebar{\arr f}_{\dot W\!A_{m-1}^{0,p}(\R^n)}, & 2&\leq p<p_{1,L}^+
,\\
\label{eqn:S:N:rough:intro}
\doublebar{\widetilde N_*(\nabla^{m-1}\s^L_\nabla\arr h)}_{L^p(\R^n)}
&\leq C(1,p)\doublebar{\arr h}_{L^p(\R^n)}, & 2&\leq p<p_{1,L}^+
,\\
\label{eqn:D:lusin:rough:intro}
\doublebar{\mathcal{A}_2^*(t\nabla^m\D^{\mat A}\arr f)}_{L^p(\R^n)}
&\leq C(1,p)\doublebar{\arr f}_{\dot W\!A_{m-1}^{0,p}(\R^n)}, & 2&\leq p<p_{1,L}^+
,\\
\label{eqn:S:lusin:rough:intro}
\doublebar{\mathcal{A}_2^*(t\nabla^m\s^L_\nabla\arr h)}_{L^p(\R^n)}
&\leq C(1,p)\doublebar{\arr h}_{L^p(\R^n)}, & 2&\leq p<p_{1,L}^+
,\\
\label{eqn:D:lusin:intro}
\doublebar{\mathcal{A}_2^*(t\nabla^m\partial_t\D^{\mat A}\arr \varphi)}_{L^p(\R^n)}
&\leq C(1,p)\doublebar{\arr \varphi}_{\dot W\!A_{m-1}^{1,p}(\R^n)}, & 2&\leq p<p_{1,L}^+
,\\
\label{eqn:S:lusin:intro}
\doublebar{\mathcal{A}_2^*(t\nabla^m\partial_t\s^L\arr g)}_{L^p(\R^n)}
&\leq C(1,p)\doublebar{\arr g}_{L^p(\R^n)}, & 2&\leq p<p_{1,L}^+
.\end{align}
Here the numbers $p_{j,L}^+$ are as in the bound~\eqref{eqn:Meyers}, and in particular satisfy the bound~\eqref{eqn:Meyers:bound}. The constants $C(j,p)$ depend only on the standard parameters $m$, $n$, $\lambda$, $\doublebar{\mat A}_{L^\infty(\R^n)}$, the number $p$, and the constants $c(j,L,p,2)$ in the bound~\eqref{eqn:Meyers}.
\end{thm}

We remark that if $2m=2$, then many cases of theorem are known.

If $2m=2$, and if $\mat A$ is constant or if $\dmn=2$, then all eight of the bounds~\cref{eqn:D:N:rough:intro,eqn:S:N:rough:intro,eqn:D:lusin:rough:intro,eqn:S:lusin:rough:intro,eqn:D:N:intro,eqn:S:N:intro,eqn:D:lusin:intro,eqn:S:lusin:intro} are valid for all $p$ with $1< p<\infty$. See \cite[Theorem~12.7]{AusS16}.

If $2m=2$ and the well known De Giorgi-Nash-Moser regularity conditions are valid,
(which is true if $\mat A$ is real and $2m=2$, and which by \cite[Appendix~B]{AlfAAHK11} is true for complex $t$-independent coefficients in dimension $\dmn=3$),
it was shown in \cite{AusM14,HofKMP15B,HofMayMou15,HofMitMor15}
that there is some $\varepsilon>0$ such that
the four bounds
\cref{eqn:D:N:rough:intro,eqn:S:N:rough:intro,eqn:D:lusin:rough:intro,eqn:S:lusin:rough:intro}
are valid for $2-\varepsilon<p<\infty$,
the bounds
\cref{eqn:S:N:intro,eqn:D:N:intro}
are valid for $1<p<2+\varepsilon$,
and the bound~\eqref{eqn:S:lusin:intro} is valid for $1<p<\infty$.
If $2m=2$ and the coefficients~$\mat A$ are real, then the bound~\eqref{eqn:D:lusin:intro} follows from the bound~\eqref{eqn:D:N:intro} with the same value of~$p$ and from \cite[Theorem~1.7]{HofKMP15A}, and so is valid for $1<p<2+\varepsilon$.

Finally, in \cite[Theorem~12.7]{AusS16}, it was established that for general second order $t$-independent systems in dimension $\dmn\geq 4$,
the three bounds
\cref{eqn:D:lusin:rough:intro,eqn:S:N:rough:intro,eqn:D:N:rough:intro} and the special case~\eqref{eqn:S:lusin:rough:intro:AS16}
of the bound~\eqref{eqn:S:lusin:rough:intro}
are valid for $2-\varepsilon<p<\frac{2n}{n-2}+\varepsilon$, and the four bounds
\cref{eqn:S:lusin:intro,eqn:D:N:intro,eqn:D:lusin:intro,eqn:S:N:intro}
are valid for  $\frac{2n}{n+2}-\varepsilon<p<2+\varepsilon$, where $\varepsilon$ is a positive number.

Recall that if $\dmn\geq 4$ then $p_{1,L}^+\geq\frac{2n}{n-2}+\varepsilon$; we remark that the De Giorgi-Nash-Moser condition implies that $p_{1,L}^+=\infty$, and if $\mat A$ is constant then $p_{0,L}^+=\infty$. To the author's knowledge, even if $2m=2$, if $\dmn\geq 3$ and the coefficients are variable then the $2+\varepsilon\leq p<p_{1,L}^+$ case of the bound~\eqref{eqn:D:lusin:intro} is new; the $2+\varepsilon\leq p<p_{1,L}^+$ case of the bound~\eqref{eqn:S:lusin:intro} is known for coefficients that satisfy the De Giorgi-Nash-Moser condition but is new for general second order operators.

\subsection{The Neumann subregularity problem}

In this section and in Section~\ref{sec:intro:extrapolation} we will discuss the historical context of Theorem~\ref{thm:neumann:p:rough}. Specifically, in this section we will discuss the history of well posedness results for the Neumann problem, especially the Neumann problem with $\dot W^{-1,p}$ boundary data, while in the next section we will discuss the historical antecedents of our particular method of proof of well posedness.

The Neumann problem for a general system of (possibly higher order) elliptic equations in a domain~$\Omega$ may be written as
\begin{align}
\label{eqn:neumann:general}
(L\vec u)_j=\sum_{\substack{\abs\alpha=m\\\abs\beta=m}} \sum_{k=1}^N \partial^\alpha(A_{\alpha\beta}^{jk}\partial^\beta u_k)&=0 \text{ in }\Omega\text{ for }1\leq j\leq N,
&
\M_{\mat A}^\Omega \vec u &\owns\arr g
,\end{align}
where $\M_{\mat A}^\Omega \vec u $ is given by
\begin{equation}
\label{eqn:neumann:system}
\arr g\in \M_{\mat A}^\Omega \vec u \text{ if }
\sum_{j=1}^N\sum_{\abs\gamma=m-1}\int_{\partial\Omega} \partial^\gamma\varphi_j\,g_{j,\gamma}\,d\sigma
= \sum_{\substack{\abs\alpha=m\\\abs\beta=m}} \sum_{j,k=1}^N \int_\Omega \partial^\alpha \varphi_j\,A^{jk}_{\alpha\beta}\,\partial^\beta u_k.\end{equation}

We mention some applications of the Neumann problem.
As observed in \cite{She07B}, an appropriate choice of coefficients $\mat A$ shows that the traction boundary problem for the Lam\'e system of elastostatics
\begin{equation}
\label{eqn:Lame}
\left\{\begin{aligned}
\mu\Delta \vec u + (\mu+\lambda) \nabla\Div \vec u&=0 \text{ in }\Omega,
\\
\mu (\nabla \vec u+(\nabla \vec u)^T)\nu
+
\lambda \nu \Div \vec u &= \vec g \text{ on }\partial\Omega
,
\end{aligned}\right.
\end{equation}
is a problem of the form~\eqref{eqn:neumann:general} with boundary data of the form~\eqref{eqn:neumann:system}.

The inhomogeneous Neumann problem for the bilaplacian with zero boundary data is given by
\begin{equation}\label{eqn:neumann:biharmonic:inhomogeneous}
(-\Delta)^2 u=h \text{ in }\Omega,
\quad \vecM^\Omega_\rho u\owns\vec 0\text{ on }\partial\Omega.\end{equation}
where $\vecM^\Omega_\rho$ is the Neumann boundary operator for the biharmonic equation given by the formula
\begin{equation}
\label{eqn:neumann:biharmonic}
\vec g\in \vecM^\Omega_\rho u\text{ if }
\int_{\Omega} \rho\Delta u\Delta \varphi +(1-\rho)\sum_{j,k=1}^\dmn \partial_{x_jx_k}u\,\partial_{x_jx_k}\varphi
-\Delta^2 u\,\varphi
= \int_{\partial\Omega} \vec g\cdot \nabla \varphi
\end{equation}
for all sufficiently smooth test functions~$\varphi$. The number $\rho$ is called the Poisson ratio. This inhomogeneous problem describes a thin elastic plate with free edges, acted on by vertical forces of surface density~$h$. The Poisson ratio is physically meaningful and depends on the material of the plate.
See \cite{Nad63,GirN95,Ver05,Swe09} or the survey paper \cite{BarM16B}.

The theory of the Neumann problem is closely tied to the theory of the Dirichlet problem, which may be written as
\begin{equation}
\label{eqn:Dirichlet:system}
L\vec u=\vec 0 \text{ in }\Omega,\quad\nabla^{m-1} \vec u=\arr f \text{ on }\partial \Omega
,\end{equation}
where $L$ is as in the problem~\eqref{eqn:neumann:general}.
Both problems have been investigated for many different operators~$L$. It is common to study the Dirichlet problem with the estimate
\begin{equation}
\label{eqn:N:dirichlet}
\doublebar{\widetilde N_\Omega(\nabla^{m-1} \vec u)}_{L^p(\partial\Omega)}
\leq C \doublebar{\arr f}_{L^p(\partial\Omega)}\end{equation}
where $\widetilde N_\Omega$ is either the standard nontangential maximal function $N_\Omega $ given by $N_\Omega H(X)=\sup\{\abs{H(Y)}:\abs{X-Y}<2\dist(Y,\partial\Omega)\}$, or the modified nontangential maximal operator introduced in \cite{KenP93}. See, for example, \cite{Dah79} (the Laplace operator $L=-\Delta$), \cite{JerK81A} (second order operators with real symmetric $t$-independent coefficients), \cite{KenKPT00,HofKMP15A} (second order operators with real nonsymmetric $t$-independent coefficients), \cite{DahKV88,FabKV88,Gao91,She06A} (second order systems with real symmetric constant coefficients), \cite{MarMMM16} (second order systems with constant coefficients in the half-space), \cite{DahKV86,PipV92,She06A} (the bilaplacian $L=(-\Delta)^2$), and \cite{PipV95A,Ver96} (higher order systems with real symmetric constant coefficients).

It is also common to study the Dirichlet problem with the estimate
\begin{equation}
\label{eqn:N:regularity}
\doublebar{\widetilde N_\Omega(\nabla^{m} \vec u)}_{L^p(\partial\Omega)}
\leq C \doublebar{\arr f}_{\dot W^{1,p}(\partial\Omega)}\end{equation}
where ${\dot W^{1,p}(\partial\Omega)}$ is the boundary Sobolev space of functions whose tangential derivatives lie in~$L^p(\partial\Omega)$. This is often called the Dirichlet regularity problem.
See, for example, \cite{JerK81B,Ver84} ($L=-\Delta$), \cite{KenP93} (second order operators with real symmetric $t$-independent coefficients), \cite{KenR09,Rul07,HofKMP15B} (second order operators with real nonsymmetric $t$-independent coefficients), \cite{DahKV88,FabKV88,Gao91,She06A,MarMMM16} (second order systems with constant coefficients), \cite{Ver90,PipV92,KilS11B} (the bilaplacian $L=(-\Delta)^2$), and \cite{PipV95A,Ver96} (higher order systems with real symmetric constant coefficients).

By contrast, the Neumann problem has been studied primarily under the estimate
\begin{equation}
\label{eqn:N:neumann}\doublebar{\widetilde N_\Omega(\nabla^{m} \vec u)}_{L^p(\partial\Omega)}
\leq C \doublebar{\arr g}_{L^p(\partial\Omega)}.\end{equation}
See, for example, \cite{JerK81B,DahK87} ($L=-\Delta$), \cite{KenP93} (second order operators with real symmetric $t$-independent coefficients), \cite{KenR09,Rul07} (second order operators in dimension $\dmn=2$ with real nonsymmetric $t$-independent coefficients), \cite{DahKV88,FabKV88,She07B} (second order systems, such as the Lam\'e system~\eqref{eqn:Lame}, with real symmetric constant coefficients), \cite{CohG85,Ver05,She07B} (the bilaplacian $L=(-\Delta)^2$), and \cite{BarHM18} (higher order self-adjoint operators with $t$-independent coefficients).

It has only been relatively recently that the Neumann problem has been considered for $\doublebar{\widetilde N_\Omega(\nabla^{m-1} u)}\in{L^p(\partial\Omega)}$. For such a problem to be well posed, the Neumann boundary data must have one fewer degree of smoothness than in the bound~\eqref{eqn:N:neumann}. Thus, the Neumann subregularity problem is the problem~\eqref{eqn:neumann:general} together with the estimate
\begin{equation}
\label{eqn:N:neumann:subregular}\doublebar{\widetilde N_\Omega(\nabla^{m-1} \vec u)}_{L^p(\partial\Omega)}
\leq C \doublebar{\arr g}_{\dot W^{-1,p}(\partial\Omega)}.\end{equation}
Even for the Laplace operator, the Neumann subregularity problem was not studied until \cite{Ver05}, when it was needed to solve the standard Neumann problem (with the estimate~\eqref{eqn:N:neumann}) for the biharmonic operator~$(-\Delta)^2$. The subregularity problem for $(-\Delta)^2$ was also studied in \cite{Ver05}. Inspired by \cite{Ver05}, Hofmann, Mayboroda and the author chose to include results for the subregularity problem for higher order self-adjoint operators with $t$-independent coefficients in \cite{BarHM18}.

We would also like to mention \cite{AusM14,AusS16}, and the theory of boundary value problems in fractional smoothness spaces. \cite{AusM14} contains numerous extrapolation and duality type results for second order elliptic systems with $t$-independent coefficients that satisfy a boundary regularity condition, and in particular contains a duality result between the standard Neumann problem and the subregular Neumann problem. \cite{AusS16} considers the equivalences of norms
\begin{align*}
\doublebar{\widetilde N_+(\nabla \vec w)}_{L^p(\R^n)}
\approx \doublebar{\mathcal{A}_2^+(t\nabla\partial_t \vec w)}_{L^p(\R^n)}
&\approx \doublebar{\vec w(\,\cdot\,,0)}_{\dot W^{1,p}(\R^n)} + \doublebar{\scalarM_{\mat A}^+ \vec w}_{L^p(\R^n)}
,\\
\doublebar{\widetilde N_+ \vec v}_{L^p(\R^n)}
\approx \doublebar{\mathcal{A}_2^+(t\nabla \vec v)}_{L^p(\R^n)}
&\approx \doublebar{\vec v(\,\cdot\,,0)}_{L^p(\R^n)} + \doublebar{\scalarM_{\mat A}^+ \vec v}_{\dot W^{-1,p}(\R^n)} \end{align*}
for solutions $v$ and $w$ to second order $t$-independent systems in the upper half space. This encompasses both Fatou-type results (showing that nontangentially bounded solutions do have boundary values) and their converses; the full equivalence necessarily includes Neumann boundary values on the right hand side, and so must contend with subregular Neumann boundary values.

The Besov and Triebel-Lizorkin spaces $\dot B^{s,p}_q(\partial\Omega)$ and $F^{s,p}_q(\partial\Omega)$, where the parameter $s$ measures smoothness, may be viewed as lying between the spaces $L^p(\partial\Omega)$ and $\dot W^{1,p}(\partial\Omega)$ (if $0<s<1$) or between the spaces $L^p(\partial\Omega)$ and $\dot W^{-1,p}(\partial\Omega)$ (if $-1<s<0$). See, for example, the standard texts \cite{Tri83,RunS96}.

The Dirichlet problem has been studied with data in $\dot B^{s,p}_q(\partial\Omega)$ or $\dot F^{s,p}_q(\partial\Omega)$, $0<s<1$. See, for example, \cite{JerK95,MayMit04A} (the Laplace operator $L=-\Delta$), \cite{BarM16A,AmeA18} (second order operators with $t$-independent coefficients), \cite{MitMW11,MitM13B} (the bilaplacian $L=(-\Delta)^2$), \cite{MitM13A} (systems, possibly of higher order, with constant coefficients), and \cite{MazMS10} (systems with variable coefficients).

The Neumann problem has been studied for boundary data in fractional smoothness spaces; it is generally well posed only for negative orders of smoothness. See, for example, \cite{FabMM98,Zan00,MayMit04A} (the Laplace operator $L=-\Delta$), \cite{BarM16A,AmeA18} (second order operators with $t$-independent coefficients), \cite{MitM13B} (the bilaplacian $L=(-\Delta)^2$), and \cite{MitM13A} (systems with constant coefficients). Indeed \cite{FabMM98,Zan00,MayMit04A} predate \cite{Ver05}, and thus the fractional smoothness case was the first studied instance of the Neumann problem with boundary data in a negative smoothness space.

\subsection{Extrapolation techniques}
\label{sec:intro:extrapolation}

In this section we will discuss the historical antecedents of the method of proof of Theorems~\ref{thm:neumann:p:rough} and~\ref{thm:potentials}.
Specifically, we will discuss the papers \cite{She06A,She06B}, on which our argument is modeled, and related works.

Recall the general Dirichlet problem for a system~\eqref{eqn:Dirichlet:system}. Suppose that $\mat A$ is real, constant, and satisfies the symmetry condition $A_{\alpha\beta}^{jk}=A_{\beta\alpha}^{kj}$ and the Legendre-Hadamard ellipticity condition (see \cite{PipV95B}), and that $\Omega$ is a bounded Lipschitz domain.
In \cite{PipV95B} it was shown that if $p$ is sufficiently close to~$2$, then  the problem \eqref{eqn:Dirichlet:system}, with either the estimate~\eqref{eqn:N:dirichlet} or the estimate~\eqref{eqn:N:regularity}, is well posed.

In \cite{She06B}, Shen used good-$\lambda$ inequalities, the case $p$ near~$2$, and local regularity afforded by the $\dot W^{1,2}$ estimate~\eqref{eqn:N:regularity} to show that the Dirichlet problem with the $L^p$ estimate~\eqref{eqn:N:dirichlet} is well posed for $2<p<\frac{2n}{n-2}+\varepsilon$, where $\dmn\geq 4$ and $\Omega\subset\R^\dmn$.
This was done earlier in the special cases $2m=2$ and for the polyharmonic operator $(-\Delta)^m$ in \cite{She06A}. In the second order case $2m=2$, a duality argument allowed Shen to show that the regularity problem is well posed for $\frac{2n}{n+2}-\varepsilon<p<2$. (A similar duality argument for the biharmonic operator $\Delta^2$ was established in \cite{KilS11B}.)

Our proof of Theorem~\ref{thm:neumann:p:rough} will be modeled on the arguments of \cite{She06A,She06B}. In a forthcoming paper \cite{Bar19pC}, we will establish a duality argument that will allow us to prove a similar theorem regarding the Neumann problem with the estimate~\eqref{eqn:N:neumann} for $\arr g\in L^p(\R^n)$, $\frac{2n}{n+2}-\varepsilon<p<2$.

We wish to mention some other extrapolation arguments that will not be used in the present paper but which have been used in the past to prove results of interest.

First, similar arguments have been used for Neumann problems in the biharmonic and second order constant cases.

The homogeneous biharmonic Neumann problem is given by
\begin{equation*}
(-\Delta)^2 u=0 \text{ in }\Omega,
\quad \vecM^\Omega_\rho u\owns\vec g\text{ on }\partial\Omega,\end{equation*}
where $\vec M^\Omega_\rho$ is given by formula~\eqref{eqn:neumann:biharmonic}. (It complements the inhomogeneous problem~\eqref{eqn:neumann:biharmonic:inhomogeneous}.) It was shown in \cite{Ver05} that for some values of the parameter~$\rho$, if $p$ is close enough to $2$ then the biharmonic Neumann problem with the estimate~\eqref{eqn:N:neumann} is well posed. Similarly, it was shown in \cite{DahKV88} that the traction boundary value problem~\eqref{eqn:Lame} is well posed for boundary data in~$L^2(\partial\Omega)$ and with nontangential estimates. It was observed in \cite{She07B} that the arguments of \cite{DahKV88,FabKV88} imply well posedness of the $L^2$ Neumann problem for general second order systems of the form~\eqref{eqn:neumann:system} with real symmetric coefficients that satisfy the Legendre-Hadamard ellipticity condition.

In \cite{She07B}, Shen extrapolated from well posedness of the $L^2$ Neumann problem for second order systems and the biharmonic equation to well posedness of the $L^p$ Neumann problem, where $\frac{2n}{n+2}-\varepsilon<p<2$ in dimension $\dmn\geq 4$. We remark that the method of proof of \cite{She07B} is somewhat different from (and more complicated than) that of \cite{She06A,She06B}, as \cite{She06A,She06B} begin with $L^2$ well posedness and derive $L^p$ well posedness for certain values of $p$ greater than~$2$, while \cite{She07B} begins with $L^2$ well posedness and derives $L^p$ well posedness for certain values of~$p$ less than~$2$. In the present paper, despite our focus on the Neumann problem, we will use the methods that were used in \cite{She06A,She06B} to solve the Dirichlet problem rather than the more complicated methods used in \cite{She07B} to solve the Neumann problem. It is for this reason that the present paper proves results for the subregular problem~\eqref{eqn:neumann:rough:p} and leaves the $L^p$ analogue of the Neumann problem~\eqref{eqn:neumann:regular:2} to the forthcoming paper \cite{Bar19pC}.

We also wish to mention that Shen's extrapolation method is fairly new.

For some boundary value problems, the $L^p$ problem was solved simultaneously with the $L^2$ problem. See, for example, the harmonic $L^p$-Dirichlet problem in \cite{Dah79,JerK81A} and $\dot W^{1,p}$-regularity problem in \cite{Ver84}.

In the cases where the $L^2$-Neumann problem was established first and then used to solve the $L^p$-Neumann problem, a more common technique has been to show that the corresponding boundary value problem with boundary data in the Hardy space $H^1$ is well posed, and then interpolate to yield $L^p$ well posedness for $1<p<2$. This technique was used in \cite{DahK87} for the harmonic Neumann problem
\begin{equation*}-\Delta u=0 \text{ in }\Omega,\quad \nu\cdot \nabla u=g\text{ on }\partial\Omega,\quad \doublebar{N_\Omega(\nabla u)}_{L^p(\partial\Omega)}\leq C\doublebar{g}_{L^p(\partial\Omega)},\end{equation*}
in \cite{DahK90} for the Lam\'e system~\eqref{eqn:Lame} with the estimate~\eqref{eqn:N:neumann}, in \cite{KenP93} for the Neumann problem
\begin{equation*}\Div \mat A\nabla u=0 \text{ in }\Omega, \quad \nu\cdot \mat A \nabla u=g\text{ on }\partial\Omega,\quad \doublebar{\widetilde N_\Omega(\nabla u)}_{L^p(\partial\Omega)}\leq C\doublebar{g}_{L^p(\partial\Omega)},\end{equation*}
in a starlike Lipschitz domain~$\Omega$ with coefficients $\mat A$ that are real, symmetric, and radially independent, and in \cite{AusM14} for any second-order system of the form~\eqref{eqn:neumann:system} with variable $t$-independent coefficients, in the domain above a Lipschitz graph, that satisfy a boundary regularity condition and for which the $L^2$-Neumann problem is solvable. Similar extrapolation results were found in these four papers for the corresponding $\dot W^{1,p}$-regularity problem, and also in \cite{PipV92,PipV95A,Ver96} for the biharmonic or general symmetric constant coefficient regularity problem~\eqref{eqn:Dirichlet:system} with the estimate~\eqref{eqn:N:regularity} in a three-dimensional domain~$\Omega$.

The $H^1$ interpolation technique necessarily yields well posedness for $L^p$ boundary data for the entire range $1<p<2$. The sharp range of $p$ for which a higher order $L^p$-Neumann problem is well posed is not known, even for special cases such as the biharmonic Neumann problem. However, the range of $p$ for which the biharmonic $\dot W^{1,p}$-regularity problem
\begin{equation*}(-\Delta)^2 u=0 \text{ in }\Omega,
\quad \nabla u=\vec f\text{ on }\partial\Omega,
\quad
\doublebar{N_\Omega(\nabla^2 u)}_{L^p(\partial\Omega)}\leq C\doublebar{\vec f}_{\dot W^{1,p}(\partial\Omega)}\end{equation*}
is well posed in all Lipschitz domains $\Omega\subset\R^\dmn$
is known to be $[6/5,2]$ in dimension $\dmn=4$,  is $[4/3,2]$ in dimension $\dmn=5$, $6$, or~$7$, and is known to be a subset of $[4/3,2]$ in dimension $\dmn\geq 8$.
Similarly, the $L^p$-Dirichlet problem
\begin{equation*}(-\Delta)^2 u=0 \text{ in }\Omega,
\quad \nabla u=\vec f\text{ on }\partial\Omega,
\quad
\doublebar{N_\Omega(\nabla u)}_{L^p(\partial\Omega)}\leq C\doublebar{\vec f}_{L^p(\partial\Omega)}\end{equation*}
is well posed in all Lipschitz domains $\Omega\subset\R^\dmn$
if and only if $p$ is in $[2,6]$ ($\dmn=4$), in $[2,4]$ ($\dmn=5$, $6$, or~$7$), or in a (currently unknown) subset of $[2,4]$ ($\dmn\geq 8$).

This suggests that the Neumann problem~\eqref{eqn:neumann:regular:2}, with estimates in $L^p(\R^n)$ rather than $L^2(\R^n)$, is probably not well posed for the full range $1<p\leq 2$ in dimension $4$ and higher, and so $H^1$ interpolation techniques are probably not applicable. Similarly, the $\dot W^{-1,p}$-Neumann problem~\eqref{eqn:neumann:rough:p} is probably not well posed for the full dual range $2\leq p<\infty$.

\subsection{Outline}

The outline of this paper is as follows. In Section~\ref{sec:dfn} we will define our terminology, including supplying precise definitions for the Whitney-Sobolev spaces $\dot W\!A^{s,p}_{m-1}(\R^n)$ and of the double and single layer potentials $\D^{\mat A}$ and~$\s^L$.
In Section~\ref{sec:regular} we will state some known regularity results for solutions to elliptic equations, and will prove the bounds~\eqref{eqn:Meyers:bound} for $t$-independent coefficients.

In Section~\ref{sec:N:potentials} we will establish nontangential maximal estimates on layer potentials, that is, the estimates~\cref{eqn:S:N:rough:intro,eqn:D:N:rough:intro,eqn:S:N:intro,eqn:D:N:intro} in Theorem~\ref{thm:potentials}. In Section~\ref{sec:N:neumann} we will establish the nontangential component of Theorem~\ref{thm:neumann:p:rough} by showing that, if $\arr g$ lies in a dense subset of $\dot W^{-1,p}\cap \dot W^{-1,2}$, then the solution $v$ to the $\dot W^{-1,2}$-Neumann problem~\eqref{eqn:neumann:rough:2} satisfies the estimate $\doublebar{\widetilde N_+(\nabla^{m-1} v)}_{L^p(\R^n)}\leq C_p\doublebar{\arr g}_{\dot W^{-1,p}(\R^n)}$ for appropriate values of~$p$. The argument for the Neumann problem is much more involved than that for layer potentials; we remark that several of the lemmas proven in Section~\ref{sec:N:potentials} will be of use in Section~\ref{sec:N:neumann}.

Given nontangential estimates, passing to area integral estimates (for both layer potentials and solutions to the Neumann problem) is a relatively straightforward matter; we will do this in Section~\ref{sec:lusin:+}.

\subsection*{Acknowledgements}

The author would like to thank Steve Hofmann and Svitlana Mayboroda for many useful conversations on topics related to this paper. The author would also like to thank
the Mathematical Sciences Research Institute for hosting a Program on Harmonic Analysis,
the Instituto de Ciencias Matem\'aticas for hosting a Research Term on ``Real Harmonic Analysis and Its Applications to Partial Differential Equations and Geometric Measure Theory'',
and
the IAS/Park City Mathematics Institute for hosting a Summer Session with a research topic of Harmonic Analysis,
at which many of the results and techniques of this paper were discussed.

\section{Definitions}\label{sec:dfn}

In this section, we will provide precise definitions of the notation and concepts used throughout this paper.

We will always work with operators~$L$ of order~$2m$ in the divergence form~\eqref{eqn:divergence} (interpreted in the weak sense of formula~\eqref{eqn:weak} below) acting on functions defined in~$\R^\dmn$, $\dmn\geq 2$.

As usual, we let $B(X,r)$ denote the ball in $\R^\dmn$ of radius $r$ and center $X$. We let $\R^\dmn_+$ and $\R^\dmn_-$ denote the upper and lower half spaces $\R^n\times (0,\infty)$ and $\R^n\times(-\infty,0)$; we will identify $\R^n$ with $\partial\R^\dmn_\pm$.
If $Q$ is a cube (or interval), we will let $\ell(Q)$ be its side length (or length), and we let $cQ$ be the concentric cube of side length $c\ell(Q)$. If $E$ is a set of finite measure~$\abs{E}$, we let
\begin{equation*}\fint_E f(x)\,dx=\frac{1}{\abs{E}}\int_E f(x)\,dx.\end{equation*}

If $E$ is a measurable set in Euclidean space and $\arr H$ is a (possibly vector-valued or array-valued) globally defined function, we will let $\1_E \arr H=\chi_E \arr H$, where $\chi_E$ is the characteristic function of~$E$. If $\arr H:E\mapsto V$ is defined in all of $E$ for some vector space~$V$, but is not globally defined, we will let $\1_E \arr H$ be the extension of $\arr H$ by zero, that is,
\begin{equation*}\1_E \arr H(X)=\begin{cases} \arr H(X), & X\in E, \\ \arr 0_V, & \text{otherwise}.\end{cases}\end{equation*}
We will use $\1_\pm$ as a shorthand for $\1_{\R^\dmn_\pm}$.

We let $\nabla$ denote the standard gradient in~$\R^\dmn$. We will let $\nabla_\pureH$ denote either the gradient in $\R^n$, or the gradient in the first $n$ variables in $\R^\dmn$.

\subsection{Multiindices and arrays of functions}

We will routinely work with multiindices in~$(\N_0)^\dmn$. (We will occasionally work with multiindices in $(\N_0)^\dmnMinusOne$.) Here $\N_0$ denotes the nonnegative integers. If $\zeta=(\zeta_1,\zeta_2,\dots,\zeta_\dmn)$ is a multiindex, then we define $\abs{\zeta}$ and $\partial^\zeta$ in the usual ways, as $\abs{\zeta}=\zeta_1+\zeta_2+\dots+\zeta_\dmn$ and $\partial^\zeta=\partial_{x_1}^{\zeta_1}\partial_{x_2}^{\zeta_2} \cdots\partial_{x_\dmn}^{\zeta_\dmn}$.

Recall that a vector $\vec H$ is a list of numbers (or functions) indexed by integers $j$ with $1\leq j\leq N$ for some $N\geq 1$. We similarly let an array $\arr H$ be a list of numbers or functions indexed by multiindices~$\zeta$ with $\abs\zeta=k$ for some $k\geq 1$.
In particular, if $\varphi$ is a function with weak derivatives of order up to~$k$, then we view the gradient $\nabla^k\varphi$ as such an array.

The inner product of two such arrays of numbers $\arr F$ and $\arr G$ is given by
\begin{equation*}\bigl\langle \arr F,\arr G\bigr\rangle =
\sum_{\abs{\zeta}=k}
\overline{F_{\zeta}}\, G_{\zeta}.\end{equation*}
If $\arr F$ and $\arr G$ are two arrays of functions defined in a set $\Omega$ in Euclidean space, then the inner product of $\arr F$ and $\arr G$ is given by
\begin{equation*}\bigl\langle \arr F,\arr G\bigr\rangle_\Omega =
\int_\Omega \langle \arr F(X),\arr G(X)\rangle\,dX =
\sum_{\abs{\zeta}=k}
\int_{\Omega} \overline{F_{\zeta}(X)}\, G_{\zeta}(X)\,dX.\end{equation*}

We let $\vec e_j$ be the unit vector in $\R^\dmn$ in the $j$th direction; notice that $\vec e_j$ is a multiindex with $\abs{\vec e_j}=1$. We let $\arr e_{\zeta}$ be the unit array corresponding to the multiindex~$\zeta$; thus, $\langle \arr e_{\zeta}, \arr H\rangle = H_{\zeta}$.

\subsection{Function spaces and Dirichlet boundary values.}

Let $\Omega$ be a measurable set in Euclidean space. We let $C^\infty_0(\Omega)$ be the space of all smooth functions that are compactly supported in~$\Omega$. We let $L^p(\Omega)$ denote the usual Lebesgue space with respect to Lebesgue measure with norm given by
\begin{equation*}\doublebar{f}_{L^p(\Omega)}=\biggl(\int_\Omega \abs{f(x)}^p\,dx\biggr)^{1/p}.\end{equation*}

If $\Omega$ is a connected open set, then we let the homogeneous Sobolev space $\dot W^{k,p}(\Omega)$ be the space of equivalence classes of functions $u$ that are locally integrable in~$\Omega$ and have weak derivatives in $\Omega$ of order up to~$k$ in the distributional sense, and whose $k$th gradient $\nabla^k u$ lies in $L^p(\Omega)$. Two functions are equivalent if their difference is a polynomial of order at most~$k-1$.
We impose the norm
\begin{equation*}\doublebar{u}_{\dot W^{k,p}(\Omega)}=\doublebar{\nabla^k u}_{L^p(\Omega)}.\end{equation*}
Then $u$ is equal to a polynomial of order at most $k-1$ (and thus equivalent to zero) if and only if its $\dot W^{k,p}(\Omega)$-norm is zero.

If $1<p<\infty$, then we let $\dot W^{-1,p}(\R^n)$ denote the dual space to $\dot W^{1,p'}(\R^n)$, where $1/p+1/p'=1$; this is a space of distributions on~$\R^n$.

The use of a dot to denote homogeneous Sobolev spaces (as opposed to the inhomogeneous spaces $W^{k,p}(\Omega)$ with $\|u\|_{W^{k,p}(\Omega)}^p=\sum_{j=0}^k\doublebar{\nabla^j u}_{L^p(\Omega)}^p$) is by now standard. The use of a dot to denote arrays of functions is also standard (see, for example, \cite{Agm57,CohG83,CohG85,PipV95A,She06B,KilS11B,MitM13A,MitM13B}). We apologize for any confusion arising from this overloading of notation, but these established conventions seem to require it.

We say that $u\in L^p_{loc}(\Omega)$ or $u\in\dot W^{k,p}_{loc}(\Omega)$ if $u\in L^p(U)$ or $u\in\dot W^{k,p}(U)$ for all bounded open sets $U$ with $\overline U\subset\Omega$.

Following \cite{BarHM17pB}, we define the boundary values $\Trace^\pm u$ of a function $u$ defined in an appropriate subset of $\R^\dmn_\pm$ by
\begin{equation}
\label{eqn:trace}
\Trace^\pm  u
= f \text{ in $\Omega$}\quad\text{if}\quad
\lim_{t\to 0^\pm} \doublebar{u(\,\cdot\,,t)-f}_{L^1(K)}=0
\end{equation}
for all compact sets $K\subset\Omega$. We define
\begin{equation*}
\Tr_j^\pm u= \Trace^\pm \nabla^j u.\end{equation*}
We remark that if $\nabla u$ is locally integrable up to the boundary, then $\Trace^\pm u$ exists, and furthermore $\Trace^\pm u$ coincides with the traditional trace in the sense of Sobolev spaces. Furthermore, if $\nabla u$ is locally integrable in a neighborhood of the boundary, then $\Trace^+u=\Trace^-u$ as locally integrable functions; in this case we will refer to the boundary values (from either side) as $\Trace u$.

We are interested in functions with boundary data in Lebesgue or Sobolev spaces. However, observe that if $j\geq 1$, then the components of $\Tr_j^\pm u$ are derivatives of a common function and so must satisfy certain compatibility conditions. We thus define the following Whitney-Lebesgue, Whitney-Sobolev and Whitney-Besov spaces of arrays that satisfy these conditions.

\begin{defn} \label{dfn:Whitney}
Let
\begin{equation*}\mathfrak{D}=\{\Tr_{m-1}^+\varphi:\varphi\text{ smooth and compactly supported in $\R^\dmn$}\}.\end{equation*}

If $1\leq p<\infty$, then we let $\dot W\!A^{0,p}_{m-1}(\R^n)$ be the closure of the set $\mathfrak{D}$ in $L^p(\R^n)$.
We let $\dot W\!A^{1,p}_{m-1}(\R^n)$ be the closure of $\mathfrak{D}$ in $\dot W^{1,p}(\R^n)$. Finally, we let
$\dot W\!A^{1/2,2}_{m-1}(\R^n)$ be the closure of $\mathfrak{D}$ in the Besov space $\dot B^{1/2,2}_{2}(\R^n)$; the norm in this space may be written as
\begin{equation}
\label{eqn:B:norm}
\doublebar{ f}_{\dot B^{1/2,2}_{2}(\R^n)} = \biggl(\int_{\R^n}\abs{\widehat {f}(\xi)}^2\abs{\xi}\,d\xi\biggr)^{1/2}\end{equation}
where $\widehat f$ denotes the Fourier transform of~$f$.
\end{defn}

\begin{rmk}\label{rmk:W2:trace}
It is widely known that $\arr f\in \dot W\!A^{1/2,2}_{m-1}(\R^n)$ if and only if $\arr f=\Tr_{m-1}^+ F$ for some $F$ with $\nabla^m F\in L^2(\R^\dmn_+)$. This was essentially proven in \cite{Liz60,Jaw77}; see  \cite[Lemma~2.6]{BarHM17pA} for further discussion.
\end{rmk}

\begin{rmk} There is an extensive theory of Besov spaces (see, for example, \cite{Tri83,RunS96}). We will make use only of the Besov space $\dot B^{1/2,2}_2(\R^n)$ given by formula~\eqref{eqn:B:norm} and the space $\dot B^{-1/2,2}_{2}(\R^n)$. This space has norm
\begin{equation}
\label{eqn:B:norm:-}
\doublebar{g}_{\dot B^{-1/2,2}_{2}(\R^n)} = \biggl(\int_{\R^n}\abs{\widehat {g}(\xi)}^2\frac{1}{\abs{\xi}}\,d\xi\biggr)^{1/2}.
\end{equation}
The two properties of this space that we will use are, first, that ${\dot B^{-1/2,2}_{2}(\R^n)}$ is the dual space to ${\dot B^{1/2,2}_{2}(\R^n)}$, and, second, that $f\in {\dot B^{1/2,2}_{2}(\R^n)}$ if and only if the gradient $\nabla_\pureH f$ exists in the distributional sense and satisfies $\doublebar{\nabla_\pureH f}_{\dot B^{-1/2,2}_{2}(\R^n)}\approx \doublebar{f}_{\dot B^{1/2,2}_{2}(\R^n)}$.
\end{rmk}

\subsection{Elliptic differential operators and Neumann boundary values}
\label{sec:dfn:L}

Let $\mat A = \begin{pmatrix} A_{\alpha\beta} \end{pmatrix}$ be a matrix of measurable coefficients defined on $\R^\dmn$, indexed by multtiindices $\alpha$, $\beta$ with $\abs{\alpha}=\abs{\beta}=m$. If $\arr H$ is an array indexed by multiindices of length~$m$, then $\mat A\arr H$ is the array given by
\begin{equation*}(\mat A\arr H)_{\alpha} =
\sum_{\abs{\beta}=m}
A_{\alpha\beta} H_{\beta}.\end{equation*}

We let $L$ be the $2m$th-order divergence form operator associated with~$\mat A$. That is, we say that
\begin{equation}
\label{eqn:weak}
Lu=0 \text{ in }\Omega \text{ if }u\in \dot W^{m,2}_{loc}(\Omega) \text{ and } \langle \nabla^m \varphi, \mat A\nabla^m u\rangle_\Omega=0 \text{ for all } \varphi\in C^\infty_0(\Omega).
\end{equation}
This is the standard weak definition of divergence form operators with rough coefficients.

Throughout we require our coefficients to be pointwise bounded and to satisfy the
G\r{a}rding inequality~\eqref{eqn:elliptic}, which we restate here as
\begin{align*}
\re {\bigl\langle\nabla^m \varphi,\mat A\nabla^m \varphi\bigr\rangle_{\R^\dmn}}
&\geq
\lambda\doublebar{\nabla^m\varphi}_{L^2(\R^\dmn)}^2
\quad\text{for all $\varphi\in\dot W^{m,2}(\R^\dmn)$}
\end{align*}
for some $\lambda>0$.
In some cases, we will also require our coefficients to satisfy the stronger G\r{a}rding inequality~\eqref{eqn:elliptic:slices}.

Recall that the Neumann boundary values of a function $w$ that satisfies $Lw=0$ in $\R^\dmn_+$ and $\mathcal{A}_2^+(t\nabla^m\partial_t w)+\widetilde N_+(\nabla^m w)\in L^2(\R^n)$ are given by formula~\eqref{eqn:Neumann:intro}, that is, satisfy
\begin{equation*}\arr g\in \M_{\mat A}^+ w\quad\text{if}\quad
\langle \arr g,\Tr_{m-1}\varphi\rangle_{\R^n}=\langle \mat A\nabla^m w,\nabla^m\varphi\rangle_{\R^\dmn_+}\end{equation*}
for all $\varphi\in C^\infty_0(\R^n)$. We let $\M_{\mat A}^-$ be the analogous operator in the lower half space. We may view $\M_{\mat A}^+ w$ as either an equivalence class of array-valued distributions or as a linear operator on $\{\Tr_{m-1}\varphi:\varphi\in C^\infty_0(\R^\dmn)\}$.
By \cite[Theorem~6.2]{BarHM17pB}, for such $w$ we have that $\M_{\mat A}^+ w$ extends by density to a well defined bounded operator on~$\dot W\!A^{0,2}_{m-1}(\R^n)$.

By \cite[Lemma~2.4]{BarHM17pB}, if $u\in \dot W^{m,2}(\R^\dmn_+)$ and $Lu=0$ in $\R^\dmn_+$ then $\M_{\mat A}^+ u$ is a well defined operator on $\dot W\!A^{1/2,2}_{m-1}(\R^n)$, and formula~\eqref{eqn:Neumann:intro} is valid for $w=u$ and for all $\varphi\in \dot W^{m,2}(\R^\dmn)$.

If $Lv=0$ in $\R^\dmn$ and $\mathcal{A}_2^+(t\nabla^m v)\in L^2(\R^n)$, as in the problem~\eqref{eqn:neumann:rough:2}, then $\M_{\mat A}^+ v$ is as given in \cite{BarHM17pB}; by \cite[Theorem~6.1]{BarHM17pB}, $\M_{\mat A}^+ v$ is a well defined bounded operator on~$\dot W\!A^{1,2}_{m-1}(\R^n)$. Furthermore, if $Lu=0$ in $\R^\dmn_+$ and $\mathcal{A}_2^+(t\nabla^m u)+\mathcal{A}_2^+(t\nabla^m\partial_t u)+\widetilde N_+(\nabla^m u)\in L^2(\R^n)$, then the definitions of $\M_{\mat A}^+ u$ given by formula~\eqref{eqn:Neumann:intro} and \cite{BarHM17pB} coincide.

The numbers $C$ and $\varepsilon$ denote constants whose value may change from line to line, but which are always positive and depend only on the dimension~$\dmn$, the order $2m$ of any relevant operators, the norm $\doublebar{\mat A}_{L^\infty(\R^n)}$ of the coefficients, and the number $\lambda$ in the bound~\eqref{eqn:elliptic:slices} or~\eqref{eqn:elliptic}. Any other dependencies will be noted explicitly.
We say that $A\approx B$ if there are some positive constants $\varepsilon$ and~$C$ depending only on the above quantities such that $\varepsilon B\leq A\leq CB$.

The numbers $p_{j,L}^+$ are always as in the bound~\eqref{eqn:Meyers}.
The notation $C(j,p)$ denotes a constant that depends only on the standard parameters $n$, $m$, $\lambda$, $\doublebar{\mat A}_{L^\infty(\R^n)}$, the number~$p$, and the constant $c(j,L,p,2)$ in the bound~\eqref{eqn:Meyers}. (If $p\leq 2$ then $C(j,p)$ may be taken as depending only on $n$, $m$, $\lambda$, $\doublebar{\mat A}_{L^\infty(\R^n)}$, and~$p$.)

\subsection{The double layer potential}
\label{sec:dfn:D}

In this section we define the double layer potential of Theorem~\ref{thm:potentials}.

We begin with the related Newton potential.
For any $\arr H\in L^2(\R^\dmn)$, by the Lax-Milgram lemma there is a unique function $\Pi^L\arr H$ in $\dot W^{m,2}(\R^\dmn)$ that satisfies
\begin{equation}\label{eqn:newton}
\langle \nabla^m\varphi, \mat A\nabla^m \Pi^L\arr H\rangle_{\R^\dmn}=\langle \nabla^m\varphi, \arr H\rangle_{\R^\dmn}
\quad\text{for all $\varphi\in \dot W^{m,2}(\R^\dmn)$.}
\end{equation}
We refer to the operator $\Pi^L$ as the Newton potential.

Suppose that $\arr f\in \dot W\!A^{1/2,2}_{m-1}(\R^n)$. As mentioned in Remark~\ref{rmk:W2:trace}, $\arr f=\Tr_{m-1}^+ F$ for some $F\in \dot W^{m,2}(\R^\dmn_+)$.
We define
\begin{equation}
\label{dfn:D:newton:+}
\D^{\mat A}\arr f = -\1_+ F + \Pi^L(\1_+ \mat A\nabla^m F)
.\end{equation}
This operator is well-defined, that is, does not depend on the choice of~$F$. See \cite[Section~2.4]{BarHM17} or \cite[Lemma~4.2]{Bar17}. Furthermore, it is antisymmetric about exchange of $\R^\dmn_+$ and $\R^\dmn_-$; that is, if $\Tr_{m-1}^-F=\arr f$ and $F\in \dot W^{m,2}(\R^\dmn_-)$, then
\begin{equation}
\label{dfn:D:newton:-}
\D^{\mat A}\arr f =
-\Pi^L(\1_-\mat A\nabla^m F) + \1_- F
.\end{equation}
See \cite[formula~(2.27)]{BarHM17} or \cite[formula~(4.8)]{Bar17}.
We extend $\D^{\mat A}$ by density. Specifically, in \cite{BarHM18p}, it was shown that if $\arr f\in \dot W\!A^{k,p}_{m-1}(\R^n)\cap\dot W\!A^{1/2,2}_{m-1}(\R^n)$ for $k\in \{0,1\}$ and for $p$ sufficiently close to~$2$, then there is an additive normalization of $\D^{\mat A}\arr f$ that satisfies the bound~\eqref{eqn:D:N:2} or the bounds~\cref{eqn:D:N:rough:2,eqn:D:lusin:rough:2}; by density we may extend $\D^{\mat A}$ to an operator defined on all of $\dot W\!A^{k,p}_{m-1}(\R^n)$ and mapping into $\dot W^{m+k-1,2}_{loc}(\R^\dmn_\pm)\cap\dot W^{m,2}_{loc}(\R^\dmn_\pm)$.

\subsection{The single layer potential}
\label{sec:dfn:S}

Let $\arr g\in \dot B^{-1/2,2}_2(\R^n)$. Then $\arr\varphi\to\langle\arr\varphi,\arr g\rangle_{\R^n}$ is a bounded linear operator on the space $\dot W\!A^{1/2,2}_{m-1}(\R^n)$,
and so $F\to \langle \Tr_{m-1} F,\arr g\rangle_{\R^n}$ is a bounded linear operator on $\dot W^{m,2}(\R^\dmn)$. By the Lax-Milgram lemma, there is a unique function $\s^L\arr g\in \dot W^{m,2}(\R^\dmn)$ that satisfies
\begin{align}
\label{dfn:S}
\langle \nabla^m\varphi, \mat A\nabla^m\s^L\arr g\rangle_{\R^\dmn}&=\langle \Tr_{m-1}\varphi,\arr g\rangle_{\R^n}
&&\text{for all }\varphi\in \dot W^{m,2}(\R^\dmn)
.\end{align}
See \cite{Bar17}. We remark that this definition coincides with the definition of $\s^L\arr g$ involving the Newton potential given in \cite{BarHM17,BarHM17pA}. In \cite{BarHM18p} we showed  that $\s^L$ extends by density to an operator defined on all of $L^p(\R^n)$ for $p$ sufficiently close to~$2$ and that satisfies the bound~\eqref{eqn:S:N:2}.

\begin{rmk}
If $L$ is an operator of the form~\eqref{eqn:weak}, then $L$ may generally be associated to many choices of coefficients $\mat A$; for example, if $A_{\alpha\beta}=\widetilde A_{\alpha\beta}+M_{\alpha\beta}$, where $\arr M$ is constant and $M_{\alpha\beta}=-M_{\beta\alpha}$, then the operators $L$ and $\widetilde L$ associated to $\mat A$ and $\mat{\widetilde A}$ are equal. The single layer potential $\s^L$ depends only on the operator~$L$, while the double layer potential $\D^{\mat A}$ depends on the particular choice of coefficients~$\mat A$.
\end{rmk}

We now recall the operator $\s^L_\nabla$ of \cite{BarHM17pA,BarHM18p}. If $\abs\beta=m$ and $\beta_\dmn\geq 1$, and if
$h\in L^2(\R^n)$, then we let
\begin{equation}\label{eqn:S:S:vertical}
\nabla^{m-1}\s^L_\nabla (h\arr e_\beta)(x,t) = -\nabla^{m-1}\partial_t\s^L(h\arr e_\zeta)(x,t)\quad\text{where }\beta=\zeta+\vec e_\dmn.\end{equation}
If $h\in L^2(\R^n)\cap \dot B^{2,2}_{1/2}(\R^n)$, and if $\abs\beta=m$ and $\beta_\dmn<\abs\beta=m$,
then we let
\begin{equation}\label{eqn:S:S:horizontal}
\nabla^{m}\s^L_\nabla (h\arr e_\beta)(x,t) = -\nabla^{m}\s^L(\partial_{x_j} h\arr e_\zeta)(x,t)\quad\text{where }\beta=\zeta+\vec e_j\end{equation}
where $j$ is any number with $1\leq j\leq n$ and with $\vec e_j\leq \beta$.

As shown in \cite[formula~(2.27)]{BarHM17pA} and  \cite[Lemma~4.4]{BarHM18p}, $\s^L_\nabla$ is well defined in the sense that, if $1\leq \beta_\dmn\leq m-1$, then the two formulas~\eqref{eqn:S:S:vertical} and~\eqref{eqn:S:S:horizontal} yield the same result, and if $\beta_\dmn\leq m-1$ then the choice of distinguished direction $x_j$ in formula~\eqref{eqn:S:S:horizontal} does not matter.
By \cite[Theorem~1.13]{BarHM17pA} and \cite[Theorem~1.15]{BarHM18p}, $\s^L_\nabla\arr h$ satisfies the bounds \eqref{eqn:S:lusin:rough:2} and~\eqref{eqn:S:N:rough:2}, and so $\s^L_\nabla$ extends by density to a well defined operator $\s^L_\nabla:L^p(\R^n)\mapsto \dot W^{m-1,2}_{loc}(\R^\dmn_\pm)\cap \dot W^{m,2}_{loc}(\R^\dmn_\pm)$ for all $p$ sufficiently close to~$2$.

\section{Regularity of solutions to elliptic equations}
\label{sec:regular}

It is well known that solutions to the elliptic equation $Lu=0$ display many regularity properties. In this section we will state some known regularity results that will be used throughout the paper, and will then establish the bounds~\eqref{eqn:Meyers:bound}.

We begin with the higher order analogue of the Caccioppoli inequality. This lemma was proven in full generality in \cite{Bar16} and some important preliminary versions were established in \cite{Cam80,AusQ00}.
\begin{lem}[The Caccioppoli inequality]\label{lem:Caccioppoli}
Let $L$ be an operator of the form~\eqref{eqn:weak} of order~$2m$ associated to bounded coefficients~$\mat A$ that satisfy the ellipticity condition~\eqref{eqn:elliptic}.
Let $ u\in \dot W^{m,2}(B(X,2r))$ with $L u=0$ in $B(X,2r)$.

Then we have the bound
\begin{equation*}
\fint_{B(X,r)} \abs{\nabla^j  u(x,s)}^2\,dx\,ds
\leq \frac{C}{r^2}\fint_{B(X,2r)} \abs{\nabla^{j-1}  u(x,s)}^2\,dx\,ds
\end{equation*}
for any $j$ with $1\leq j\leq m$.
\end{lem}

Next, we mention that if $\mat A$ is $t$-independent then solutions to $Lu=0$ have additional regularity. In particular, the following lemma was proven in the case $m = 1$ in \cite[Proposition 2.1]{AlfAAHK11} and generalized to the case $m \geq 2$ in \cite[Lemma 3.20]{BarHM17pA}.
\begin{lem}\label{lem:slices}
Let $L$ be an operator of the form~\eqref{eqn:weak} of order~$2m$ associated to bounded $t$-independent coefficients~$\mat A$ that satisfy the ellipticity condition~\eqref{eqn:elliptic}.
Let $Q\subset\R^n$ be a cube of side length~$\ell(Q)$ and let $I\subset\R$ be an interval with $\abs{I}=\ell(Q)$.
If $t\in I$ and $Lu=0$ in $2Q\times 2I$, then
\begin{equation*}\int_Q \abs{\nabla^j \partial_t^k u(x,t)}^p\,dx \leq \frac{C(j,p)}{\ell(Q)}
\int_{2Q}\int_{2I} \abs{\nabla^j \partial_s^k u(x,s)}^p\,ds\,dx\end{equation*}
for any $0\leq j\leq m$, any $0< p < p_{m-j,L}^+$, and any integer $k\geq 0$. \end{lem}

In the next two propositions, we will show that if $\mat A$ is $t$-independent, then we can improve on the bound~\eqref{eqn:Meyers:p} on the numbers $p_{j,L}^+$ in the bound~\eqref{eqn:Meyers}. We remark that Lemma~\ref{lem:slices} will be crucial in both cases.

\begin{prp}\label{prp:2d:p0}
Let $L$ be an operator of the form~\eqref{eqn:weak} of order~$2m$ associated to bounded coefficients~$\mat A$ that satisfy the ellipticity condition~\eqref{eqn:elliptic}.

If $\mat A$ is $t$-independent in the sense of formula~\eqref{eqn:t-independent}, and if the ambient dimension $\dmn=2$, then
\begin{equation*}
\esssup_{S} \abs{\nabla^m u} \leq C\biggl(\fint_{2S} \abs{\nabla^m u}^2\biggr)^{1/2}\end{equation*}
whenever $S\subset\R^2$ is a square, $u\in \dot W^{m,2}(2S)$, and $Lu=0$ in~$2S$. In particular,
$p_{0,L}^+=\infty$, where $p_{0,L}^+$ is as in the bound~\eqref{eqn:Meyers}.
\end{prp}

In the second order case $2m=2$, Proposition~\ref{prp:2d:p0} is \cite[Th\'eor\`eme~II.2]{AusT95}.
In order to prove this proposition for $2m\geq 4$, we will need the following lemma.

\begin{lem}
\label{lem:A0:2d}
Let $\mat A$ be bounded, $t$-independent, and satisfy the ellipticity condition~\eqref{eqn:elliptic} in dimension $\dmn=2$. Let $\alpha_j=j\vec e_1+(m-j)\vec e_2$, and let $A_{jk}=A_{\alpha_j\alpha_k}$.
Then $\re A_{mm}(x)\geq \lambda$ for almost every $x\in\R$.

\end{lem}

\begin{proof}
Let $\varphi_R(x,t)=\rho(x)\, \eta(t/R)$, where $\rho$ and $\eta$ are smooth, real-valued, compactly supported, and not everywhere zero.
Then the bound \eqref{eqn:elliptic} implies that \multlinegap=0pt
\begin{multline*}
\lambda\sum_{j=0}^m
\frac{1}{R^{2m-2j}}\int_\R\abs{\rho^{(j)}(x)}^2\,dx
\int_{\R} \abs{\eta^{(m-j)}(t/R)}^2\,dt
\\\leq
\re
\sum_{j=0}^m \sum_{k=0}^m
\frac{1}{R^{2m-j-k}}
\int_{\R}
\rho^{(j)}(x) \,A_{jk}(x)\,\rho^{(k)}(x)\,dx
\int_\R \eta^{(m-j)}(t/R)\,\eta^{(m-k)}(t/R)\,dt.\end{multline*}
Making the change of variables $t\mapsto Rt$, dividing both sides by~$R$, and taking the limit as $R\to\infty$,
we see that
\begin{equation*}
\lambda
\int_\R\abs{\rho^{(m)}(x)}^2\,dx
\int_{\R} \abs{\eta(t)}^2\,dt
\leq
\re
\int_{\R}
A_{mm}(x)\,(\rho^{(m)}(x))^2\,dx
\int_\R (\eta(t))^2\,dt.\end{equation*}
Canceling the integral of $\eta$, we observe that
\begin{equation*}
\lambda
\int_\R(\rho^{(m)}(x))^2\,dx
\leq
\re
\int_{\R}
A_{mm}(x)\,(\rho^{(m)}(x))^2\,dx
\end{equation*}
for all real-valued, smooth, compactly supported functions~$\rho$. This implies that $\re A_{mm}\geq \lambda$ almost everywhere, as desired.
\end{proof}

\begin{proof}[Proof of Proposition~\ref{prp:2d:p0}]
As observed in \cite[Theorem~24]{Bar16}, by the bound \eqref{eqn:Meyers} with $p>2$ and by Morrey's inequality, if $\dmn=2$ and $Lv=0$ in $2S$, then $\nabla^{m-1}v$ is continuous in $\overline S$ and satisfies
\begin{equation*}\max_{(x,t)\in \overline S}
\abs{\nabla^{m-1} v(x,t)}\leq C\left(\fint_{2S} \abs{\nabla^{m-1} v}^2\right)^{1/2}.\end{equation*}
If $L$ is $t$-independent, then $\partial_t^k u\in \dot W^{m,2}_{loc}(2S)$ and $L(\partial_t^k u)=0$ in $2S$ for any nonnegative integer~$k$. Thus, if $Lu=0$ in $2S$ then $\nabla^{m-1}\partial_t u$ is continuous in $\overline S$, and furthermore
\begin{equation}\label{eqn:2d}
\max_{(x,t)\in \overline S}
\abs{\nabla^{m-1}\partial_t u(x,t)}\leq C
\left(\fint_{2S} \abs{\nabla^{m} u}^2\right)^{1/2}.\end{equation}
We need only bound $\partial_x^m u(x,t)$.

By the definition~\eqref{eqn:weak} of $Lu$, if $A_{jk}$ is as in Lemma~\ref{lem:A0:2d}, then
\begin{align*}
0&=
	\int_{2S} \sum_{j=0}^m \sum_{k=0}^m \partial_x^j \partial_t^{m-j}\varphi(x,t)A_{jk}(x) \partial_x^k \partial_t^{m-k} u(x,t)\,dx\,dt
\end{align*}
for all smooth test functions~$\varphi$ that are compactly supported in~$2S$. Because $\partial_t^{2m-j-k} u\in \dot W^{m,2}_{loc}(2S)$, we may integrate by parts in $t$ and see that
\begin{align*}
0&=
	\int_{2S} \sum_{j=0}^m (-1)^{m-j}\sum_{k=0}^m \partial_x^j \varphi(x,t)A_{jk}(x) \partial_x^k \partial_t^{2m-j-k} u(x,t)\,dx\,dt.
\end{align*}

Let $(x_0,t_0)$ be the midpoint of~$S$. Let
\begin{equation*}f_m(x,t) = \sum_{k=0}^{m-1}A_{mk}(x) \partial_x^k \partial_t^{m-k} u(x,t)\end{equation*}
and if $0\leq j\leq m-1$, let $f_j(x)$ satisfy
\begin{equation*}\partial_x^i f_j(x,t)\big\vert_{x=x_0}=0\text{ for all } i<m-j,\quad
\partial_x^{m-j} f_j(x,t)=\sum_{k=0}^m A_{jk}(x) \partial_x^k \partial_t^{2m-j-k} u(x,t).\end{equation*}
Thus,
\begin{align*}0&=\int_{2S}
\sum_{j=0}^{m-1} (-1)^{m-j}\partial_x^j \varphi(x,t)
\partial_x^{m-j} f_j(x,t)
\\&\qquad+
\partial_x^m \varphi(x,t) f_m(x,t)
+
\partial_x^m \varphi(x,t)A_{mm}(x) \partial_x^m u(x,t)
\,dx\,dt
.\end{align*}
Integrating by parts in~$x$, we have that
\begin{align*}0&=\int_{2S}
\partial_x^m \varphi(x,t)\sum_{j=0}^{m}
f_j(x,t)
+
\partial_x^m \varphi(x,t)A_{mm}(x) \partial_x^m u(x,t)
\,dx\,dt
.\end{align*}
Let $Q$ and $I$ be intervals such that $S=Q\times I$.
By the bound~\eqref{eqn:2d} (if $j=m$) or Lemma~\ref{lem:slices} and the Caccioppoli inequality (if $0\leq j\leq m-1$), we have that if $t\in I$ then
\begin{equation*}\esssup_{x\in Q} \abs{f_j(x,t)} \leq C\biggl(\fint_{2S} \abs{\nabla^{m} u}^2\biggr)^{1/2}.\end{equation*}
Furthermore, $t\to f_j(\,\cdot\,,t)$ and $t\to \partial_1^m u(\,\cdot\,,t)$ are both continuous $2I\to L^2(2Q)$. Choosing $\varphi(x,t)=\rho(x)\frac{1}{R}\eta((t-\tau)/R)$ for smooth compactly supported functions $\rho$ and $\eta$ and letting $R\to 0^+$, we have that
\begin{align*}0&=\int_{2Q}
\rho^{(m)}(x)\sum_{j=0}^{m}
f_j(x,\tau)
+
\rho^{(m)}(x)A_{mm}(x) \partial_x^m u(x,\tau)
\,dx\,d\tau
\end{align*}
for any $\tau\in I$. Thus, for any such $\tau$,
\begin{equation*}A_{mm}(x) \partial_x^m u(x,\tau) = \sum_{j=0}^{m}
f_j(x,\tau) + P_\tau(x)\end{equation*}
for almost every $x\in Q$,
where $P_\tau$ is a polynomial of degree at most $m-1$. Because $P_\tau$ is a polynomial in~$x$, we have that
\begin{align*}\sup_{x\in Q} \abs{P_\tau(x)}
&\leq C\fint_Q \abs{P_\tau(x)}\,dx
\leq C\fint_Q \abs{A_{mm}(x) \partial_x^m u(x,t)}\,dx
+C\sum_{j=0}^{m} \fint_Q \abs{
f_j(x,t)}\,dx
\end{align*}
which by Lemma~\ref{lem:slices} is at most $C\fint_{2S} \abs{\nabla^{m} u}$. Applying this bound on~$P_\tau$ and the above bound on $f_j(x,\tau)$ completes the proof.
\end{proof}

The following proposition completes the proof of the bounds~\eqref{eqn:Meyers:bound}. A similar argument was used in \cite[Appendix~B]{AlfAAHK11} to show that, in dimension $\dmn=3$, the De Giorgi-Nash-Moser condition is valid for all operators with $t$-independent coefficients.

\begin{prp}\label{prp:Meyers:t-independent}
Let $L$ be an operator of the form~\eqref{eqn:weak} of order~$2m$ associated to bounded coefficients~$\mat A$ that satisfy the ellipticity condition~\eqref{eqn:elliptic}.

If $\mat A$ is $t$-independent in the sense of formula~\eqref{eqn:t-independent}, then the extended real numbers $p_{j,L}^+$ in the bound~\eqref{eqn:Meyers} satisfy
\begin{equation*}\frac{1}{p_{j,L}^+}\leq \max\biggl(0,\frac{1}{p_{j-1,L}^+}-\frac{1}{n}\biggr) \leq \max\biggl(0,\frac{1}{p_{0,L}^+}-\frac{j}{n}\biggr) \end{equation*}
for all integers $j$ with $1\leq j\leq m$. Furthermore, the numbers $c(j,L,p,q)$ in the bound~\eqref{eqn:Meyers} may be bounded above by a constant that depends only on $p$, $q$, the standard parameters, and on $c(j-1,L,r,2)$, where $1/p+1/n=1/r$.
\end{prp}

An induction argument shows that if $1/p>1/(2+\varepsilon)-j/n$, then $c(j,L,p,q)$ may be bounded by a number depending only on $p$, $q$ and the standard parameters.

\begin{proof}[Proof of Proposition~\ref{prp:Meyers:t-independent}]
Let $Q\subset\R^n$ be a cube and let $I$ be an interval with $\abs{I}=\ell(Q)=\ell$, the  side length of~$Q$. Let $u$ be a solution to $Lu=0$ in $16Q\times 16I$.
Choose some $p$ with
\begin{equation*}\frac{1}{p_{j-1,L}^+}-\frac{1}{n}<\frac{1}{p} \leq\frac{1}{2}-\frac{j}{n}.\end{equation*}
Recall that $1/r=1/p+1/n$, so $r<\min(n,p_{j-1,L}^+)$. We observe that $1/r<1/2+1/n-j/\pdmn\leq 1/2+1/n-1/\pdmn$, and so if $n\geq 1$ then $r>1$.

By the Gagliardo-Nirenberg-Sobolev inequality in $Q\subset\R^n$,
\begin{equation*}
\biggl(\fint_{Q} \abs{\nabla^{m-j} u(x,t)-\arr P_t}^p\,dx\biggr)^{1/p}
\leq C_r\ell \biggl(\fint_{Q} \abs{\nabla^{m-j+1} u(x,t)}^r\,dx\biggr)^{1/r}
\end{equation*}
where $\arr P_t$ is a constant that satisfies $\int_Q \nabla^{m-j}u(x,t)-\arr P_t\,dx=0$.

Observe that
\begin{equation*}\abs{\arr P_t}=\abs[bigg]{\fint_Q \nabla^{m-j}u(x,t)\,dx}.\end{equation*}
By Lemma~\ref{lem:slices} and H\"older's inequality, if $t\in I$ then
\begin{equation*}\abs{\arr P_t}
\leq C\fint_{2I}\fint_{2Q}
\abs{ \nabla^{m-j}u(x,s)}\,dx\,ds
\leq C\left(\fint_{2I}\fint_{2Q}
\abs{ \nabla^{m-j}u(x,s)}^2\,dx\,ds\right)^{1/2}
\end{equation*}
Thus,
\begin{align*}
\biggl(\fint_{Q} \abs{\nabla^{m-j} u(x,t)}^p\,dx\biggr)^{1/p}
&\leq
	C_r\ell \biggl(\fint_{Q} \abs{\nabla^{m-j+1} u(x,t)}^r\,dx\biggr)^{1/r}
	\\&\nonumber\qquad+
	C_r\biggl(\fint_{2I}\fint_{2Q}
	\abs{ \nabla^{m-j}u(x,s)}^2\,dx\,ds\biggr)^{1/2}
.\end{align*}
Recall that $r<p_{j-1,L}^+$. By Lemma~\ref{lem:slices}, if $t\in I$ then
\begin{equation*} \ell\biggl(\fint_{Q} \abs{\nabla^{m-j+1} u(x,t)}^r\,dx\biggr)^{1/r}
\leq
C(j-1,r)  \ell\biggl(\fint_{2I}\fint_{2Q} \abs{\nabla^{m-j+1} u(x,s)}^r\,dx\,ds\biggr)^{1/r}.\end{equation*}
By H\"older's inequality (if $r<2$) or the bound~\eqref{eqn:Meyers} (if $r>2$), we have that
\begin{equation*}  \ell\biggl(\fint_{Q} \abs{\nabla^{m-j+1} u(x,t)}^r\,dx\biggr)^{1/r}
\leq
C(j-1,r)  \ell\biggl(\fint_{4I}\fint_{4Q} \abs{\nabla^{m-j+1} u(x,s)}^2\,dx\,ds\biggr)^{1/2}.\end{equation*}
By the Caccioppoli inequality,
\begin{equation*}\ell \biggl(\fint_{Q} \abs{\nabla^{m-j+1} u(x,t)}^r\,dx\biggr)^{1/r}
\leq
C(j-1,r)\biggl(\fint_{8I}\fint_{8Q} \abs{\nabla^{m-j} u(x,s)}^2\,dx\,ds\biggr)^{1/2}.\end{equation*}
Thus
\begin{align*}
\biggl(\fint_{Q} \abs{\nabla^{m-j} u(x,t)}^p\,dx\biggr)^{1/p}
&\leq
	C(j-1,r)\biggl(\fint_{8I}\fint_{8Q} \abs{\nabla^{m-j} u(x,s)}^2\,dx\,ds\biggr)^{1/2}
.\end{align*}
Taking an average in~$t$, we have that if $1/p_{j-1,L}^+-1/n<1/p<1/2-j/n$, then
\begin{align*}
\biggl(\fint_I\fint_{Q} \abs{\nabla^{m-j} u(x,t)}^p\,dx\,dt\biggr)^{1/p}
&\leq
	C(j-1,r)\biggl(\fint_{8I}\fint_{8Q} \abs{\nabla^{m-j} u(x,s)}^2\,dx\,ds\biggr)^{1/2}
.\end{align*}
By the bound~\eqref{eqn:Meyers} and the following remarks, if $1/2>1/p\geq 1/2-j/n$ this is still true with the constant $C(j-1,r)$ depending only on the standard parameters and on~$p$ (or, equivalently, on~$r$). By H\"older's inequality this is true for $0<p\leq 2$.

By the bound~\eqref{eqn:Meyers} and a covering argument (if $q<2$) or H\"older's inequality (if $q>2$), if $0<q<\infty$ then
\begin{equation*}\biggl(\fint_{8I}\fint_{8Q} \abs{\nabla^{m-j} u(x,s)}^2\,dx\,ds\biggr)^{1/2}
\leq \widetilde C(q) \biggl(\fint_{16I}\fint_{16Q} \abs{\nabla^{m-j} u(x,s)}^q\,dx\,ds\biggr)^{1/q}\end{equation*}
where $\widetilde C(q)$ depends on $q$, $\dmn$ and $c(j,L,2,q)$, and so may be taken depending only on $q$ and the standard parameters. Thus, if $1/p_{j-1,L}^+-1/n<1/p<1/q<\infty$, then
\begin{align*}
\biggl(\fint_I\fint_{Q} \abs{\nabla^{m-j} u}^p\biggr)^{1/p}
&\leq
	C(j-1,r)\,\widetilde C(q)\biggl(\fint_{8I}\fint_{8Q} \abs{\nabla^{m-j} u}^q \biggr)^{1/q}
.\end{align*}
A covering argument shows that the bound~\eqref{eqn:Meyers} is valid with $c(j,L,p,q)=C(j-1,r)\allowbreak \,\widetilde C(q)$, as desired.
\end{proof}

\section{Nontangential estimates on layer potentials}
\label{sec:N:potentials}

In this section we will prove the nontangential bounds~\cref{eqn:D:N:intro,eqn:S:N:intro,eqn:D:N:rough:intro,eqn:S:N:rough:intro} on layer potentials.
This will require extensive preliminaries.

As described in Section~\ref{sec:intro:extrapolation}, in \cite{She06A,She06B}, Shen used good-$\lambda$ inequalities to bound solutions to the Dirichlet problem for constant coefficient systems of second or higher order. The following technical lemma is similar to those used in Shen's work and will be used in the proofs of Lemmas~\ref{lem:N:+} and~\ref{lem:lusin:+} below.
In this lemma we will make use of the following capped maximal-type function. The use of dyadic maximal functions lets us avoid covering arguments.

Let $Q\subset\R^n$ be a cube. If $j\geq 0$ is an integer, let $\mathcal{G}_j(Q)$ be the set of $2^{jn}$ pairwise-disjoint open subcubes of side length $2^{-j}\ell(Q)$ such that $\overline Q=\cup_{R\in \mathcal{G}_j(Q)} \overline{R}$. Let $\mathcal{G}(Q)=\cup_{j=0}^\infty \mathcal{G}_j(Q)$.
If $b>1$ and $F\in L^1(bQ)$, then we define
\begin{equation*}\mathcal{M}_{b,Q}F(x) = \sup
\biggl\{\fint_{bR} \abs{F}:R\in\mathcal{G}(Q),\>R\owns x\biggr\}.\end{equation*}

\begin{lem}\label{lem:lambda}
Let $p_2>2$, $A_0\geq1$, and $C_0>0$ be constants.
Let $Q_0\subset\R^n$ be a cube, let $F\in L^2(8Q_0)$, and let $\Phi\in L^p(16Q_0)$ for some $p$ with $2<p<p_2$.

Suppose that whenever $0<\gamma\leq1$, $A\geq A_0$, and $\lambda > 0$, and whenever $Q\in \mathcal{G}(Q_0)$ is a subcube that satisfies
\begin{equation*}
\fint_{16Q} \abs{\Phi}^2 \leq \gamma\lambda, \quad 8^n\abs{Q}\lambda<\int_{8Q} \abs{F}^2,
\quad \int_{15Q} \abs{F}^2 \leq 16^n\abs{Q}\lambda,\end{equation*}
we have that
\begin{equation*}\abs{\{x\in Q:\mathcal{M}_{8,Q} (\abs{F}^2)(x)>A\lambda\}}
\leq C_0\biggl(\frac{\gamma}{A}+\frac{1}{A^{p_2/2}}\biggr) \abs{Q}
.\end{equation*}

Then there is a number $C_p$ depending only on $A_0$, $C_0$, $p$, $p_2$, and the dimension~$n$, such that
\begin{align*}
\int_{Q_0}\abs{F(x)}^p\,dx
&\leq
\frac{C_p}{\abs{Q_0}^{p/2-1}}\doublebar{F}_{L^2(8Q_0)}^p
+
C_p
\int_{16Q_0}\abs{\Phi}^{p}
.\end{align*}
\end{lem}

\begin{proof} Let
\begin{align*}
E(F,\lambda)=\{x\in Q_0: \mathcal{M}_{8,Q_0} (\abs{F}^2)(x)>\lambda\}
\end{align*}
and observe that by the Lebesgue differentiation theorem and the definition of the Lebesgue integral,
\begin{equation*}\int_{Q_0}\abs{F(x)}^p\,dx \leq \int_{Q_0}\mathcal{M}_{8,Q_0}(\abs{F}^2)(x)^{p/2}\,dx = \int_0^\infty \abs{E(F,\lambda)}\frac{p}{2}\lambda^{p/2-1}\,d\lambda.\end{equation*}
Choose some $\lambda>0$. Let $\mathcal{G}=\mathcal{G}(Q_0)$ be the grid of dyadic subcubes of~$Q_0$ as given above.
For each $x\in E(F,\lambda)$, there is some largest cube $Q$ such that $Q\owns x$, $Q\in\mathcal{G}$, and $\fint_{8Q} \abs{F}^2>\lambda$. Thus,
\begin{equation*}E(F,\lambda)=\bigcup_{Q\in \mathcal{F}(\lambda)} Q \end{equation*}
where $\mathcal{F}(\lambda)$ is the set of all maximal cubes in $\{Q\in\mathcal{G}:\fint_{8Q} \abs{F}^2>\lambda\}$. Observe that the cubes in $\mathcal{F}(\lambda)$ are pairwise-disjoint.
Let
\begin{equation*}\mathcal{H}(\lambda,\gamma)=\biggl\{Q\in \mathcal{F}(\lambda): \fint_{16Q} \abs{\Phi}^{2} \leq \gamma\lambda\biggr\}.\end{equation*}
If $A\geq 1$, then $E(F,A\lambda)\subset E(F,\lambda)$, and so
\begin{align*}
\abs{E(F,A\lambda)}
&=
\sum_{Q\in \mathcal{F}(\lambda)}
\abs{\{x\in Q: \mathcal{M}_{8,Q_0}(F^2)(x)>A\lambda\}}
\\&\leq
	\sum_{Q\in \mathcal{F}(\lambda)\setminus \mathcal{H}(\lambda,\gamma)}
	\abs{Q}
	+
	\sum_{Q\in \mathcal{H}(\lambda,\gamma)}
	\abs{\{x\in Q: \mathcal{M}_{8,Q_0}(F^2)(x)>A\lambda\}}
.\end{align*}
If $Q\in \mathcal{F}(\lambda)\setminus\mathcal{H}(\lambda,\gamma)$, then for all $x\in Q$ we have that
\begin{equation*}\mathcal{M}_{16,Q_0}(\abs{\Phi}^2)(x)\geq \fint_{16Q} \abs{\Phi}^2 > \gamma\lambda.\end{equation*}
If $Q\in \mathcal{F}(\lambda)$, then $\fint_{8R} \abs{F}^2\leq \lambda<\fint_{8Q} \abs{F}^2$ for all $R\in \mathcal{G}$ with $R\supsetneq Q$.
Thus, if $x\in Q$, then $\mathcal{M}_{8,Q_0}(F^2)(x)=\mathcal{M}_{8,Q}(F^2)(x)$.

Thus, if $A\geq 1$, then
\begin{align*}
\abs{E(F,A\lambda)}
&\leq
	\abs{\{x\in Q_0:\mathcal{M}_{16,Q_0}(\Phi^{2})(x)>\gamma\lambda\}}
	\\&\qquad
	+
	\sum_{Q\in \mathcal{H}(\lambda,\gamma)}
	\abs{\{x\in Q: \mathcal{M}_{8,Q}(F^2)(x)>A\lambda\}}
.\end{align*}
We claim that if ${Q\in \mathcal{H}(\lambda,\gamma)}$ and $\lambda$ is large enough, then $Q$ satisfies the conditions of the lemma. By definition of $\mathcal{H}(\lambda,\gamma)$, we have that $\fint_{16Q} \abs{\Phi}^{2} \leq \gamma\lambda$. By definition of $\mathcal{F}(\lambda)$, we have that $8^n\lambda\abs{Q} <
\int_{8Q} \abs{F}^2$. We are left with the upper bound on $\int_{15Q} \abs{F}^2$.

If $Q\in \mathcal{F}(\lambda)$, then
\begin{equation*}8^n\lambda\abs{Q} <
\int_{8Q} \abs{F}^2
\leq \doublebar{F}_{L^2(Q_0)}^2
\end{equation*}
and so $\abs{Q}<\frac{\doublebar{F}_{L^2(Q_0)}^2}{8^n\lambda}$. Let
\begin{equation*}\lambda_0=\frac{\doublebar{F}_{L^2(Q_0)}^2}{8^n\abs{Q_0}}.\end{equation*}
If $\lambda>\lambda_0$, then $\abs{Q}<\abs{Q_0}$ and so $Q\neq Q_0$.
In particular, the dyadic parent $P(Q)$ of $Q$ is an element of $\mathcal{G}$, and so by maximality of~$Q$,
\begin{equation*}\lambda 16^n\abs{Q} =\lambda\abs{8P(Q)}\geq
\int_{8P(Q)} \abs{F}^2
\end{equation*}
and because $15Q\subset 8P(Q)$,
\begin{equation*}\lambda 16^n\abs{Q}
\geq \int_{15Q} \abs{F}^2.
\end{equation*}

Thus, if $\lambda>\lambda_0$ and ${Q\in \mathcal{H}(\lambda,\gamma)}$, then for all $A\geq A_0$, where $A_0$ is as in the statement of the lemma, we have that
\begin{equation*}\abs{\{x\in Q:\mathcal{M}_{8,Q} (\abs{F}^2)(x)>A\lambda\}}
\leq C_0\biggl(\frac{\gamma}{A}+\frac{1}{A^{p_2/2}}\biggr) \abs{Q}
.\end{equation*}
Recalling that $E(F,\lambda)=\sum_{Q\in \mathcal{F}(\lambda)}\abs{Q}\geq \sum_{Q\in \mathcal{H}(\lambda,\gamma)}\abs{Q}$, we have that if $\lambda>\lambda_0$ and $A\geq A_0$, then
\begin{align*}
\abs{E(F,A\lambda)}
&\leq
	\abs{\{x\in Q_0:\mathcal{M}_{16,Q_0}(\abs{\Phi}^{2})(x)>\gamma\lambda\}}
	+
	C_0\biggl(\frac{\gamma}{A}+\frac{1}{A^{p_2/2}}\biggr)
	\abs{E(F,\lambda)}
.\end{align*}
Multiplying both sides by $A^{p/2}(p/2)\lambda^{p/2-1}$ and integrating, we have that if $A\geq A_0$ and $\Lambda>\lambda_0$, then
\begin{align*}
\int_{A\lambda_0}^{A\Lambda} \abs{E(F,\lambda)} \frac{p}{2}\lambda^{p/2-1}\,d\lambda
&\leq
C_0 \biggl(A^{(p-2)/2}\gamma+\frac{1}{A^{(p_2-p)/2}}\biggr)
\int_{\lambda_0}^\Lambda \abs{E(F,\lambda)} \frac{p}{2}\lambda^{p/2-1}\,d\lambda
\\&\qquad+
\biggl(\frac{A}{\gamma}\biggr)^{p/2}
\int_{Q_0}\mathcal{M}_{16,Q_0}(\abs{\Phi}^{2})^{p/2}
.\end{align*}
Applying the
$L^{p/2}$-boundedness of the maximal operator, and using the fact that
$\abs{E(F,\lambda)}\leq \abs{Q_0}$, we have that
\begin{align*}
\int_{0}^{A\Lambda} \abs{E(F,\lambda)} \frac{p}{2}\lambda^{p/2-1}\,d\lambda
&\leq
C_0 \biggl(A^{(p-2)/2}\gamma+\frac{1}{A^{(p_2-p)/2}}\biggr)
\int_{\lambda_0}^\Lambda \abs{E(F,\lambda)} \frac{p}{2}\lambda^{p/2-1}\,d\lambda
\\&\qquad+
\abs{Q_0}(A\lambda_0)^{p/2}
+
C_p\biggl(\frac{A}{\gamma}\biggr)^{p/2}
\int_{16Q_0}\abs{\Phi}^{p}
\end{align*}
for some constant $C_p$ depending only on $n$ and~$p$.
Let
\begin{equation*}A=\max(A_0,(4C_0)^{2/(p_2-p)}),\qquad\gamma=\min\left(\frac{1}{4C_0A^{(p-2)/2}}, 1\right).\end{equation*}
Then
\begin{align*}
\int_{0}^{A\Lambda} \abs{E(F,\lambda)} \frac{p}{2}\lambda^{p/2-1}\,d\lambda
&\leq
C_p\abs{Q_0}\lambda_0^{p/2}
+
C_p
\int_{16Q_0}\abs{\Phi}^{p}
\end{align*}
where $C_p$ depends only on $A_0$, $C_0$, $p$, $p_2$, and~$n$.
Taking the limit as $\Lambda\to \infty$ and recalling the definition of~$\lambda_0$ completes the proof.
\end{proof}

We now consider bounds on nontangential maximal operators.
Recall the definition~\eqref{dfn:NTM:modified:*} of the two-sided modified nontangential maximal function. We define
\begin{align*}
\widetilde N^{\ell}_n \arr u(x) &=
\sup_{-\ell\leq t \leq\ell} \>\sup_{\abs{x-y}<\abs{t}} \biggl(\fint_{B((y,t),\abs{t}/2)}
\abs{\arr u}^2\biggr)^{1/2}
,\\
\widetilde N^{\ell}_f \arr u(x) &=
\sup_{\abs{t}\geq\ell}\> \sup_{\abs{x-y}<\abs{t}}  \biggl(\fint_{B((y,t),\abs{t}/2)} \abs{\arr u}^2\biggr)^{1/2}
.\end{align*}
Observe that
\begin{equation*}\widetilde N_* \arr u(x)=\max(\widetilde N^{\ell}_f \arr u(x), \widetilde N^{\ell}_n \arr u(x)).\end{equation*}
We may thus use bounds on $\widetilde N_n^{\ell}$ and $\widetilde N_f^{\ell}$ to bound~$\widetilde N_*$.

We begin with a very simple bound on $\widetilde N_f^\ell$.
The following lemma is well known; for the sake of completeness we will provide a proof.
\begin{lem}\label{lem:N:far}
If $x\in\R^n$ and $\ell>0$, and if $0<p<\infty$, then
\begin{equation*}
\widetilde N_f^{\ell}\arr u(x)^p
\leq
2^n\fint_{\abs{x-z}<\ell}\widetilde N_f^{\ell}\arr u(z)^p\,dz
\end{equation*}
whenever the right hand side is finite.
\end{lem}
\begin{proof}
Let $x\in\R^n$. Choose some $t$ with $\abs{t}\geq \ell$ and some $y\in\R^n$ with $\abs{x-y}<\abs{t}$.
If $z\in\R^n$ is such that $\abs{z-y}<\abs{t}$, then
\begin{equation*}\biggl(\fint_{B((y,t),\abs{t}/2)} \abs{\arr u}^2\biggr)^{p/2}\leq \widetilde N_f^{\ell}\arr u(z)^p.\end{equation*}
The set of all $z$ with $\abs{y-z}<\abs{t}$ and with $\abs{x-z}<\ell$ contains a disk $\Delta_y$ of radius $\ell/2$, and so
\begin{equation*}\biggl(\fint_{B((y,t),\abs{t}/2)} \abs{\arr u}^2\biggr)^{p/2} \leq
\fint_{\Delta_y} \widetilde N_f^{\ell} \arr u(z)^p\,dz
\leq
2^n\fint_{\abs{x-z}<\ell}\widetilde N_f^{\ell} \arr u(z)^p\,dz
\end{equation*}
whenever $\abs{x-y}<\abs{t}$.
Recalling the definition of $\widetilde N_f^{\ell} \arr u(x)$ completes the proof.
\end{proof}

We now show how local estimates may be used with Lemma~\ref{lem:lambda} to produce global estimates.

\begin{lem}\label{lem:N:+}
Suppose that $\arr u\in L^2_{loc}(\R^\dmn_+\cup\R^\dmn_-)$ is such that $\widetilde N_*\arr u\in L^2(\R^n)$.
Suppose that there exists a number $p_2>2$, a function $\Phi$, and a set of functions $\arr u_Q$ indexed by cubes $Q$ in $\R^n$ such that,
for any cube $Q\in \R^n$,
\begin{align*}
\doublebar{\widetilde N_n^{\ell}\arr u_Q}_{L^2(8Q)}&\leq C_0\doublebar{\Phi}_{L^2(16Q)}
,\\
\doublebar{\widetilde N_n^{\ell}(\arr u-\arr u_Q)}_{L^{p_2}(8Q)} &\leq
\frac{C_0}{\abs{Q}^{1/2-1/p_2}} \doublebar{\Phi}_{L^2(16Q)}
+ \frac{C_0}{\abs{Q}^{1/2-1/p_2}} \doublebar{\widetilde N_*\arr u}_{L^2(15Q)}
\end{align*}
where $\ell=\ell(Q)/4$.

Then, for every $p$ with $2<p<p_2$, there is a number $C_p$ depending only on $p$, $p_2$, $C_0$ and the dimension $n$, such that
\begin{align*}
\doublebar{\widetilde N_*\arr u}_{L^p(\R^n)}
&\leq C_p \doublebar{\Phi}_{L^p(\R^n)}
\end{align*}
whenever the right hand side is finite.
\end{lem}

\begin{proof}
We will use Lemma~\ref{lem:lambda}, with $F=\widetilde N_*\arr u$, and let ${Q_0}\to \R^n$.

Choose some $\gamma$ with $0<\gamma\leq 1$ and some $\lambda>0$. Let $Q\subset Q_0$ be such that
\begin{align}
\label{eqn:Q:Psi}
\fint_{16Q} \abs{\Phi}^2 &\leq \gamma\lambda,
\\
\label{eqn:Q:big}
8^n\abs{Q}\lambda&<\int_{8Q} \abs{\widetilde N_*\arr u}^2,
\\
\label{eqn:Q:small}
\int_{15Q} \abs{\widetilde N_*\arr u}^2 &\leq 16^n\abs{Q}\lambda.\end{align}
We need to show that there is some $C_0$ and $A_0$ independent of $\gamma$ and $\lambda$ such that, if $A\geq A_0$, then
\begin{equation*}\abs{\{x\in Q:\mathcal{M}_{8,Q} ((\widetilde N_*\arr u)^2)(x)>A\lambda\}}
\leq C_0\biggl(\frac{\gamma}{A}+\frac{1}{A^{p_2/2}}\biggr) \abs{Q}
.\end{equation*}

Let $\arr u_n=\arr u_{Q}$, and let $\arr u_f=\arr u-\arr u_n=\arr u-\arr u_{Q}$. We compute
\begin{align*}\widetilde N_*\arr u(x)^2
&\leq
\max\bigl((\widetilde N_n^\ell\arr u_n(x)
+ \widetilde N_n^\ell\arr u_f(x))^2
, \widetilde N_f^\ell\arr u(x)^2\bigr)
\\&\leq
2\widetilde N_n^\ell\arr u_n(x)^2
+ 2\widetilde N_n^\ell\arr u_f(x)^2
+ \widetilde N_f^\ell\arr u(x)^2
\end{align*}
and so
\begin{multline}
\label{eqn:E:split}
\abs{\{x\in Q:\mathcal{M}_{8,Q} ((\widetilde N_*\arr u)^2)(x) > A\lambda\}}
\\\begin{aligned}
&\leq
	\abs{\{x\in Q:\mathcal{M}_{8,Q}((\widetilde N_n^\ell\arr u_n)^2)(x)>A\lambda/5\}}
	\\&\qquad+
	\abs{\{x\in Q:\mathcal{M}_{8,Q}((\widetilde N_n^\ell\arr u_f)^2)(x)>A\lambda/5\}}
	\\&\qquad+
	\abs{\{x\in Q:\mathcal{M}_{8,Q}((\widetilde N_f^\ell\arr u)^2)(x)>A\lambda/5\}}
.\end{aligned}\end{multline}

By assumption on $\arr u_Q$,
\begin{equation*}
\doublebar{\widetilde N_n^\ell\arr u_n}_{L^2(8Q)}^2
\leq C_0\doublebar{\Phi}_{L^2(16Q)}^2.\end{equation*}
By the weak $L^1$ boundedess of $\mathcal{M}$,
\begin{equation*}
\abs{\{x\in Q: \mathcal{M}_{8,Q}((\widetilde N_n^\ell\arr u_n))^2)(x) > A\lambda/5\}}
\leq
\frac{C\doublebar{\Phi}_{L^2(16Q)}^2}{A\lambda}
.\end{equation*}
By the bound~\eqref{eqn:Q:Psi},
\begin{equation}
\label{eqn:E:bound:1}
\abs{\{x\in Q: \mathcal{M}_{8,Q}((\widetilde N_n^\ell\arr u_n)^2)(x) > A\lambda/5\}}
\leq
C\frac{\gamma}{A}\abs{Q}
.\end{equation}

By assumption on $\arr u_f=\arr u-\arr u_Q$ and by the bounds \eqref{eqn:Q:Psi} and~\eqref{eqn:Q:small},
\begin{equation*}
\doublebar{\widetilde N_n^{\ell}\arr u_f}_{L^{p_2}(8Q)}
\leq
\frac{C_0}{\abs{Q}^{1/2-1/p_2}}
\bigl(\doublebar{\Phi}_{L^2(16Q)}
+  \doublebar{\widetilde N_*\arr u}_{L^2(15Q)}\bigr)
\leq
C\abs{Q}^{1/p_2}\sqrt{\lambda}
.\end{equation*}
By boundedness of $\mathcal{M}$ on $L^{p_2/2}$,
\begin{equation*}\doublebar{\mathcal{M}_{8,Q}((\widetilde N_n^\ell\arr u_f)^2)}_{L^{p_2/2}(Q)}
\leq C\abs{Q}^{2/p_2}\lambda
.\end{equation*}
Thus,
\begin{align}
\label{eqn:E:bound:2}
\abs{\{x\in Q:\mathcal{M}_{8,Q}((\widetilde N_n^\ell\arr u_f)^2)>A\lambda/5\}}
&\leq
C\frac{1}{A^{p_2/2}}\abs{Q}.
\end{align}

By Lemma~\ref{lem:N:far},
\begin{equation*}
\widetilde N_f^\ell\arr u(x)^2
\leq
2^n\fint_{\abs{x-z}<\ell}\widetilde N_f^\ell\arr u(z)^2\,dz
.\end{equation*}
If $x\in Q$ then the region of integration is contained in $(3/2)Q\subset 15Q$. Thus, by the bound~\eqref{eqn:Q:small} and the definition of~$\ell$,
\begin{align*}
\sup_{x\in Q}
\widetilde N_f^\ell\arr u(x)^2
&\leq
\frac{2^n}{\omega_n \ell^n}\int_{15Q}(\widetilde N_f^\ell\arr u)^2
\leq
\frac{32^n}{\omega_n \ell^n}\lambda \abs{Q}
=
\frac{128^n}{\omega_n }\lambda
\end{align*}
where $\omega_n$ is the area of the unit disk in $\R^n$. We let $A_0=5\frac{128^n}{\omega_n}$; then if $A>A_0$, then
\begin{equation}
\label{eqn:E:bound:3}
\{x\in Q:\mathcal{M}_{8,Q}(\widetilde N_f^\ell\arr u^2)(x)>A\lambda/5\}=\emptyset
.\end{equation}

By the bounds~\cref{eqn:E:split,eqn:E:bound:1,eqn:E:bound:2,eqn:E:bound:3}, the conditions of Lemma~\ref{lem:lambda} are satisfied, and so the proof of Lemma~\ref{lem:N:+} is complete.
\end{proof}

The next lemma allows us to apply Lemma~\ref{lem:N:+} in the case where $\arr u-\arr u_Q=\nabla^{m-j} v$, where $Lv=0$ in $8Q\times (-\ell(Q),\ell(Q))$.

\begin{lem}
\label{lem:N:solution}
Let $L$ be an operator of the form~\eqref{eqn:weak} of order~$2m$ associated to bounded $t$-independent coefficients~$\mat A$ that satisfy the ellipticity condition~\eqref{eqn:elliptic}.

Let $Q$ be a cube and let $\ell=\ell(Q)/4$.
Let $0\leq j\leq m$ and let $2<p_2<p_{j,L}^+$, where $p_{j,L}^+$ is as in the bound~\eqref{eqn:Meyers}. Let $v\in \dot W^{m,2}_{loc}(10Q\times(-3\ell,3\ell))$ and suppose that $Lv=0$ in $10Q\times(-3\ell,3\ell)$. Then
\begin{equation*}
\biggl(\fint_{8Q} \widetilde N^\ell_n(\nabla^{m-j}v)^{p_2}\biggr)^{1/{p_2}}
\leq C(j,p_2)\biggl(\fint_{10Q}\widetilde N_n^{3\ell} (\nabla^{m-j}v)^2\biggr)^{1/2}
.\end{equation*}
\end{lem}

\begin{proof}
Let $x\in 8Q$, let $-\ell<{t}<\ell$ with $t\neq 0$, and let $\abs{x-y}<\abs{t}$. We wish to bound the quantity
\begin{equation*}\biggl(\fint_{B((y,t),\abs{t}/2)} \abs{\nabla^{m-j}v(z,r)}^2\,dz\,dr\biggr)^{1/2}
\end{equation*}
with a bound depending only on $x$ and~$v$, not $y$ or~$t$.

Using the same argument as in the proof of \cite[Lem\-ma~3.19]{BarHM18p}, we have that
\begin{align*}\fint_{B((y,t),\abs{t}/2)} \abs{\nabla^{m-j} v(z,r)}^2\,dr\,dz
&\leq
C \biggl( \fint_{\abs{z-y}<\abs{t}} \int_{-2\abs{t}}^{2\abs{t}} \abs{\partial_r^{m-j+1} v(z,r)}\,dr\,dz\biggr)^2
\\&\qquad+C \biggl(\fint_{\abs{z-y}<\abs{t}} \abs{\Tr_{m-j} v(z)}\,dz\biggr)^2
.\end{align*}
Thus, if $x\in 8Q$, then
\begin{equation*}\widetilde N^\ell_n(\nabla^{m-j}v)(x)
\leq C\mathcal{M} U(x)+ C\mathcal{M} (\1_{9Q}\Tr_{m-j}v)(x)\end{equation*}
where $\mathcal{M}$ is the Hardy-Littlewood maximal operator and where
\begin{equation*}U(z)=\1_{9Q} V(z),\quad V(z)=\int_{-2\ell}^{2\ell} \abs{\partial_r^{m-j+1} v(z,r)}\,dr.\end{equation*}
By Lemma~\ref{lem:slices} and the bound~\eqref{eqn:Meyers}, if ${p_2}<p_{j,L}^+$ then
\begin{align*}\int_{\R^n}\abs{\1_{9Q}\Tr_{m-j}v(z)}^{p_2}\,dz
&\leq
C(j,p_2)\abs{Q}\biggl(\fint_{10Q} \fint_{-\ell}^{\ell}\abs{\nabla^{m-j}v(z,r)}^2\,dr\,dz\biggr)^{p_2/2}
.\end{align*}
By H\"older's inequality,
\begin{equation*}\int_{\R^n} U(z)^{p_2}\,dz \leq (4\ell)^{p_2-1}\int_{9Q} \int_{-2\ell}^{2\ell} \abs{\partial_r^{m-j+1} v(z,r)}^{p_2}\,dr\,dz.\end{equation*}
Because $\partial_\dmn^{m-j+1}  v$ is a solution, we may apply the
bound~\eqref{eqn:Meyers} provided $p_2<p_{m,L}^+$; by this bound and the Caccioppoli inequality,
\begin{equation*}\int_{\R^n} U(z)^{p_2}\,dz \leq C(m,p_2)\abs{Q}
\biggl(\fint_{10Q} \fint_{-3\ell}^{3\ell} \abs{\partial_r^{m-j} v(z,r)}^2\,dr\,dz\biggr)^{p_2/2}.\end{equation*}
By the $L^{p_2}$-boundedness of the maximal operator, we have that
\begin{equation*}\int_{8Q} \widetilde N_n^\ell(\nabla^{m-j} v)(x)^{p_2}\,dx
\leq
C(j,p_2)\abs{Q}
\biggl(\fint_{10Q} \fint_{-3\ell}^{3\ell} \abs{\partial_r^{m-j} v(z,r)}^2\,dr\,dz\biggr)^{p_2/2}.\end{equation*}
It is straightforward to control the right hand side by ${\widetilde N_n^{3\ell} (\nabla^{m-j}v)}$.
This completes the proof.
\end{proof}

We now establish the bounds \cref{eqn:S:N:rough:intro,eqn:D:N:rough:intro,eqn:S:N:intro,eqn:D:N:intro}.

\begin{thm}\label{thm:potentials:N:+}
Let $L$ be an operator of the form~\eqref{eqn:weak} of order~$2m$ associated to bounded $t$-independent coefficients~$\mat A$ that satisfy the ellipticity condition~\eqref{eqn:elliptic}.
Let $p_{j,L}^+$ be as in the bound~\eqref{eqn:Meyers}.
If $2<p<p_{1,L}^+$, if $\arr h\in L^p(\R^n)\cap L^2(\R^n)$, and if $\arr f=\Tr_{m-1} F$ for some smooth, compactly supported function~$F$, then
\begin{align}
\doublebar{\widetilde N_*(\nabla^{m-1}\s^L_\nabla \arr h)}_{L^p(\R^n)}
&\leq C(1,p) \doublebar{\arr h}_{L^p(\R^n)}
,\\
\label{eqn:D:N:rough:+}
\doublebar{\widetilde N_*(\nabla^{m-1}\D^{\mat A} \arr f)}_{L^p(\R^n)}
&\leq C(1,p) \doublebar{\arr f}_{\dot W\!A_{m-1}^{0,p}(\R^n)}
.\end{align}
Let $j=0$ or $j=1$ and let $2<p<p_{j,L}^+$. Let $\arr g\in L^p(\R^n)\cap L^2(\R^n)$ and let $\arr\varphi=\Tr_{m-1}\phi$ for some smooth, compactly supported function~$\phi$. Then
\begin{align}
\label{eqn:S:N:+}
\doublebar{\widetilde N_*(\nabla^{m-j}\partial_t^j\s^L \arr g)}_{L^p(\R^n)}
&\leq C(j,p) \doublebar{\arr g}_{L^p(\R^n)}
,\\
\doublebar{\widetilde N_*(\nabla^{m-j}\partial_t^j\D^{\mat A} \arr \varphi)}_{L^p(\R^n)}
&\leq C(j,p) \doublebar{\arr \varphi}_{\dot W\!A_{m-1}^{1,p}(\R^n)}
.\end{align}
\end{thm}

We may as usual extend layer potentials by density in such a way that the given bounds are valid for all inputs such that the respective right hand sides are finite.

\begin{proof}[Proof of Theorem~\ref{thm:potentials:N:+}]
We make one of the following six choices of $u$, $u_Q$, $\Phi_1$, and~$j$, where $Q$ is any cube in $\R^n$.
\begin{align*}
u&=\s^L_\nabla \arr h,& u_Q&=\s^L_\nabla(\1_{16Q}\arr h), & \Phi_1&=\abs{\arr h}
, & j&=1
,\\
u&=\D^{\mat A}\arr f,& u_Q &= \D^{\mat A}(\arr f_{10Q}), &  \Phi_1&=\abs{\arr f}
, & j&=1
,\\
u&=\partial_t^j\s^L \arr g, & u_Q&=\partial_t^j\s^L(\1_{16Q}\arr g), &  \Phi_1&=\abs{\arr g}
, & j&\in \{0,1\}
,\\
u&=\partial_t^j\D^{\mat A}\arr \varphi, & u_Q &= \partial_t^j\D^{\mat A}(\arr \varphi_{10Q}), &  \Phi_1&=\abs{\nabla\arr \varphi}
, & j&\in \{0,1\}
.\end{align*}
Here, $\arr\varphi_{10Q}=\Tr_{m-1}((\phi-P_{11Q})\eta_{10Q}+P_{11Q})$, where $\eta_{10Q}$ is smooth, supported in $11Q$ and identically equal to 1 in~$10Q$, and where $P_{11Q}$ is an appropriate polynomial of degree at most~$m-1$. Then $\arr\varphi_{10Q}=\arr\varphi$ in~$10Q$ and $\nabla\arr\varphi_{10Q}=0$ outside~$11Q$.
As is standard (see, for example, \cite[Lemma~3.3]{She06B} or \cite[Definition~5.9]{BarHM18p}), if we choose $P_{11Q}$ correctly, then by the Poincar\'e inequality, $\doublebar{\nabla\arr \varphi_{10Q}}_{L^q(\R^n)}=\doublebar{\nabla\arr \varphi_{10Q}}_{L^q(11Q)}\leq C_q\doublebar{\nabla\arr \varphi}_{L^q(11Q)}$ for any $q$ with $1\leq q\leq \infty$.

We construct $\arr f_{10Q}$ similarly, but with a polynomial of degree at most~$m-2$, so that
$\arr f_{10Q}=\arr f$ in~$10Q$, $\arr f_{10Q}=0$ outside~$11Q$, and $\doublebar{\arr f_{10Q}}_{L^q(\R^n)}=\doublebar{\arr f_{10Q}}_{L^q(11Q)}\leq C_q\doublebar{\arr f}_{L^q(11Q)}$ for any $1\leq q\leq \infty$.

(In this section, we need only require that $\doublebar{\arr f_{10Q}}_{L^q(\R^n)}\leq C_q\doublebar{\arr f}_{L^q(16Q)}$; we will need the tighter bound $\doublebar{\arr f_{10Q}}_{L^q(\R^n)}\leq C_q\doublebar{\arr f}_{L^q(11Q)}$ in Section~\ref{sec:N:neumann}.)

Choose some $p<p_{j,L}^+$. By standard self-improvement properties of reverse H\"older estimates (see, for example, \cite[Chapter~V, Theorem~1.2]{Gia83}), there is a $p_2>p$ such that the bound~\eqref{eqn:Meyers} is valid for $p=p_2$ (that is, such that $p_2<p_{j,L}^+$), with $p_2$ and $c(j,L,p_2,2)$ depending only on $p$ and $c(j,L,p,2)$.

In all cases, $u-u_Q\in\dot W^{m,2}_{loc}(10Q\times(-\ell(Q),\ell(Q)))$, and $L(u-u_Q)=0$ in $10Q\times (-\ell(Q),\ell(Q))$.
By Lemma~\ref{lem:N:solution} with $v=u-u_Q$, we have that
\begin{align*}
\biggl(\fint_{8Q} \widetilde N^\ell_n(\nabla^{m-j}(u-u_Q))^{p_2}\biggr)^{1/{p_2}}
&\leq
	C(j,p_2)\biggl(\fint_{10Q}\widetilde N_n^{3\ell} (\nabla^{m-j}(u-u_Q))^2\biggr)^{1/2}
.\end{align*}
Let $\arr u=\nabla^{m-j} u$ and let $\arr u_Q=\nabla^{m-j} u_Q$.
By
the bounds~\cref{eqn:S:N:2,eqn:D:N:2,eqn:S:N:rough:2,eqn:D:N:rough:2} established in \cite{BarHM18p}, we have that
\begin{equation*}\doublebar{\widetilde N_*(\arr u)}_{L^2(\R^n)}\leq C\doublebar{\Phi_1}_{L^2(\R^n)}<\infty,\qquad\doublebar{\widetilde N_*(\arr u_Q)}_{L^2(\R^n)}\leq C\doublebar{\Phi_1}_{L^2(16Q)}\end{equation*}
and so
\begin{align*}
\doublebar{\widetilde N^\ell_n(\arr u-\arr u_Q)}_{L^{p_2}(8Q)}
&\leq
	\frac{C(j,p_2)}{\abs{Q}^{1/2-1/p_2}}
	\doublebar{\widetilde N_n^{3\ell} (\arr u)}_{L^2(10Q)}
	+
	\frac{C(j,p_2)}{\abs{Q}^{1/2-1/p_2}} \doublebar{\Phi_1}_{L^2(16Q)}
.\end{align*}
The conditions of Lemma~\ref{lem:N:+} are thus satisfied, and so
\begin{equation*}\doublebar{\widetilde N_+(\nabla^{m-j} u)}_{L^p(\R^n)}
\leq C(j,p) \doublebar{\Phi_1}_{L^p(\R^n)},\end{equation*}
as desired.
\end{proof}

\section{Nontangential estimates for the Neumann problem}
\label{sec:N:neumann}

In this section we will prove part of Theorem~\ref{thm:neumann:p:rough}. Our precise result is Theorem~\ref{thm:neumann:N}, stated below. However, we will need extensive preliminaries before Theorem~\ref{thm:neumann:N} can be proven.

Our analysis will involve the interplay between the Dirichlet and Neumann boundary values of solutions. Thus, we will want a straightforward estimate on the Dirichlet boundary values of various solutions.
The following lemma is much weaker than traditional Fatou-type theorems, as the existence of boundary traces is assumed, not proven; however, the lemma is straightforward to establish and will be useful in our proof of Theorem~\ref{thm:neumann:p:rough}.

\begin{lem}\label{lem:trace:N}
Let $\gamma$ be a (possibly zero) multiindex in $(\N_0)^n$.
Let $\arr u$ be defined and locally integrable in $\R^\dmn_+$. Suppose that the weak derivative $\partial^\gamma u$ exists in $\R^\dmn_+$ and that $\widetilde N_+(\partial^\gamma\arr u)\in L^p(\R^n)$ for some $p$ with $1<p\leq\infty$.
Suppose that $\Trace^+ \arr u$ exists in the sense of formula~\eqref{eqn:trace}; that is, there is an array of functions $\Trace^+\arr u$ such that
\begin{equation*}\lim_{t\to 0^+} \int_K \abs{\arr u(x,t)-\Trace^+\arr u(x)}\,dx=0\end{equation*}
for any compact set $K\subset\R^n$.
Then $\partial^\gamma\Trace^+ \arr u$ exists in the weak sense and satisfies
\begin{equation*}\doublebar{\partial^\gamma\Trace^+\arr u}_{L^p(\R^n)}\leq \doublebar{\widetilde N_+(\partial^\gamma\arr u)}_{L^p(\R^n)}.\end{equation*}
\end{lem}
\begin{proof}
Let $\arr\eta$ be smooth and supported in a compact set $K\subset\R^n$. By the weak definition of derivative, we seek to bound $\langle \partial^\gamma \arr \eta,\Trace^+ \arr u\rangle_{\R^n}$.

Let $W(y,t)=B((y,t),t/2)$ be the Whitney balls in the definition of~$\widetilde N_+$. For each $y\in\R^n$ and each $s$, $t\in \R$, let $E_t(y,s)$ be the horizontal cross section of $W(y,t)$ given by $E_t(y,s)=\{x\in\R^n:(x,s)\in W(y,t)\}$. Observe that $E_t(y,s)=\emptyset$ if $s\geq 3t/2$ or $s\leq t/2$, and that $E_t(y,s)$ is a disk in $\R^n$ centered at~$y$ if $t/2<s<3t/2$, with radius depending only on $t$ and~$s$, not on~$y$.

Let $e_t(s)=\abs{E_t(y,s)}$ and $w(t)=\abs{W(y,t)}$, where $e_t(s)$ is defined in terms of $n$-dimensional Lebesgue measure and $w(t)$ is defined in terms of $\dmn$-dimensional Lebesgue measure. Observe that the right hand sides are independent of $y$ provided $y\in\R^n$. Furthermore, $\int_\R e_t(s)\,ds = \int_{t/2}^{3t/2} e_t(s)\,ds=w(t)$. Thus, for any $t>0$,
\begin{align*}
\langle \partial^\gamma\arr\eta,\Trace^+\arr u\rangle_{\R^n}
&=\int_K \langle \partial^\gamma\arr\eta,\Trace^+\arr u\rangle
=\int_K \langle \partial^\gamma\arr\eta(y),\Trace^+\arr u(y)\rangle \int_{t/2}^{3t/2}\frac{e_t(s)}{w(t)}\,ds\,dy
.\end{align*}
Changing the order of integration and adding and subtracting appropriate values of $\arr u$, we see that
\begin{align*}
\langle \partial^\gamma\arr\eta,\Trace^+\arr u\rangle_{\R^n}
&=
	\int_{t/2}^{3t/2}\frac{e_t(s)}{w(t)}\int_K \langle \partial^\gamma\arr\eta(y),\Trace^+\arr u(y)-\arr u(y,s)\rangle \,dy\,ds
	\\&\qquad+
	\int_{t/2}^{3t/2}\frac{e_t(s)}{w(t)}\int_K \langle \partial^\gamma\arr\eta(y), \arr u(y,s)\rangle \,dy\,ds
.\end{align*}
Because $\frac{e_t}{w(t)}$ is nonnegative and integrates to~$1$, the first integral is at most
\begin{equation*}\sup_{s<3t/2} \doublebar{\partial^\gamma\arr\eta}_{L^\infty(\R^n)} \int_{K} \abs{\Trace^+ \arr u(y)-\arr u(y,s)}\,dy
\end{equation*}
which by definition of $\Trace^+$ converges to zero as $t\to 0^+$.
Thus,
\begin{align*}
\langle \partial^\gamma\arr\eta,\Trace^+\arr u\rangle_{\R^n}
&=
	\lim_{t\to 0^+}
	(-1)^{\abs\gamma}\int_{t/2}^{3t/2}\frac{e_t(s)}{w(t)}\int_K \langle \arr\eta(y), \partial^\gamma\arr u(y,s)\rangle \,dy\,ds
.\end{align*}
Recall that $e_t(s)=\abs{E_t(y,s)}$. Thus,
\begin{align*}
\langle \arr\eta,\partial^\gamma\Trace^+\arr u\rangle_{\R^n}
&=
	\lim_{t\to 0^+}
	(-1)^{\abs\gamma} \int_{t/2}^{3t/2}\frac{1}{w(t)}\int_{E_t(y,s)}\,dx\int_{\R^n} \langle \arr\eta(y), \partial^\gamma\arr u(y,s)\rangle \,dy\,ds
.\end{align*}
Observe that $x\in E_t(y,s)$ if and only if $y\in E_t(x,s)$.
Changing the order of integration and recalling the definitions of $E_t$ and $w(t)$, we see that
\begin{align*}
\langle \partial^\gamma\arr\eta,\Trace^+\arr u\rangle_{\R^n}
&=
	\lim_{t\to 0^+}(-1)^{\abs\gamma}\int_{\R^n}
	\fint_{W(x,t)}\langle \arr\eta(y), \partial^\gamma\arr u(y,s)\rangle \,dy\,ds\,dx
.\end{align*}
Let $K_t=\{x\in\R^n: \arr\eta(y)\neq \arr 0 \text{ for some }(y,s)\in W(x,t)\}$. Because $\arr\eta$ is compactly supported, $\cup_{0<t<1}K_t$ is bounded.
Adding and subtracting $\arr \eta(x)$, we have that
\begin{align*}
\langle \partial^\gamma\arr\eta,\Trace^+\arr u\rangle_{\R^n}
&=
	\lim_{t\to 0^+}\biggl((-1)^{\abs\gamma}\int_{K_t}
	\fint_{W(x,t)}\langle \arr\eta(y)-\arr\eta(x), \partial^\gamma\arr u(y,s)\rangle \,dy\,ds\,dx
	\\&\qquad\phantom{\lim_{t\to0^+}}+(-1)^{\abs\gamma}\int_{K}
	\fint_{W(x,t)}\langle \arr\eta(x), \partial^\gamma\arr u(y,s)\rangle \,dy\,ds\,dx\biggr)
.\end{align*}
If $y\in W(x,t)$ then $\abs{\arr\eta(y)-\arr \eta(x)} \leq \frac{1}{2}t\doublebar{\nabla\arr\eta}_{L^\infty(\R^n)}$. Furthermore, $\fint_{W(x,t)} \abs{\partial^\gamma \arr u} \leq \widetilde N_+(\partial^\gamma \arr u)(x)$.
 Thus, the first integral is at most
\begin{equation*}
\int_{K_t}
	\frac{1}{2}t\doublebar{\nabla\arr\eta}_{L^\infty(\R^n)}
	\widetilde N_+(\partial^\gamma \arr u)(x)
	\,dx
\end{equation*}
which converges to zero as $t\to 0^+$. Thus,
\begin{align*}
\langle \partial^\gamma\arr\eta,\Trace^+\arr u\rangle_{\R^n}
&=
	\lim_{t\to 0^+}(-1)^{\abs\gamma}\int_{K}
	\fint_{W(x,t)}\langle \arr\eta(x), \partial^\gamma\arr u(y,s)\rangle \,dy\,ds\,dx
\end{align*}
and so
\begin{align*}
\abs{\langle \partial^\gamma\arr\eta,\Trace^+\arr u\rangle_{\R^n} }
&\leq
	\int_{\R^n} \abs{\arr \eta(x)} \widetilde N_+(\partial^\gamma \arr u)(x)\,dx
\end{align*}
for all smooth, compactly supported functions~$\arr\eta$.
By density of such functions, $\arr\eta\mapsto \langle \partial^\gamma\arr\eta,\Trace^+\arr u\rangle_{\R^n}$ extends to a bounded linear operator on all of~$L^{p'}(\R^n)$; thus, $\partial^\gamma\Trace^+\arr u$ exists in the weak sense and lies in $L^p(\R^n)$, as desired.
\end{proof}

We now recall a few properties of layer potentials. Specifically, the jump and continuity relations
\begin{align}
\label{eqn:D:jump}
\Tr_{m-1}^+\D^{\mat A}\arr f-\Tr_{m-1}^-\D^{\mat A} \arr f&=-\arr f
,\\
\label{eqn:D:cts}\M_{\mat A}^+\D^{\mat A}\arr f+\M_{\mat A}^-\D^{\mat A} \arr f&\owns\arr 0
,\\
\label{eqn:S:cts}
\Tr_{m-1}^+\s^L\arr g-\Tr_{m-1}^-\s^L \arr g&=\arr 0
,\\
\label{eqn:S:jump}\M_{\mat A}^+\s^L\arr g+\M_{\mat A}^-\s^L \arr g&\owns \arr g
\end{align}
are valid for all $\arr f\in \dot W\!A^{1/2,2}_{m-1}(\R^n)$ and all $\arr g\in \dot B^{-1/2,2}_2(\R^n)$.
See \cite{Bar17}; these relations are also clear from the definitions~\eqref{dfn:D:newton:+} and~\eqref{dfn:S} of $\D^{\mat A}$ and~$\s^L$ and formula~\eqref{eqn:Neumann:intro} for Neumann boundary values. (As noted in Section~\ref{sec:dfn:L}, if $Lu=0$ in $\R^\dmn_\pm$ and $u\in\dot W^{m,2}(\R^\dmn_\pm)$, then by \cite[Lemma~2.4]{BarHM17pB} we have that formula~\eqref{eqn:Neumann:intro} for $\M_{\mat A}^\pm u$ is valid for all $\varphi\in \dot W^{m,2}(\R^\dmn_\pm)$.)

Recall that we extend $\D^{\mat A}$ and $\s^L$ by density. If $\mat A$ is as in Theorem~\ref{thm:potentials:N:+}, then by Lemma~\ref{lem:trace:N} we have that the jump relation~\eqref{eqn:D:jump} is valid for all $\arr f\in \dot W\!A^{0,p}_{m-1}(\R^n)$, $2\leq p<p_{1,L}^+$, or for all $\arr f\in \dot W\!A^{1,p}_{m-1}(\R^n)$, $2\leq p<p_{0,L}^+$. By the bounds~\cref{eqn:D:lusin:2,eqn:D:N:2,eqn:D:lusin:rough:2} and \cite[Theorems~6.1 and~6.2]{BarHM17pB}, we have that the continuity relation~\eqref{eqn:D:cts} is valid for all $\arr f \in\dot W\!A^{0,2}_{m-1}(\R^n)$ and all $\arr f \in\dot W\!A^{1,2}_{m-1}(\R^n)$. Similarly, the jump and continuity relations~\eqref{eqn:S:cts} and~\eqref{eqn:S:jump} are valid for all $\arr g\in L^2(\R^n)$ and all $\arr g\in \dot W^{-1,2}(\R^n)$.

We now establish a compatibility  result of the type discussed in \cite{Axe10}.

\begin{lem}\label{lem:compatible}
Let $L$ be an operator of the form~\eqref{eqn:weak} of order~$2m$ associated to bounded self-adjoint $t$-independent coefficients~$\mat A$ that satisfy the ellipticity condition~\eqref{eqn:elliptic:slices}.
Let $\arr g\in L^2(\R^n)\cap \dot W^{-1,2}(\R^n)$.
Then there is a function $v$, unique up to adding polynomials of degree at most $m-2$, that satisfies
\begin{equation}
\label{eqn:neumann:both:2}
\left\{\begin{gathered}\begin{aligned}
Lv&=0 \text{ in }\R^\dmn_+
,\\
\M_{\mat A}^+ v &\owns \arr g,
\end{aligned}\\\begin{aligned}
\doublebar{\mathcal{A}_2^+(t\nabla^m v)}_{L^2(\R^n)}
+
\doublebar{\widetilde N_+(\nabla^{m-1} v)}_{L^2(\R^n)}
&\leq C\doublebar{\arr g}_{\dot W^{-1,2}(\R^\dmnMinusOne)}
,\\
\doublebar{\mathcal{A}_2^+(t\nabla^m \partial_t v)}_{L^2(\R^n)}
+
\doublebar{\widetilde N_+(\nabla^{m} v)}_{L^2(\R^n)}
&\leq C\doublebar{\arr g}_{L^2(\R^\dmnMinusOne)}
.\end{aligned}\end{gathered}\right.\end{equation}
Furthermore, $v$ is also the solution to the problem~\eqref{eqn:neumann:rough:2}, and if $w$ is the solution to the problem~\eqref{eqn:neumann:regular:2}, then $w=v+P$ for some polynomial $P$ of degree $m-1$.
\end{lem}

\begin{proof} We will follow the argument of \cite{BarHM18}, in which well posedness of the problems~\eqref{eqn:neumann:regular:2} and~\eqref{eqn:neumann:rough:2} were established.

We begin with the case where $\mat A=\mat A_0$ is the (constant) coefficient matrix discussed in \cite[Section~6]{BarHM18}. In this case the solution $w_0$ to the problem~\eqref{eqn:neumann:regular:2} with $\mat A=\mat A_0$ is given by
\begin{equation*}\widehat w_0(\xi,t) = \sum_{k=1}^{m} f_k(\xi) \exp(2\pi i \abs{\xi} e^{\pi i k/(m+1)} t)\end{equation*}
where $\widehat w_0$ denotes the Fourier transform of $w_0$ in the $x$ variables alone, and where
\begin{equation*}f_k(\xi) =
\sum_{\abs\gamma=m-1}
M_{k\gamma}\,\widehat{g_\gamma}(\xi) \frac{\xi^{\gamma_\pureH}}{\abs\xi^{m+\abs{\gamma_\pureH}}}\end{equation*}
for some constants $M_{k\gamma}$.\footnote{There is a minor error in \cite[Section~6]{BarHM18}, namely a forgotten complex conjugate.} Here $\gamma_\pureH=(\gamma_1,\dots,\gamma_n)$.
A straightforward computation and Plancherel's theorem yields that
\begin{align*}\doublebar{\mathcal{A}_2^+(t\nabla^m w_0)}_{L^2(\R^n)}^2
+ \sup_{t>0}\doublebar{\nabla^{m-1}w_0(\,\cdot\,,t)}_{L^2(\R^n)}^2
&\leq C\sum_{k=1}^{m} \int_{\R^n} \abs{\xi}^{2m-2} \abs{f_k(\xi)}^2\,d\xi
.\end{align*}
But
\begin{equation*}\abs{f_k(\xi)} \leq C \abs{\xi}^{-m}\abs{\widehat{\arr g}(\xi)}\end{equation*}
and so
\begin{align*}
\sum_{k=1}^{m} \int_{\R^n} \abs{\xi}^{2m-2} \abs{f_k(\xi)}^2\,d\xi
&\leq C\int_{\R^n} \abs{\xi}^{-2} \abs{\widehat{\arr g}(\xi)}^2\,d\xi
=C \doublebar{\arr g}_{\dot W^{-1,2}(\R^n)}
.\end{align*}
Thus,
\begin{equation*}\doublebar{\mathcal{A}_2^+(t\nabla^m w_0)}_{L^2(\R^n)}^2
+ \sup_{t>0}\doublebar{\nabla^{m-1}w_0(\,\cdot\,,t)}_{L^2(\R^n)}^2
\leq C \doublebar{\arr g}_{\dot W^{-1,2}(\R^n)}
.\end{equation*}
Thus, $v=w_0$ solves the problem~\eqref{eqn:neumann:rough:2} as well as the problem~\eqref{eqn:neumann:regular:2}. By the bounds~\eqref{eqn:neumann:N:2}, the nontangential estimates in the problem~\eqref{eqn:neumann:both:2} are valid. Thus, solutions to the problem~\eqref{eqn:neumann:both:2} exist; uniqueness of solutions to the problem~\eqref{eqn:neumann:rough:2} implies uniqueness of solutions (and thus well posedness) of the problem~\eqref{eqn:neumann:both:2}.

The change of variables $(x,t)\to (x,-t)$ shows that a problem analogous to the problem~\eqref{eqn:neumann:both:2} in the lower half space is well posed.

We now apply the method of layer potentials of \cite{Ver84,BarM13,BarM16A,Bar17}, as in \cite[Section~7.2]{BarHM18}. Specifically, we will use \cite[Theorems~6.23 and~6.24]{Bar17}. We need to verify Conditions~6.14--6.22 in \cite{Bar17}. Let the space $\XX^+$ be given by
\begin{align*}\doublebar{v}_{\XX^+}
&= \doublebar{\mathcal{A}_2^+(t\nabla^m v)}_{L^2(\R^n)}
+ \doublebar{\widetilde N_+(\nabla^{m-1}v)}_{L^2(\R^n)}
\\&\qquad
+ \doublebar{\mathcal{A}_2^+(t\nabla^m \partial_t v)}_{L^2(\R^n)}
+ \doublebar{\widetilde N_+(\nabla^{m}v)}_{L^2(\R^n)}
,\end{align*}
and let $\XX^-$ be the analogous space of functions defined in $\R^\dmn_-$. Let
\begin{equation*}\mathfrak{D}=\dot W\!A^{0,2}_{m-1}(\R^n)\cap\dot W\!A^{1,2}_{m-1}(\R^n)\end{equation*} with
$\doublebar{\arr f}_\DD=\doublebar{\arr f}_{\dot W\!A^{0,2}_{m-1}(\R^\dmnMinusOne)}+\doublebar{\arr f}_{\dot W\!A^{1,2}_{m-1}(\R^\dmnMinusOne)},$
and let
\begin{equation*}\mathfrak{N} = (\dot W\!A^{0,2}_{m-1}(\R^n))^*\cap(\dot W\!A^{1,2}_{m-1}(\R^n))^*.\end{equation*}
We remark that if $\arr g\in L^2(\R^n)\cap \dot W^{-1,2}(\R^n)$ then $\arr g$ is a representative of an element $\arr G$ of~$\mathfrak{N}$ with $\doublebar{\arr G}_{\mathfrak{N}} \leq C\doublebar{\arr g}_{L^2(\R^n)\cap \dot W^{-1,2}(\R^n)}$.

By the main results of \cite{BarHM17pB}, if $\mat A$ is as in Theorem~\ref{thm:potentials}, and if $v\in \XX^\pm$ and $Lv=0$ in $\R^\dmn_\pm$, then $\Tr_{m-1}^\pm v \in \DD$ and $\M_{\mat A}^\pm v \in \NN$. By the bounds~\cref{eqn:S:lusin:2,eqn:D:lusin:2,eqn:D:lusin:rough:2,eqn:S:lusin:rough:2,eqn:S:N:2,eqn:S:N:rough:2,eqn:D:N:2,eqn:D:N:rough:2,eqn:S:S:horizontal}, we have that $\D^{\mat A}:\DD\to\XX^\pm$ and $\s^L:\NN\to\XX^\pm$ are bounded operators.

By \cite[Theorem~4.3]{BarHM18}, if $v\in\XX^\pm$ and $Lv=0$ in $\R^\dmn_\pm$, then the Green's formula
\begin{equation*}\1_\pm\nabla^m v=\mp \nabla^m \D^{\mat A}(\Tr_{m-1}^\pm v)+\nabla^m \s^L(\M_{\mat A}^\pm v)\end{equation*}
is valid.

Finally, the jump relations \cref{eqn:D:jump,eqn:D:cts} are valid for all $\arr f\in \DD$, and the jump relations \cref{eqn:S:jump,eqn:S:cts} are valid for all $\arr G\in \NN$. (We let $\s^L\arr G=\s^L\arr g$ for any representative $\arr g\in {L^2(\R^n)\cap \dot W^{-1,2}(\R^n)}$ of~$\arr G$; by the definition~\eqref{dfn:S} of $\s^L$, $\s^L \arr G$ is well defined.)

Thus, Conditions~6.14--6.22 in \cite{Bar17} are valid. Therefore, by \cite[Theorems~6.23 and~6.24]{Bar17}, we have that $\M_{\mat A_0}^+\D^{\mat A_0}$ is an invertible operator $\DD\to \NN$. Furthermore, the norm of $(\M_{\mat A_0}^+\D^{\mat A_0})^{-1}$ depends only on the standard parameters and the constant $C$ in the estimate $\doublebar{v}_{\XX^\pm}\leq C\doublebar{\arr G}_\NN$.

As in \cite[Section 7.2]{BarHM18}, let $\mat A_s=(1-s)\mat A_0+s\mat A$ and let $\M_s=\M_{\mat A_s}^+$. By the bounds~\cref{eqn:D:lusin:2,eqn:D:N:2,eqn:D:lusin:rough:2} and \cite[Theorems~6.1 and~6.2]{BarHM17pB}, we have that $\M_s$ is bounded $\DD\to \NN$, uniformly for $s$ in a complex neighborhood $\Omega$ of $[0,1]$. Let $r\in [0,1]$ be such that $\M_r:\DD\to \NN$ is invertible. As in \cite[Section 7.2]{BarHM18}, if $r<s\leq 1$ and $\abs{s-r}$ is small enough (depending only on $\sup_{z\in \Omega} \doublebar{\M_z}_{\DD\to \NN}$ and $\doublebar{\M_r^{-1}}_{\NN\to\DD}$), then $\M_s$ is invertible $\DD\to\NN$, and indeed satisfies
\begin{equation*}\M_s^{-1} = \sum_{j=0}^\infty \M_r^{-1}[(\M_r-\M_s)\M_r^{-1}]^j.\end{equation*}
Thus, $v=\D^{\mat A_s}(\M_s^{-1}\arr g)$ solves the Neumann problem
\begin{equation*}Lv=0\text{ in }\R^\dmn_+,\quad \M_{\mat A_s}^+v=\arr g,\quad \doublebar{v}_{\XX^+}\leq C\doublebar{\M_s^{-1}}_{\NN\to\DD} \doublebar{\arr g}_\NN.\end{equation*}
Then $v$ is also the solution to the problems \cref{eqn:neumann:rough:2,eqn:neumann:regular:2}, and so $v$ is a solution to the problem~\eqref{eqn:neumann:both:2} (with two distinct estimates rather than the single estimate $\doublebar{v}_{\XX^+}\leq C\doublebar{\arr g}_\NN$). Furthermore, the constant $C$ in the problem~\eqref{eqn:neumann:both:2} may be controlled by the constants in the problems~\eqref{eqn:neumann:regular:2} and~\eqref{eqn:neumann:rough:2}, and so is uniformly bounded for all $0\leq s\leq 1$ for which the problem~\eqref{eqn:neumann:both:2} is well posed. This affords a uniform bound on $\doublebar{\M_s^{-1}}_{\NN\to\DD}$ for all such~$s$.

Thus by continuity, $\M_s$ is invertible and the problem~\eqref{eqn:neumann:both:2} is well posed for all $0\leq s\leq 1$, in particular for $s=1$ and $\mat A_s=\mat A$. This completes the proof.
\end{proof}

In our proof of Theorem~\ref{thm:neumann:p:rough}, we will need the following analogue to Lemma~\ref{lem:N:solution} for solutions $u$ whose Neumann boundary values are zero in a cube.

\begin{lem}\label{lem:N:trace:neumann}
Let $L$ be an operator of the form~\eqref{eqn:weak} of order~$2m$ associated to bounded $t$-independent coefficients~$\mat A$ that satisfy the ellipticity condition~\eqref{eqn:elliptic}.

Let $Q\subset\R^n$ be a cube and let $\ell=\ell(Q)/4$.
Let $u\in \dot W^{m,2}(11Q\times(0,3\ell))$.
Suppose that $Lu=0$ in $10Q\times(0,3\ell)$ and that
\begin{equation}
\label{eqn:neumann:zero:Q}
\langle \nabla^m\varphi,\mat A\nabla^m u\rangle_{10Q\times(0,3\ell)}=0
\text { for all } \varphi\in C^\infty_0(10Q\times(-3\ell,3\ell)).\end{equation}

If $2<p<p_{1,L}^+$, then
\begin{align*}\biggl(\fint_{8Q} \widetilde N_{n}^\ell(\1_+\nabla^{m-1} u)^p\biggr)^{1/p}
&\leq C(1,p) \biggl(\fint_{11Q} \abs{\Tr_{m-1}^+ u}^p\biggr)^{1/p}
\\&\qquad+ C(1,p) \biggl(\fint_{10Q} \widetilde N_n^{3\ell}(\nabla^{m-1} u)^2\biggr)^{1/2}.\end{align*}
\end{lem}

We remark that if $\arr g\in \M_{\mat A}^+ u$ in the sense of formula~\eqref{eqn:Neumann:intro} for some function $\arr g$ with $\arr g=\arr 0$ in $10Q$, then formula~\eqref{eqn:neumann:zero:Q} is valid.

\begin{proof}[Proof of Lemma~\ref{lem:N:trace:neumann}]
By standard trace theorems (see, for example, the standard text \cite{Eva98}), $\Tr_{m-1} u\in L^2(11Q)$. Let $\arr f_{10Q}=(\Tr_{m-1}^+ u)_{10Q}$ be as in the proof of Theorem~\ref{thm:potentials:N:+}. That is, let $P_{11Q}$ be a polynomial of degree at most $m-2$ such that $\int_{11Q} \Trace (\partial^\xi u-\partial^\xi P_{11Q})=0$ for all $\xi$ with $\abs\xi\leq m-2$. Let $\eta_{10Q}$ be a smooth cutoff function equal to~$1$ in $10Q\times(-\ell,\ell)$ and supported in $11Q\times(-3\ell,3\ell)$. Let $\tilde u_Q=\eta_{10Q}(u-P_{11Q})+P_{11Q}$ and let $\arr f_{10Q}=\Tr_{m-1}^+ \tilde u_Q$. We extend $\arr f_{10Q}$ by zero outside of~$11Q$.

We observe that $\tilde u_Q\in \dot W^{m,2}(\R^\dmn_+)$, and so by Remark~\ref{rmk:W2:trace}, $\arr f_{10Q}\in \dot W\!A^{1/2,2}_{m-1}(\R^n)$. We further observe that $\arr f_{10Q}=\Tr_{m-1}^+ u$ in $10Q$, $\arr f_{10Q}$ is supported in $11Q$, and by the Poincar\'e inequality we have that $\doublebar{\arr f_{10Q}}_{L^2(11Q)}\leq C\doublebar{\Tr_{m-1}^+ u}_{L^2(11Q)}$.

Let $u_Q=\1_+u+\D^{\mat A}(\arr f_{10Q})$. Because $\arr f_{10Q}\in \dot W\!A^{1/2,2}_{m-1}(\R^n)$, we have that $\D^{\mat A}(\arr f_{10Q})\in \dot W^{m,2}(\R^\dmn)$ and so $u_Q\in \dot W^{m,2}(10Q\times((-3\ell,0)\cup(0,3\ell)))$.
By the jump relation~\eqref{eqn:D:jump},
\begin{align*}\Tr_{m-1}^+ u_Q
&= \Tr_{m-1}^+ u+\Tr_{m-1}^+\D^{\mat A}(\arr f_{10Q})
\\&= \Tr_{m-1}^+ u-\arr f_{10Q} + \Tr_{m-1}^- \D^{\mat A}(\arr f_{10Q})
\\&= \bigl(\Tr_{m-1}^+ u-\arr f_{10Q}\bigr) +\Tr_{m-1}^- u_Q.\end{align*}
Thus $\Tr_{m-1}^+ u_Q=\Tr_{m-1}^- u_Q$ in $10Q$,
and so $u_Q\in \dot W^{m,2}(10Q\times(-3\ell,3\ell))$.
Let $\varphi \in C^\infty_0(10Q\times(-3\ell,3\ell))$. Then
\begin{align*}
\langle \nabla^m\varphi,\mat A\nabla^m u_Q\rangle_{10Q\times(-3\ell,3\ell)}
&=
\langle \nabla^m\varphi,\mat A\nabla^m u\rangle_{10Q\times(0,3\ell)}
\\&\qquad+\langle \nabla^m\varphi,\mat A\nabla^m \D^{\mat A}(\arr f_{10Q})\rangle_{10Q\times(0,3\ell)}
\\&\qquad+\langle \nabla^m\varphi,\mat A\nabla^m \D^{\mat A}(\arr f_{10Q})\rangle_{10Q\times(-3\ell,0)}
.\end{align*}
By assumption the first inner product is zero. By the continuity relation~\eqref{eqn:D:cts}
we have that the sum
\begin{equation*}\langle \nabla^m\varphi,\mat A\nabla^m \D^{\mat A}(\arr f_{10Q})\rangle_{10Q\times(0,3\ell)}
+\langle \nabla^m\varphi,\mat A\nabla^m \D^{\mat A}(\arr f_{10Q})\rangle_{10Q\times(-3\ell,0)}
=0\end{equation*}
and so $Lu_Q=0$ in $10Q\times(-3\ell,3\ell)$.

By Lemma~\ref{lem:N:solution}, if $2<p<p_{j,L}^+$ then
\begin{equation*}
\biggl(\fint_{8Q} \widetilde N^\ell_n(\nabla^{m-j}u_Q)^{p}\,dx\biggr)^{1/{p}}
\leq C(1,p)\biggl(\fint_{10Q}\widetilde N_n^{3\ell} (\nabla^{m-j}u_Q)^2\biggr)^{1/2}
.\end{equation*}
We bound $\widetilde N_*(\nabla^{m-j}\D^{\mat A}(\arr f_{10Q}))$ using Theorem~\ref{thm:potentials:N:+}. This completes the proof.
\end{proof}

Observe that the boundary values $\Tr_{m-1}^+u$ appear in Lemma~\ref{lem:N:trace:neumann}. We now establish a bound on $\Tr_{m-1}^+ u$.
The proof of the following lemma is based on the proof of \cite[Theorem~5.1]{BarHM18}.

\begin{lem}\label{lem:Neumann:zero:reverse}
Let $L$ be an operator of the form~\eqref{eqn:weak} of order~$2m$ associated to bounded self-adjoint $t$-independent coefficients~$\mat A$ that satisfy the ellipticity condition~\eqref{eqn:elliptic:slices}.
Let $u$ satisfy the bounds
\begin{multline*}\doublebar{\mathcal{A}_2^+(t\nabla^m u)}_{L^2(\R^n)}
+\doublebar{\mathcal{A}_2^+(t\nabla^m \partial_\dmn u)}_{L^2(\R^n)}
\\+\doublebar{\widetilde N_+(\nabla^{m-1} u)}_{L^2(\R^n)}
+\doublebar{\widetilde N_+(\nabla^m u)}_{L^2(\R^n)}
<\infty.\end{multline*}
Suppose that $Lu=0$ in $\R^\dmn_+$ and that $\M_{\mat A}^+ u=0$ in $14Q$ for some cube $Q\subset\R^n$.
Then
\begin{align*}\biggl(\int_{12Q} \abs{\Tr_m^+ u}^2\biggr)^{1/2}
&\leq
	\frac{C}{\ell(Q)}\biggl(\int_{14Q} \widetilde N_+(\nabla^{m-1} u)^2\biggr)^{1/2}
.\end{align*}
\end{lem}

\begin{proof}
Let $\varepsilon>0$ and let $u_\varepsilon(x,t)=u(x,t+\varepsilon)$. Because $\mat A$ is $t$-independent, we have that $Lu_\varepsilon=0$ in $\R^\dmn_+$. We begin by bounding $\Tr_m^+ u_\varepsilon=\nabla^m u(\,\cdot\,,\varepsilon)$ in $12Q$.

Let $\varphi(x,t)=\rho(x)\,\eta(t)$, where $\rho$ and $\eta$ are smooth, real-valued, supported in $13Q$ and $(-4\ell(Q),4\ell(Q))$, and equal to~$1$ in $12Q$ and $(-2\ell(Q),2\ell(Q))$, respectively. If $0<\varepsilon<2\ell(Q)$, then
\begin{equation*}\int_{12Q} \abs{\nabla^m u(x,\varepsilon)}^2\,dx
\leq
\int_{\R^n} \abs{\nabla^m (u\varphi)(x,\varepsilon)}^2\,dx.\end{equation*}
By the bound~\eqref{eqn:elliptic:slices},
\begin{equation*}
\int_{\R^n} \abs{\nabla^m (u\varphi)(x,\varepsilon)}^2\,dx
\leq \frac{1}{\lambda} \re
\int_{\R^n}\langle \nabla^m (u\varphi)(x,\varepsilon),\mat A(x)\nabla^m(u\varphi)(x,\varepsilon)\rangle\,dx.
\end{equation*}
Because $\mat A$ is $t$-independent and self-adjoint, and because $\varphi(y,s)=0$ for all $s>4\ell(Q)$, we have that
\begin{equation*}\int_{\R^n}\langle \nabla^m (u\varphi)(\,\cdot\,,\varepsilon), \mat A\nabla^m(u\varphi)(\,\cdot\,,\varepsilon)\rangle
=-2\re\int_\varepsilon^\infty  \int_{\R^n}
\langle \nabla^m \partial_\dmn(u\varphi),\mat A\nabla^m(u\varphi)\rangle
.\end{equation*}
We wish to write the right hand side in terms of $\M_{\mat A}^+ u_\varepsilon$. Let $\abs\alpha=\abs\beta=m$. By Leibniz's rule,
\begin{align*}
\overline{\partial^\alpha \partial_\dmn(u\varphi)}\,A_{\alpha\beta}\,\partial^\beta(u\varphi)
&=
\overline{\partial^\alpha (\varphi\,\partial_\dmn(u\varphi))}\,A_{\alpha\beta}\,\partial^\beta u
\\&\qquad
-\sum_{\gamma<\alpha}
\binom{\alpha}{\gamma}
\overline{\partial^{\alpha-\gamma}\varphi\,\partial^\gamma \partial_\dmn(u\varphi)}\,A_{\alpha\beta}\,\partial^\beta u
\\&\qquad
+\sum_{\delta<\beta}
\binom{\beta}{\delta}
\overline{\partial^\alpha \partial_\dmn(u\varphi)}\,A_{\alpha\beta}\,\partial^{\beta-\delta} \varphi\,
\partial^\delta u
\end{align*}
where $\binom{\alpha}{\gamma}$ is an appropriate constant, and where $\gamma<\alpha$ if $\gamma\in (\N_0)^\dmn$, $\abs\gamma<\abs\alpha$, and $\gamma_j\leq \alpha_j$ for all $1\leq j\leq \dmn$.

Because $\mat A$ is $t$-independent, an integration by parts in~$t$ shows that
\begin{multline*}
\int_\varepsilon^\infty  \int_{\R^n}
\overline{\partial^\alpha \partial_\dmn(u\varphi)}\,A_{\alpha\beta}\,\partial^{\beta-\delta} \varphi\,
\partial^\delta u
\\\begin{aligned}
&=
-\int_\varepsilon^\infty  \int_{\R^n}
\overline{\partial^\alpha (u\varphi)}\,A_{\alpha\beta}\,\partial_\dmn\bigl(\partial^{\beta-\delta} \varphi\,
\partial^\delta u\bigr)
\\&\qquad-
\int_{\R^n}
\overline{\partial^\alpha (u\varphi)(x,\varepsilon)}\,A_{\alpha\beta}(x)\,\partial^{\beta-\delta} \varphi(x,\varepsilon)\,
\partial^\delta u(x,\varepsilon)\,dx
.\end{aligned}\end{multline*}
Thus,
\begin{multline*}
\int_{\R^n} \abs{\nabla^m (u\varphi)(x,\varepsilon)}^2\,dx
\\
\begin{aligned}
&\leq
	\frac{-2}{\lambda} \re
	\int_\varepsilon^\infty  \int_{\R^n}
	\langle \nabla^m (\varphi\partial_\dmn(u\varphi)),\mat A\nabla^m u\rangle
	\\&\qquad
	+\frac{2}{\lambda} \re
	\sum
	\int_\varepsilon^\infty  \int_{\R^n}
	\binom{\alpha}{\gamma}
	\overline{\partial^{\alpha-\gamma}\varphi\,\partial^\gamma \partial_\dmn(u\varphi)}\,A_{\alpha\beta}\,\partial^\beta u
	\\&\qquad
	+\frac{2}{\lambda} \re
	\sum
	\int_\varepsilon^\infty  \int_{\R^n}
	\binom{\beta}{\delta}
	\overline{\partial^\alpha (u\varphi)}\,A_{\alpha\beta}\,\partial_\dmn(\partial^{\beta-\delta} \varphi\,
	\partial^\delta u)
	\\&\qquad
	+\frac{2}{\lambda} \re
	\sum
	\int_{\R^n}
	\binom{\beta}{\delta}
	\overline{\partial^\alpha (u\varphi)(x,\varepsilon)}\,A_{\alpha\beta}(x)\, \partial^{\beta-\delta} \varphi(x,\varepsilon)\,
	\partial^\delta u(x,\varepsilon)\,dx
\\&=I+II+III+IV
\end{aligned}\end{multline*}
where the sums are over appropriate ranges of multiindices.

We normalize $u$ so that $\int_{13Q}\int_0^{4\ell(Q)}\nabla^j u=0$ whenever $0\leq j\leq m-2$. By the Poincar\'e inequality, if $0\leq j\leq m-2$ then
\begin{equation}
\label{eqn:neumann:proof:poincare}
\int_{13Q}\int_0^{4\ell(Q)}\abs{\nabla^j u}^2\leq
C\ell(Q)^{2m-2-2j}\int_{13Q}\int_0^{4\ell(Q)}\abs{\nabla^{m-1} u}^2.\end{equation}
Furthermore, by boundedness of the trace map (see, for example, \cite{Eva98}), we have that for such~$j$, if $0<\varepsilon<2\ell(Q)$ then
\begin{equation}\label{eqn:neumann:proof:trace}
\int_{13Q}\abs{\nabla^j u(x,\varepsilon)}^2\,dx\leq C\ell(Q)^{2m-3-2j}\int_{13Q}\int_0^{4\ell(Q)}\abs{\nabla^{m-1} u}^2.\end{equation}
Applying these bounds, we see that
\begin{equation*}II+III\leq
\frac{C}{\ell(Q)}
\int_{13Q}\int_0^{4\ell(Q)} \abs{\nabla^m u}^2
+
\frac{C}{\ell(Q)^3}
\int_{13Q}\int_0^{4\ell(Q)} \abs{\nabla^{m-1} u}^2
\end{equation*}
and
\begin{align*}IV
&\leq
	\frac{C}{\ell(Q)}
	\doublebar{\nabla^m(u\varphi)(\,\cdot\,,\varepsilon)}_{L^2(\R^n)} \doublebar{\nabla^{m-1} u(\,\cdot\,,\varepsilon)}_{L^2(13Q)}
	\\&\qquad
	+\frac{C}{\ell(Q)^{3/2}} \doublebar{\nabla^m(u\varphi)(\,\cdot\,,\varepsilon)}_{L^2(\R^n)} \biggl(\int_{13Q}\int_0^{4\ell(Q)}\abs{\nabla^{m-1} u}^2\biggr)^{1/2}
.\end{align*}
By the boundary Caccioppoli inequality (see \cite[Lemma~16]{Bar16}),
\begin{equation*}
\int_{13Q}\int_0^{4\ell(Q)} \abs{\nabla^m u}^2
\leq
	\frac{C}{\ell(Q)^2}\int_{14Q}\int_0^{5\ell(Q)} \abs{\nabla^{m-1} u}^2
\end{equation*}
and so by Young's inequality, Lemma~\ref{lem:slices} and the definition of~$\widetilde N$,
\begin{align*}II+III+IV
&\leq
	\frac{C}{\ell(Q)^2}
	\int_{14Q} \widetilde N_+(\nabla^{m-1} u)^2
	+\frac{1}{2}
	\doublebar{\nabla^m(u\varphi)(\,\cdot\,,\varepsilon)}_{L^2(\R^n)}^2
.\end{align*}
We are left with the term~$I$. Using the definition~\eqref{eqn:Neumann:intro} of Neumann boundary values and the fact that $\partial_t\varphi(x,t)=0$ for $t$ small enough, we see that
\begin{equation*}I=	\frac{-2}{\lambda} \re
	\left\langle \Tr_{m-1}(\varphi^2\partial_\dmn u_\varepsilon),
	\M_{\mat A}^+ u_\varepsilon\right\rangle_{\R^n}
.\end{equation*}
Because $\M_{\mat A}^+u=0$ in $14Q$ and  $\Tr_{m-1}(\varphi^2\partial_\dmn u_\varepsilon)=0$ outside of~$13Q$, we have that
\begin{equation*}I=\frac{-2}{\lambda} \re
	\langle \Tr_{m-1}(\varphi^2\partial_\dmn u_\varepsilon),
	\M_{\mat A}^+ u_\varepsilon-\M_{\mat A}^+ u\rangle_{\R^n}.\end{equation*}
Furthermore, by Lemma~\ref{lem:slices} and the assumed nontangential bounds on $\nabla^m u$ and $\nabla^{m-1} u$, we have that $\Tr_m^+ u_\varepsilon$ and $\Tr_{m-1}^+ u_\varepsilon$ are in $L^2(\R^n)$, uniformly in~$\varepsilon$, and by the bound~\eqref{eqn:neumann:proof:trace}, we have that $\Tr_j^+ u_\varepsilon\in L^2(13Q)$ for all $0\leq j\leq m-2$, again uniformly in~$\varepsilon$. Thus, $\Tr_{m-1}(\varphi^2\partial_\dmn u_\varepsilon)\in {\dot W\!A^{0,2}_{m-1}(\R^n)}$ with a norm that may be bounded independently of~$\varepsilon$.
By \cite[Lemma~4.2]{BarHM18}, we have that $\M_{\mat A}^+ (u-u_\varepsilon)\to 0$ in
$(\dot W\!A^{0,2}_{m-1}(\R^n))^*$ as $\varepsilon\to 0^+$. Thus, $\lim_{\varepsilon \to 0^+} I=0$, and so
\begin{equation*}
\lim_{\varepsilon\to 0^+}
\int_{\R^n} \abs{\nabla^m (u\varphi)(x,\varepsilon)}^2\,dx
\leq
	\frac{C}{\ell(Q)^2}
	\int_{14Q} \widetilde N_+(\nabla^{m-1} u)^2
.\end{equation*}
By \cite[Theorem~5.3]{BarHM17pB},
\begin{equation*}\int_{12Q} \abs{\Tr_m u}^2=\lim_{\varepsilon\to 0^+} \int_{12Q} \abs{\nabla^m u(x,\varepsilon)}^2\,dx
\leq\lim_{\varepsilon\to 0^+}
\int_{\R^n} \abs{\nabla^m (u\varphi)(x,\varepsilon)}^2\,dx
\end{equation*}
and so
\begin{equation*}\int_{12Q} \abs{\Tr_m u}^2\leq 	\frac{C}{\ell(Q)^2}
	\int_{14Q} \widetilde N_+(\nabla^{m-1} u)^2
\end{equation*}
as desired.
\end{proof}

We are now in a position to prove the main result of this section. We will establish that the solutions to the Neumann  problem~\eqref{eqn:neumann:both:2} satisfy a nontangential estimate for all $\arr g$ in a certain dense subspace of $\dot W^{-1,p}$. This subspace is defined as follows.

Suppose that $\arr h$ is an array of \emph{vector}-valued functions, so that $\vec h_\gamma:\R^n\to \C^n$ for each multiindex $\gamma$ with $\abs\gamma=m-1$. We define $\Div \arr h$ as the array given by
\begin{equation}\label{eqn:div:array}
(\Div \arr h)_\gamma = \Div \vec h_\gamma=\sum_{j=1}^n \partial_{x_j} (h_\gamma)_j.\end{equation}
\begin{rmk}
If $1<p<\infty$ and $\arr h\in L^p(\R^n)$, and if the divergence is taken in the distributional sense, then $\Div \arr h\in \dot W^{-1,p}(\R^n)$ with $\doublebar{\Div \arr h}_{\dot W^{-1,p}(\R^n)}\leq \doublebar{\arr h}_{L^p(\R^n)}$.
Conversely, by the Hahn-Banach theorem, if $\arr g\in \dot W^{-1,p}(\R^n)$ for some $1<p<\infty$, then there is an $\arr h$ with $\doublebar{\arr h}_{L^p(\R^n)}\approx \doublebar{\arr g}_{\dot W^{-1,p}(\R^n)}$ and with $\arr g=\Div\arr h$. Thus, because $C^\infty_0(\R^n)$ is dense in $L^p(\R^n)$, we have that $\{\Div\arr h:\arr h\in C^\infty_0(\R^n)\}$ is dense in $\dot W^{-1,p}(\R^n)$.
\end{rmk}

\begin{thm} \label{thm:neumann:N}
Let $L$ be an operator of the form~\eqref{eqn:weak} of order~$2m$ associated to bounded self-adjoint $t$-independent coefficients~$\mat A$ that satisfy the ellipticity condition~\eqref{eqn:elliptic:slices}.

There is some $\varepsilon>0$ such that if
\begin{equation*}
\max\biggl(0,\frac{1}{2}-\frac{1}{n}-\varepsilon\biggr)<\frac{1}{p}\leq \frac{1}{2}
\end{equation*}
then there is a number $C_p$ such that if $\arr h\in C^\infty_0(\R^n)$ is an array of vector-valued functions and $\arr g=\Div \arr h$, then the solution $v$ to the problem~\eqref{eqn:neumann:both:2} also satisfies
\begin{equation}\label{eqn:neumann:p:rough:N}
\left\{\begin{gathered}
Lv=0 \text{ in }\R^\dmn_+
,\qquad
\M_{\mat A} v \owns \arr g,
\\
\doublebar{\widetilde N_+(\nabla^{m-1}v)}_{L^p(\R^n)}
\leq C_p \doublebar{\arr h}_{L^p(\R^n)}
.\end{gathered}\right.\end{equation}
\end{thm}

\begin{proof}
Choose some such $\arr h$ and let $v$ be the solution to the problem~\eqref{eqn:neumann:both:2}.

We will use Lemma~\ref{lem:N:+} with $\arr u=\1_+\nabla^{m-1}v$.
We construct $\arr u_Q$ as follows.
For each cube $Q\subset\R^n$,
let $\eta_{14Q}$ be as in the proof of Theorem~\ref{thm:potentials:N:+}. Then $\eta_{14Q}$ is smooth and compactly supported, $\eta_{14Q}=1$ in $14Q$ and $\eta_{14Q}=0$ outside of $15.4Q\subset 16Q$. Let $\arr h_Q=\eta_{14Q}\arr h$ and let $\arr g_Q=\Div \arr h_Q$.
Let $v_Q$ be the solution to the problem~\eqref{eqn:neumann:both:2} with data $\arr g_Q$ and let $\arr u_Q=\1_+\nabla^{m-1}v_Q$.

Choose some such cube $Q\in\R^n$ and let $u=v-v_Q$. Let $R\subseteq Q$ be a cube contained in~$Q$.
Observe that $Lu=0$ in $\R^\dmn_+$, that $\widetilde N_+(\nabla^mu)\in L^2(\R^n)$ and so $u\in \dot W^{m,2}(11R\times(0,\ell(R))$, and by formula~\eqref{eqn:Neumann:intro}, $u$ satisfies the condition~\eqref{eqn:neumann:zero:Q}. Thus, $u$ satisfies the conditions of Lemma~\ref{lem:N:trace:neumann}.
Let $1/p_3=1/2-1/n$ if $\dmn \geq 4$, and let $2<p_3<\infty$ if $\dmn=2$ or $\dmn=3$. By Proposition~\ref{prp:Meyers:t-independent}, $p_3<p_{1,L}^+$ and $c(1,L,p_3,2)$ depends only on $p_3$ and the standard constants. Thus,
\begin{align*}
\biggl(\fint_{8R} \widetilde N_{n}^\ell(\1_+\nabla^{m-1} u)^p\biggr)^{1/p_3}
&\leq C \biggl(\fint_{11R} \abs{\Tr_{m-1}^+ u}^{p_3}\biggr)^{1/p_3}
\\&\qquad + C \biggl(\fint_{10R} \widetilde N_n^{3\ell}(\nabla^{m-1} u)^2\biggr)^{1/2}
.\end{align*}
By the Gagliardo-Nirenberg-Sobolev inequality,
\begin{align*}\biggl(\int_{11R} \abs{\Tr_{m-1} u}^{p_3} \biggr)^{1/{p_3}}
&\leq
	C\abs{R}^{1/p_3-1/2+1/n}\biggl(\int_{11R} \abs{\Tr_m u}^2\biggr)^{1/2}
	\\&\qquad
	+C \abs{R}^{1/{p_3}}\biggl(\fint_{11R} \abs{\Tr_{m-1} u}^2 \biggr)^{1/2}
\end{align*}
By Lemma~\ref{lem:Neumann:zero:reverse},
\begin{align*}\biggl(\int_{12R} \abs{\Tr_m^+ u}^2\biggr)^{1/2}
&\leq
	\frac{C}{\ell(R)}\biggl(\int_{14R} \widetilde N_+(\nabla^{m-1} u)^2\biggr)^{1/2}
.\end{align*}
Combining these estimates and applying Lemma~\ref{lem:trace:N}, we have that
\begin{equation*}\biggl(\fint_{8R} \widetilde N_{n}^\ell(\1_+\nabla^{m-1} u)^{p_3}\biggr)^{1/p_3}
\leq
	C \biggl(\fint_{14R} \widetilde N_+(\nabla^{m-1} u)^2\biggr)^{1/2}
.\end{equation*}
By Lemma~\ref{lem:N:far}, $\widetilde N_{f}^\ell(\1_+\nabla^{m-1} u)(x)\leq C\bigl(\fint_{14R} \widetilde N_+(\nabla^{m-1} u)^2\bigr)^{1/2}$ for all $x\in 8R$, and so
\begin{equation*}\biggl(\fint_{8R} \widetilde N_+(\nabla^{m-1} u)^{p_3}\biggr)^{1/p_3}
\leq
	C \biggl(\fint_{14R} \widetilde N_+(\nabla^{m-1} u)^2\biggr)^{1/2}
.\end{equation*}
Because ${p_3}>2$, this is a reverse H\"older inequality. It is well known that reverse H\"older inequalities are self-improving. For example, by \cite[Chapter~V, Theorem~1.2]{Gia83}, there is some $p_2>p_3$ and $C_2<\infty$ depending only on $n$, $C$ and $p_3$ such that this bound is true with $p_2$ replaced by~$p_3$ and $C$ replaced by~$C_2$.

Recall that $\arr u=\1_+\nabla^{m-1} v$ and that $\widetilde N_+(\nabla^{m-1} v)\in L^2(\R^n)$. Recall also that $\arr u_Q=\1_+\nabla^{m-1} v_Q$.
Thus, for any cube $Q\subset\R^n$,
\begin{align*}
\doublebar{\widetilde N_*\arr u_Q}_{L^2(\R^n)} &\leq C\doublebar{\arr h}_{L^2(16Q)},
\\
\biggl(\fint_{8Q} \widetilde N_+(\arr u-\arr u_Q)^{p_2}\biggr)^{1/p_2}
&\leq
	C \biggl(\fint_{14Q} \widetilde N_+(\arr u-\arr u_Q)^2\biggr)^{1/2}
.\end{align*}
By Lemma~\ref{lem:N:+}, if $2<p<p_2$ then
$\doublebar{\widetilde N_*\arr u}_{L^p(\R^n)}\leq C_p\doublebar{\arr h}_{L^p(\R^n)}$. This completes the proof.
\end{proof}

\section{Area integral estimates}
\label{sec:lusin:+}

In this section we will establish the area integral bounds \cref{eqn:S:lusin:rough:intro,eqn:D:lusin:rough:intro,eqn:S:lusin:intro,eqn:D:lusin:intro} in Theorem~\ref{thm:potentials}, and show that the solution~$v$ to the Neumann problem in Theorem~\ref{thm:neumann:N} satisfies an area integral estimate.

In the introduction, we defined the Lusin area integral $\mathcal{A}_2^+$, $\mathcal{A}_2^*$. See formulas \cref{dfn:lusin:+,dfn:lusin:*}. We will also need the corresponding operator in the lower half space; thus, we define
\begin{align*}
\mathcal{A}_2^- H(x)^2 &= \int_{-\infty}^0 \int_{\abs{x-y}<\abs{t}} \abs{H(y,t)}^2 \frac{dy\,dt}{\abs{t}^\dmn}
= \int_0^{\infty} \int_{\abs{x-y}<{t}} \abs{H(y,-t)}^2 \frac{dy\,dt}{{t}^\dmn}
.\end{align*}

We begin with the analogue to Lemma~\ref{lem:N:far} for the area integral. Let $\ell\in\R$, $\ell\neq 0$. We define
\begin{align*}
\mathcal{A}_f^{\ell} H(x) &= \biggl(\int_{\abs\ell}^\infty \int_{\abs{x-y}<t} \abs{H(y,\sgn(\ell) t)}^2 \frac{dy\,dt}{t^\dmn}\biggr)^{1/2}
,\\
\mathcal{A}_n^{\ell} H(x) &= \biggl(\int_0^{\abs\ell} \int_{\abs{x-y}<t} \abs{H(y,\sgn(\ell) t)}^2 \frac{dy\,dt}{t^\dmn}\biggr)^{1/2}
,\end{align*}
so that
\begin{align*}
\mathcal{A}_2^+ H(x)^2 &= \mathcal{A}_f^{\ell}H(x)^2 + \mathcal{A}_n^{\ell}H(x)^2\quad\text{if }\ell>0
,\\
\mathcal{A}_2^- H(x)^2 &= \mathcal{A}_f^{\ell}H(x)^2 + \mathcal{A}_n^{\ell}H(x)^2\quad\text{if }\ell<0
.\end{align*}
\begin{lem}\label{lem:lusin:far} Let $\ell\neq0$.
We have that
\begin{equation*}\mathcal{A}_f^{\ell} H(x)^2 \leq 2^n \fint_{\abs{x-z}<\abs\ell}\mathcal{A}_f^{\ell} H(z)^2\,dz\end{equation*}
whenever the right hand side is finite.
\end{lem}
\begin{proof}
Let $\Delta(x,\abs\ell)=\{z\in\R^n:\abs{x-z}<\abs\ell\}$.
We compute that
\begin{align*}
\fint_{\Delta(x,\abs\ell)}\mathcal{A}_f^{\ell} H(z)^2\,dz
&=
	\fint_{\Delta(x,\abs\ell)}
	\int_{\abs\ell}^\infty \int_{\abs{z-y}<t} \abs{H(y,\sgn(\ell) t)}^2 \frac{dy\,dt}{t^\dmn}
	\,dz
\\&=
	\int_{\abs\ell}^\infty
	\int_{\abs{x-y}<t+\abs\ell}
	\frac{\abs{\Delta(x,\abs\ell)\cap\Delta(y,t)}} {\abs{\Delta(x,\abs\ell)}}
	\abs{H(y,\sgn(\ell) t)}^2 \frac{dy\,dt}{t^\dmn}
.\end{align*}
If $\abs{x-y}<t$ and $t\geq \abs\ell$, then $\Delta(x,\abs\ell)\cap\Delta(y,t)$ contains a disk of radius $\abs\ell/2$, and so
\begin{equation*}\frac{\abs{\Delta(x,\abs\ell)\cap\Delta(y,t)}} {\abs{\Delta(x,\abs\ell)}}
>\frac{1}{2^n}.\end{equation*}
Therefore,
\begin{align*}
\fint_{\Delta(x,\abs\ell)}\mathcal{A}_f^{\ell} H(z)^2\,dz
&\geq
	\frac{1}{2^n}\int_{\abs\ell}^\infty
	\int_{\abs{x-y}<t}\abs{H(y,\sgn(\ell) t)}^2 \frac{dy\,dt}{t^\dmn}
=\frac{1}{2^n} \mathcal{A}_f^{\ell} H(x)^2
\end{align*}
as desired.
\end{proof}

The main tool in our argument will be the following lemma. This result will perform the same role as Lemmas~\ref{lem:N:+} and~\ref{lem:N:solution}.

\begin{lem}\label{lem:lusin:+} Let $L$ be an operator of the form~\eqref{eqn:weak} of order~$2m$ associated to bounded $t$-independent coefficients~$\mat A$ that satisfy the ellipticity condition~\eqref{eqn:elliptic}.

Let $\mathcal{A}_2^\pm$ denote either $\mathcal{A}_2^+$ or $\mathcal{A}_2^-$.

Let $2<p<\infty$. Let $u\in \dot W^{m,2}_{loc}(\R^\dmn_+\cup \R^\dmn_-)$ be such that
\begin{align}\label{eqn:lusin:+:2}
\mathcal{A}_2^\pm(t\nabla^m u)&\in L^2(\R^n),\\
\label{eqn:lusin:N:2}
\qquad
\doublebar{\widetilde N_*(\nabla^{m-1}u)}_{L^p(\R^n)}
&\in {L^p(\R^n)}.\end{align}
Suppose that there is a constant $C_0$, a function $\Phi\in L^p(\R^n)$ and a family of functions $u_Q$ indexed by cubes $Q$ in $\R^n$ that satisfy the conditions
\begin{gather}
\label{eqn:lusin:local:lusin:2}
\doublebar{\mathcal{A}_2^\pm(t\nabla^m u_Q)}_{L^2(8Q)}
\leq C_0\doublebar{\Phi_1}_{L^2(16Q)}
,\\
\label{eqn:lusin:local:N:2}
\doublebar{\widetilde N_*(\nabla^{m-1} u_Q)}_{L^2(10Q)}
\leq C_0\doublebar{\Phi_1}_{L^2(16Q)}
,\\
u-u_Q\in \dot W^{m,2}(10Q\times(-\ell(Q),\ell(Q))),
\\\label{eqn:lusin:solution}
L(u-u_Q)=0\text{ in }10Q\times(-\ell(Q),\ell(Q))
.\end{gather}

Then there is a constant $C$ depending only on $C_0$, $p$, and the standard parameters such that
\begin{equation*}\doublebar{\mathcal{A}_2^\pm(t\nabla^m u)}_{L^p(\R^n)} \leq C\doublebar{\Phi_1}_{L^p(\R^n)}+C\doublebar{\widetilde N_*(\nabla^{m-1}u)}_{L^p(\R^n)}
.\end{equation*}
\end{lem}

We emphasize that, while we do require two-sided nontangential estimates (that is, bounds on $\widetilde N_*(\nabla^{m-1}u)$ rather than $\widetilde N_+(\nabla^{m-1}u)$ alone), we will only need one-sided $L^2$ square function estimates.
We will need Lemma~\ref{lem:lusin:+} in this generality in the forthcoming paper \cite{Bar19pB}.

\begin{proof}[Proof of Lemma~\ref{lem:lusin:+}]
The result follows from Lemma~\ref{lem:lambda} with $F={\mathcal{A}_2^\pm(t\nabla^m u)}$ and with  $\Phi=\Phi_2=\Phi_1+\widetilde N_*(\nabla^{m-1}u)$. We need only show that the conditions of Lemma~\ref{lem:lambda} are valid.

Let $Q_0$ be a large cube.
Let $0<\gamma<1$ and $\lambda>0$, and let $Q\subset Q_0$ satisfy
\begin{align}
\label{eqn:Q:Psi:higher:lusin}
\fint_{16Q} \abs{\Phi_2}^{2} &\leq \gamma\lambda, \\
\label{eqn:Q:big:lusin}
8^n\abs{Q}\lambda&<\int_{8Q} \mathcal{A}_2^\pm(t\nabla^{m}u)^2,
\\
\label{eqn:Q:small:lusin}
\int_{15Q} {\mathcal{A}_2^\pm(t\nabla^{m}u)}^2 &\leq 16^n\abs{Q}\lambda.\end{align}
Let $\ell=\pm\ell(Q)/4$, let $u_n=u_Q$, and let $u_f=u-u_Q$.
The bound
\begin{multline}
\label{eqn:E:split:lusin}
\abs{\{x\in Q:\mathcal{M}_{8,Q} ((\mathcal{A}_2^\pm(t\nabla^{m}u))^2)(x)>A\lambda\}}
\\\begin{aligned}
&\leq
	\abs{\{x\in Q:\mathcal{M}_{8,Q}((\mathcal{A}_f^\ell(t\nabla^{m}u))^2)(x)>A\lambda/\sqrt2\}}
	\\&\qquad+
	\abs{\{x\in Q: \mathcal{M}_{8,Q}((\mathcal{A}_n^\ell(t\nabla^{m}u))^2)(x) > A\lambda/\sqrt2\}}
\end{aligned}\end{multline}
may be established in the same way as the bound \eqref{eqn:E:split}. If $A$ is large enough (depending only on dimension), then the formula
\begin{align}
\label{eqn:E:bound:3:lusin}
\{x\in Q: \mathcal{M}_{8,Q}(\mathcal{A}_f^\ell(t\nabla^{m}u)^2)(x) >A\lambda/\sqrt2\}&=\emptyset
\end{align}
follows from the bound~\eqref{eqn:Q:small:lusin} and Lemma~\ref{lem:lusin:far}, analogously to formula~\eqref{eqn:E:bound:3}.

By the bounds~\eqref{eqn:lusin:local:lusin:2} and~\eqref{eqn:Q:Psi:higher:lusin},
\begin{equation*}
\doublebar{\mathcal{A}_2^\pm(t\nabla^m u_n)}_{L^2(8Q)}^2
\leq C_0 16^n\gamma\lambda\abs{Q}
.\end{equation*}
By definition of~$\mathcal{A}_n^\ell$, we have that
\begin{equation*}\int_{8Q} \mathcal{A}_n^\ell (t\nabla^{m}u_f)^2
\leq C\int_{9Q} \int_0^{\abs\ell} \abs{\nabla^m u_f(x,\sgn(\ell) t)}^2\,t\,dt\,dx
\leq C\abs\ell\int_{9Q} \int_{-\abs\ell}^{\abs\ell} \abs{\nabla^m u_f}^2
.\end{equation*}
By formula~\eqref{eqn:lusin:solution}, we may use  the Caccioppoli inequality to see that
\begin{equation*}
\int_{8Q} \mathcal{A}_n^\ell (t\nabla^{m}u_f)^2
\leq C\int_{10Q} \fint_{-2\ell}^{2\ell} \abs{\nabla^{m-1} u_f}^2.
\end{equation*}
By definition of~$\widetilde N_*$ and~$u_f$, we have that
\begin{align*}
\int_{8Q} \mathcal{A}_n^\ell (t\nabla^{m}u_f)^2
&\leq C\int_{10Q} \widetilde N_*(\nabla^{m-1}u_f)^2
\\&\leq C\int_{10Q} \widetilde N_*(\nabla^{m-1}u)^2
+C\int_{10Q} \widetilde N_*(\nabla^{m-1}u_Q)^2
.\end{align*}
By definition of~$\Phi_2$, the bound~\eqref{eqn:lusin:local:N:2}, and the bound~\eqref{eqn:Q:Psi:higher:lusin},
\begin{align*}
\int_{8Q} \mathcal{A}_n^\ell (t\nabla^{m}u_f)^2
&\leq
 C\int_{16Q} (\Phi_2)^2 \leq C\gamma\lambda\abs{Q}
.\end{align*}
Thus, by the weak $L^1$ boundedness of the maximal operator,
\begin{equation}\label{eqn:E:bound:1:lusin}
\abs{
\{x\in Q: \mathcal{M}_{8,Q}(\mathcal{A}_n^\ell (t\nabla^m u_f)^2)(x)>A\lambda/2\}}
\leq
	\frac{C\abs{Q}\gamma}{A}
.\end{equation}
By the bounds \cref{eqn:E:split:lusin,eqn:E:bound:3:lusin,eqn:E:bound:1:lusin}, we may apply Lemma~\ref{lem:lambda} and complete the proof.
\end{proof}

We may now establish the area integral estimates mentioned in the introduction.

\begin{thm}\label{thm:lusin:+} Let $L$ be an operator of the form~\eqref{eqn:weak} of order~$2m$ associated to bounded $t$-independent coefficients~$\mat A$ that satisfy the ellipticity condition~\eqref{eqn:elliptic}.

Let $2<p<p_{1,L}^+$. Let $\arr h$, $\arr f$, $\arr g$, and~$\arr \varphi$ be as in Theorem~\ref{thm:potentials:N:+}.
Then we have the bounds
\begin{align}
\label{eqn:S:lusin:rough:+}
\doublebar{\mathcal{A}_2^*(t\nabla^{m}\s^L_\nabla \arr h)}_{L^p(\R^n)}
&\leq C(1,p) \doublebar{\arr h}_{L^p(\R^n)}
,\\
\label{eqn:D:lusin:rough:+}
\doublebar{\mathcal{A}_2^*(t\nabla^{m}\D^{\mat A} \arr f)}_{L^p(\R^n)}
&\leq C(1,p) \doublebar{\arr f}_{\dot W\!A_{m-1}^{0,p}(\R^n)}
,\\
\label{eqn:S:lusin:+}
\doublebar{\mathcal{A}_2^*(t\nabla^{m}\partial_t\s^L \arr g)}_{L^p(\R^n)}
&\leq C(1,p) \doublebar{\arr g}_{L^p(\R^n)}
,\\
\label{eqn:D:lusin:+}
\doublebar{\mathcal{A}_2^*(t\nabla^{m}\partial_t\D^{\mat A} \arr \varphi)}_{L^p(\R^n)}
&\leq C(1,p) \doublebar{\arr \varphi}_{\dot W\!A_{m-1}^{1,p}(\R^n)}
.\end{align}

Suppose furthermore that $p$ and $\mat A$ are such that if $\arr h\in C^\infty_0(\R^n)$ is an array of vector-valued functions, and if $\arr g=\Div \arr h$ in the sense of formula~\eqref{eqn:div:array}, then there is a function $v$ that solves the Neumann problems \cref{eqn:neumann:both:2,,eqn:neumann:p:rough:N}.
Then the solution $v$ also satisfies
\begin{equation}
\label{eqn:neumann:lusin:+}
\doublebar{\mathcal{A}_2^+(t\nabla^{m}v)}_{L^p(\R^n)}
\leq C_p \doublebar{\arr h}_{L^p(\R^n)}
\end{equation}
where $C_p$ depends only on the standard parameters, $p$, $c(1,L,p,2)$, and the constants in the problems \cref{eqn:neumann:both:2,,eqn:neumann:p:rough:N}.
\end{thm}

\begin{proof}
To establish the bounds~\cref{eqn:S:lusin:rough:+,eqn:D:lusin:rough:+,eqn:S:lusin:+,eqn:D:lusin:+}, we will apply Lemma~\ref{lem:lusin:+} with $u$, $u_Q$, and $\Phi_1$ as in the proof of Theorem~\ref{thm:potentials:N:+} with $j=1$.
Then the bounds~\eqref{eqn:lusin:+:2} and~\eqref{eqn:lusin:local:lusin:2} follow from the bounds~\cref{eqn:S:lusin:2,eqn:S:lusin:rough:2,eqn:D:lusin:2,eqn:D:lusin:rough:2}. The bound~\eqref{eqn:lusin:N:2}
follows from Theorem~\ref{thm:potentials:N:+}. As observed in the proof of Theorem~\ref{thm:potentials:N:+}, formula \eqref{eqn:lusin:solution} and the bound~\eqref{eqn:lusin:local:N:2} are valid. Thus, $\doublebar{\mathcal{A}_2^\pm(t\nabla^{m}u)}_{L^p(\R^n)}
\leq C(1,p)\doublebar{\Phi_1}_{L^p(\R^n)}$, as desired.

We now turn to the bound~\eqref{eqn:neumann:lusin:+}.
As in the proof of Theorem~\ref{thm:neumann:N}, let $\arr h_Q=\eta_{14Q}\arr h$,
where $\eta_{14Q}$ is smooth, supported in $16Q$ and identically equal to $1$ in~$14Q$. Let $\arr g_Q=\Div \arr h_Q$.
Let $v_Q$ be the solution to the problems~\cref{eqn:neumann:both:2,,eqn:neumann:p:rough:N} with boundary data $\arr g_Q$.

By \cite[Theorems 5.1 and~5.3]{BarHM17pB}, we have that $\arr f=\Tr_{m-1} v$ and $\arr f_Q=\Tr_{m-1} v_Q$ exist and lie in $L^2(\R^n)$. In particular,
\begin{equation*}\doublebar{\Tr_{m-1} v_Q}_{L^2(\R^n)}
\leq C\doublebar{\mathcal{A}_2^+(t\nabla^m v_Q)}_{L^2(\R^n)}
\leq C^2\doublebar{\arr h_Q}_{L^2(\R^n)}
\leq C^2\doublebar{\arr h}_{L^2(16Q)}
.\end{equation*}
By Lemma~\ref{lem:trace:N} and the estimate in the problem~\eqref{eqn:neumann:p:rough:N}, we have that $\Tr_{m-1} v$, $\Tr_{m-1} v_Q\in \dot W^{0,p}_{m-1}(\R^n)$ with
\begin{equation*}\doublebar{\Tr_{m-1}v}_{\dot W^{0,p}_{m-1}(\R^n)}\leq C \doublebar{\arr h}_{L^p(\R^n)},\qquad\doublebar{\Tr_{m-1}v_Q}_{\dot W^{0,p}_{m-1}(\R^n)}\leq C \doublebar{\arr h}_{L^p(16Q)}.\end{equation*}

Let $u=\1_+v+\D^{\mat A}\arr f$ and let $u_Q=\1_+v_Q+\D^{\mat A}\arr f_Q$.

Then $\mathcal{A}_2^*(t\nabla^m u)\in L^2(\R^n)$ by the bound~\eqref{eqn:D:lusin:rough:2} and because $v$ is a solution to the problem~\eqref{eqn:neumann:both:2}. $\doublebar{\widetilde N_*(\nabla^{m-1} u)}_{L^p(\R^n)}\leq C(1,p)\doublebar{\arr h}_{L^p(\R^n)}$ by the bound~\eqref{eqn:D:N:rough:+} and because $v$ is a solution to the problem~\eqref{eqn:neumann:p:rough:N}.
By the bounds~\cref{eqn:D:lusin:rough:2,eqn:D:N:rough:2} on the double layer potential and because $v_Q$ is a solution to the problem~\eqref{eqn:neumann:both:2}, we have that $\doublebar{\mathcal{A}_2^*(t\nabla^m u_Q)}_{L^2(\R^n)} + \doublebar{\widetilde N_*(\nabla^{m-1} u_Q)}_{L^2(\R^n)} \leq C\doublebar{\arr h}_{L^2(16Q)}$.

Thus, the bounds \cref{eqn:lusin:+:2,eqn:lusin:N:2,eqn:lusin:local:lusin:2,eqn:lusin:local:N:2} are valid with $\Phi_1=\abs{\arr h}$.

Finally, $u-u_Q$ is in $\dot W^{m,2}(10Q\times((-\ell(Q),0)\cup(0,\ell(Q))))$  because $\widetilde N_*(\nabla^m u)$, $\widetilde N_*(\nabla^m u_Q)\in L^2(\R^n)$. As in the proof of Lemma~\ref{lem:N:trace:neumann}, by the jump relation~\eqref{eqn:D:jump}, we have that
$u-u_Q$ is in $\dot W^{m,2}(10Q\times(-\ell(Q),\ell(Q)))$, and by the continuity relation~\eqref{eqn:D:cts}, the definition~\eqref{eqn:Neumann:intro} of Neumann boundary values and the weak definition~\eqref{eqn:weak} of~$L$, we have that $L(u-u_Q)=0$ in $10Q\times(-\ell(Q),\ell(Q))$.

Thus, by Lemma~\ref{lem:lusin:+}, we have that the bound~\eqref{eqn:neumann:lusin:+} is valid.
\end{proof}

A straightforward density argument lets us pass from Theorems~\ref{thm:potentials:N:+} and~\ref{thm:lusin:+} to Theorem~\ref{thm:potentials}, and from Theorem~\ref{thm:neumann:N} and Theorem~\ref{thm:lusin:+} to Theorem~\ref{thm:neumann:p:rough}.

%\bibliographystyle{unsrt}
%\bibliography{bibli}

\end{document}